\documentclass[10pt]{amsart}
\usepackage{graphicx}
 \usepackage{caption}
\usepackage{psfrag}
\usepackage{float}
\usepackage{subfigure}
\usepackage{psfrag}
\usepackage{epsfig}
\usepackage{epstopdf}
\usepackage{amsmath}
\usepackage{amscd}

\usepackage{amsthm,amsfonts,amsxtra,amssymb,amscd}
\usepackage[all]{xy}
\xyoption{2cell}
\usepackage{enumerate}

\usepackage{pdfsync}
\def\pd#1#2{\dfrac{\partial#1}{\partial#2}}

 \usepackage[unicode]{hyperref}

\theoremstyle{plain}
\newtheorem*{lemma*}{Lemma}
\newtheorem{lemma}[subsection]{Lemma}
\newtheorem*{theorem*}{Theorem}
\newtheorem{theorem}{Theorem}
\newtheorem*{proposition*}{Proposition}
\newtheorem{proposition}[subsection]{Proposition}
\newtheorem*{corollary*}{Corollary}
\newtheorem{corol}{Corollary}
\newtheorem{corollary}[subsection]{Corollary}
\theoremstyle{definition}
\newtheorem*{definition*}{Definition}
\newtheorem{constraint}{Constraint}\newtheorem{ass}{Assumption}
\newtheorem{definition}[subsection]{Definition}
\newtheorem*{example*}{Example}
\newtheorem{example}[subsection]{Example}
\theoremstyle{remark}
\newtheorem*{remark*}{Remark}\newtheorem*{proof*}{Proof}
\newtheorem{remark}[subsection]{Remark}

\def\A{{\mathcal A}}
\def\Z{{\mathbb Z}}
\def\R{{\mathbb R}}
\def\C{{\mathbb C}}
\def\ome{{\omega}}
\def\Ome{{\Omega}}
\def\N{{\mathbb N}}

\def\MM{{\mathcal M}}

\newcommand{\T}{\mathbb{T}}
\newcommand{\ii}{{\rm i}}
 \title[A normal form of the non--linear Schr\"odinger  equation ]{A normal form of the non--linear Schr\"odinger  equation}
\author{ M. Procesi*, \and 
 C. Procesi{**}. }
 \thanks{ *Universit\`a Federico II  Napoli,supported by ERC under FP7, {**}  Universit\`a di Roma, La Sapienza }

\begin{document}
\begin{abstract}
In this note we discuss normal forms of the completely resonant non--linear Schr\"odinger  equation on a torus, with particular  applications  to   quasi periodic solutions.
 \end{abstract}\maketitle

\section{Introduction}
The aim of this paper is to study an approximate normal form for the completely resonant cubic non--linear Schr\"odinger  equation
\begin{equation}\label{mnls}\ii v_t-\Delta v  =\kappa |v|^2v, \end{equation}
 NLS for short,  on an $n$--dimensional  torus ($n>1$) whose coordinates we denote by $\varphi$ (that is with periodic boundary conditions), with the purpose of applying  KAM Theory as for instance in \cite{EK}, \cite{GY},\cite{GYX}.  We  shall  perform a very detailed study in all dimensions  with the most precise results in dimensions 2 and 3.

The NLS in dimension 1 has a long history, it is completely integrable and   several explicit solutions are known.  Moreover it has a convergent normal form, see \cite{KL}. In higher dimension  we loose  the complete integrability and all the techniques associated to it, but we  can start from the well known fact  that the NLS  \eqref{mnls} is an infinite dimensional Hamiltonian system (see Formula \eqref{Ham}) whose linear part  consists of infinitely many independent oscillators with rational frequencies and hence completely resonant (all the bounded solutions are periodic). The presence of the non-linear term couples the oscillators and modulates the frequencies so that one expects small quasi--periodic (and almost-periodic) solutions to exist for appropriate choices of the initial data.  
To prove the existence of  such quasi--periodic solutions for Hamiltonian PDE's there are two main approaches in the literature: one by KAM theory and the other by using Lyapunov-Schmidt decomposition and then Nash--Moser implicit function theorems (the so--called Craig--Wayne--Bourgain method). 
It is important to notice however that for both approaches it is necessary to  start from a suitably non degenerate normal form and the existence of such normal form is not apparent for equation \eqref{mnls}.

 \noindent The object of this paper is to construct an appropriate normal form and show that it satisfies the hypothesis of a KAM algorithm, we will also briefly discuss how our normal form relates to the other main method of producing quasi--periodic solutions.

\vskip5pt Before describing our results let us give an extremely sketchy overview of the existing literature on quasi--periodic solutions for the NLS.
   
\subsection{Background}   In   the literature on the NLS in higher dimensions, most of the results are restricted to simplified models such as
\begin{equation}\label{parnls}i v_t-\Delta v +M_\sigma v = f(\varphi,|v|^2) v \end{equation}
where $M_\sigma$ is a ``Fourier multiplier'' i.e. a linear operator, depending on a finite number-- say $m$-- of free parameters $\sigma$, which commutes with the Laplacian.  The role of the ``Fourier multiplier'' is to ensure that the equation \eqref{parnls}, linearized at $v=0$ has quasi--periodic solutions with $m$ frequencies 
which excite $m$ Fourier modes (say ${v_1},\dots,{v_m}$ with $v_i\in \Z^n$) and leave the others at rest.  The modes $v_i$ on which the linear motion takes place are called the {\em tangential sites}, and one constructs quasi--periodic solutions of equation \ref{parnls}, which are approximately confined to these Fourier modes. 
\vskip10pt
The first proofs of existence of periodic and quasi--periodic solutions for equation \eqref{parnls} were given by Bourgain, see \cite{Bo2} and \cite{Bo3}.
Then Eliasson and Kuksin, in \cite{EK}, proved both existence and linear stability of  quasi--periodic solutions for equations like \eqref{parnls}, by using KAM theory.
 In this setting the main difficulty   is to prove measure estimates on the set of initial data for which quasi--periodic solutions occur. More precisely  one needs to impose the {\em second Melnikov condition}, this was done  by a subtle analysis of the ``Toeplitz Lipschitz'' properties of the NLS Hamiltonian. 
 \vskip5pt  
In the case of dimension $n=2$ Geng, You and Xu proved the existence of quasi-periodic solutions for equation \eqref{mnls} by a combination of non-integrable normal form, momentum conservation (in the spirit of \cite{GY}) and the ideas of  Kuksin and Eliasson.

The existence of wave packets of  periodic solutions  for equation \eqref{mnls} (as well as for the beam equation)  in any dimension is proved in \cite{GP2} and \cite{GP3}.     In those papers the authors were able to deal also with Dirichlet boundary conditions  for which the normal form is much more complicated.
\subsection{Description of the methods and results}

While trying to understand the connections between  the result of \cite{GYX}  and \cite{GP2}--\cite{GP3}, the first author  started to understand that one could even attack   the much more complicated case of quasi--periodic solutions in case $n>2$. In particular it became clear  that,  the challenging problem of completely understanding   the NLS Hamiltonian after one step of Birkhoff  normal form,  had  subtle combinatorial and geometric aspects which needed a very careful and 
non--trivial analysis which is fully performed in this paper.

Resonant Birkhoff  normal form is a well-known approach to  resonant or degenerate dynamical systems and works very well in the case of the one-dimensional NLS, see \cite{Bou1}. Roughly speaking, consider a Hamiltonian  
$$ H= H^{(2)}(p,q)  + H^{(4)}(p,q)\,,\quad H^{(2)}(p,q)= \sum_k \lambda_k (p_k^2+q_k^2)$$ where $H^{(4)}$ is a polynomial of degree $4$ and the $\lambda_k$ are all rational.

A step of ``resonant Birkhoff  normal form'' is a sympletic change of variables which reduces the Hamiltonian $H$ to 
$$H_N= H^{(2)}(p,q)  + H^{(4)}_{res}(p,q)+ H^{(6)} $$ where $H^{(6)}$ is an analytic function of degree at least $6$ while $H^{(4)}_{res}$ is of degree $4$ and Poisson commutes with $H^{(2)}$. 
Then one wishes to treat $H^{(2)}(p,q)  + H^{(4)}_{res}(p,q)$ as the new unperturbed Hamiltonian and $H^{(6)}$ as a small perturbation. This can work provided that $H^{(2)}(p,q)  + H^{(4)}_{res}(p,q)$ is simple enough (possibly completely integrable) and has quasi-periodic motions for large classes of initial data (which now play the role of the parameters $\sigma$). An ideal situation is when the $\lambda_k$ are non--resonant up to order $4$ so that the normal form is integrable, for example, in
$$ H^{(2)}+ H^{(4)}_{res}= \sum_{k=1}^N \lambda_k (p_k^2+q_k^2)+ \sum_{k=1}^m (p_k^2+q_k^2)^2,$$   the quartic term produces an integrable  twist on the first $m$ frequencies. Then one chooses these as tangential sites and passes to action--angle variables $  p_k^2+q_k^2= \xi_k +y_k$ for $k=1,\dots m$. It is easily seen that the anisochronous twist implies that the linear frequencies now depend on the initial datum $\xi$ and hence for almost all $\xi$ the motions are quasi-periodic.   

Our setting is quite far from being ideal and one has to start by 
  dividing, in a very careful way,  the oscillators into two suitable subsets,  the {\em tangential} and the {\em normal}  sites. This we shall do by imposing the condition that the tangential sites are {\em generic} (cf. Definition \ref{gener} and \S \ref{therin} for a precise statement), with the strategy of  analyzing the equation  near the solutions in which the normal sites are at rest and the tangential sites move quasi--periodically on an $m$-torus.   

This is done by {\em doubling} the action variables in the tangential sites, so each action variable is written as $\xi_i+y_i$ where the $\xi_i$ are treated as parameters  (which are used to parametrize the quasi--periodic or almost--periodic solutions obtained by the perturbative method), the $y_i$ instead remain part of the actual symplectic dynamical variables.  The introduction of these parameters  allows to treat  the orbits  which produce diffusive phenomena as singularities of the perturbative algorithm (the problem of small denominators). This technique was introduced by P\"oschel to prove the existence of lower dimensional tori in finite-dimensional systems and was extended to infinitely-many degrees of freedom in various papers, see \cite{P1}, \cite{BB}.

In our case the normal form Hamiltonian $H^{(2)}+ H^{(4)}_{res}$ is non--integrable and rather complicated. The structure of this normal form was first discussed by Bourgain in \cite{Bo2} and then revisited in \cite{GP2} and\cite{GP3} for the more complicated case of Dirichlet boundary conditions.

The key point in \cite{GP2}  is that, for most of the choices of tangential sites, the leading order of the normal form Hamiltonian is quadratic and block diagonal. 

This is not explicit in  \cite{GP2}, so let us reformulate the result  using, as we shall do in this paper,      complex symplectic variables $w_k=(z_k,\bar z_k)$ for the normal sites and   action angle variables $(\xi +y, x)$ for the tangential sites.
\begin{equation}\label{norham}H^{(2)}+ H^{(4)}_{res}= (\ome(\xi),y) + \frac12 (y, A y)-\frac{1}{2}w M(\xi,x) J w^t +O(w^3) \end{equation} where  $A$ is an invertible twist matrix, $J$ is the standard symplectic matrix and $M(\xi,x)$ is a complex matrix such that the quadratic form $Q= -\frac{1}{2}w M(\xi,x) J w^t$ is real, semi-definite and {\em block diagonal with blocks of uniformly bounded dimension}. The   entries of $M$ depend analytically on the parameters $\xi$ and on the angle variables $x$ and the Hamilton equations for the normal form decouple to (we represent  $w$ as a row vector):
$$\dot y=0\,,\quad \dot x= \ome(\xi)\,,\quad \dot w = i w M(\xi,x_0+\ome(\xi)t) \,. $$

In the present paper we largely improve this result by showing that, for generic choices of the tangential sites, non only the quadratic form  is block diagonal but that the dimension of the blocks is bounded by $n+1$; in particular the fact that the dimension of the block does not depend on the number of tangential sites enables us to consider infinitely many tangential sites, this is crucial if one wants to construct almost--periodic solutions.   

{\bf 1.}\quad By relating the normal form to appropriate combinatorial graphs we describe completely and efficiently the structure of the quadratic form; of particular relevance is the fact that the infinite-dimensional quadratic form is described by a finite number of combinatorially defined graphs.

{\bf 2.}\quad The next substantial result, Theorem \ref{teo1}, consists in proving that the normal form Hamiltonian \eqref{norham}  is reducible to constant coefficients:

$$(\ome(\xi),y)-\frac12 w M'(\xi)J w^t, $$ where the quadratic form is still real but we loose information on the signature of the single blocks.
  In general wether one may reduce a quadratic Hamiltonian to constant coefficients is a difficult question even for finite dimensional systems. We are moreover in an infinite-dimensional context so that there is also a convergence issue to be treated. Due to our wise choice of the tangential sites however the change of variables which reduces the normal form to constant coefficients is completely  explicit and one can check easily all the convergence issues. See \cite{YX} for similar results in the context of finite dimensional systems. 

{\bf 3.}\quad The final step is to diagonalize the Hamiltonian (by a linear change of variables $U(\xi)$) and obtain:
\begin{equation}\label{norma}
(\ome(\xi),y)+\sum_{|k|> C}\tilde\Ome_k(\xi) |z_k|^2 +\tilde Q(\xi,w) +\quad {\rm Perturbation}, 
\end{equation}
where $\ome(\xi),\tilde\Ome(\xi)$ are real and $\tilde Q(\xi,w)$ is a finite dimensional quadratic form (i.e. it involves only a finite number of normal sites $w_h=u_h,\bar u_h$.  The matrix $U(\xi)$ and the functions  $\ome(\xi),\tilde\Ome(\xi)$ are linearly homogeneous analytic functions of $\xi$ for all $\xi$ outside a real algebraic hypersurface.
 See Theorem \ref{teo1b} for notations and details.

{\bf 4.}\quad  We show that in dimension $n=2$ the non-elliptic terms in \eqref{norma} vanish for appropriate choices of the parameters $\xi_i$. 
As a consequence the non-integrable normal form used in the paper \cite{GYX} can be reduced to a standard one, Corollary \ref{coro}.
This allows us to identify which of the solutions  found in  \cite{GYX} are  linearly stable (resp. unstable).

{\bf 5.}\quad  In dimension $>2$  we then study the question of how to apply an abstract  KAM algorithm to our Hamiltonian.   We follow \cite{P1} and \cite{BB}.
A main point is to prove a  ``non-degeneracy  condition'', as stated in \cite{P1}.  We obtain a complete result in dimension 3 and produce a finite (but rather heavy) algorithm to check the property in any dimension.
 Let us denote by $\tilde\Ome_k$ {\em all} the eigenvalues of the matrix $M'$ (as we have said it is possible that a finite number is complex valued).  We study  the set of values   $\xi$, where at least one of  Melnikov resonances 
\begin{equation}\label{mel}
 (\ome(\xi),\nu)=0\,, \quad (\ome(\xi),\nu)+\tilde\Ome_k(\xi)=0\,,\quad  (\ome(\xi),\nu)+\tilde\Ome_k(\xi)+\sigma \tilde\Omega_h(\xi)=0
\end{equation}    occur. In the first equality we assume $\nu\neq 0$ while in the third equality we assume that  $(\sigma,\nu,h,k)\neq (-,0,h,h)$  because these give  trivial resonances. Since equation \eqref{mnls} has the total momentum as a preserved quantity we only need to consider those choices of  $\sigma=+,-$, $\nu\in \Z^m$ and of normal sites $h,k$ which are compatible with momentum conservation (see Proposition \ref{reteo} item {\it v)}).

In Theorem \ref{teo2} we prove in dimensions 2,3 that the set of $\xi$ which satisfies a non-trivial Melnikov resonance  is   of  measure zero  for all choices of  $(\sigma,\nu,h,k)$  compatible with momentum conservation.

These Conditions are necessary in order to follow the KAM algorithm as in \cite{P1}, and sufficient  at least at a formal level,  since the expression in \eqref{mel} appear as denominators in the homological equation.   Even if at the moment we are unable to prove the second Melnikov condition in dimension $>3$  we still show that a slightly weaker statement on separation of eigenvalues, proved in Proposition \ref{wemel} is sufficient to perform the KAM algorithm in our case.  Thus  although in dimensions $\leq 3$ our results are more precise, we still have an answer to the original question in all dimensions.
Finally we briefly discuss how our Theorems lead to the convergence of an abstract KAM scheme for $\xi$ in some {\em Cantor--like set} defined iteratively. To conclude a KAM theorem one should prove that this set is non--empty and of positive measure; we do not discuss here this last question (which can be handled by following \cite{EK}). The explicit construction of a KAM iteration scheme and the measure estimates will appear in a forthcoming paper.

  If one is willing  to give up linear stability results one can probably  use the CWB  method. Then  it is only necessary to check the first two Melnikov conditions and this we do in complete  generality  (Theorem \ref{teo2} item 1.).

In our setting, the singularities (i.e. the values for which one of the Melnikov denominators is zero) appear at the loci where the eigenvalues of some matrices depending parametrically (and polynomially) from the $\xi_i$  coalesce or  become 0.    These loci are  algebraic hypersurfaces, and then the full KAM algorithm producing the space of parameters for quasi--periodic solutions converges outside countably many small   neighborhoods of these hypersurfaces.

The problem arises   in the study of the second Melnikov equation  where we have to understand when it is that two eigenvalues become equal or opposite.
This is essentially equivalent to using the   classical Theory of Sylvester.  The condition  for a polynomial to have distinct roots is the non--vanishing of the discriminant while   the condition  for two polynomials to have a root in common  is the  vanishing of the resultant.  In our case    these resultants and discriminants are polynomials in the parameters $\xi_i$ so, in order to make sure that the singularities are only in measure 0 sets (in our case even an algebraic hypersurface), it is necessary to show that these polynomials are formally non--zero.   This is a purely algebraic problem involving, in each dimension $n$, only finitely many explicit polynomials (cf. \S \ref{matrici}) and so it can be checked by a finite algorithm.

The problem is that, even in dimension 3, the total number of these polynomials is quite high  (in the order of the hundreds or thousands) so that the algorithm becomes quickly non practical. 

In order to avoid this  we have experimented with a conjecture which is stronger than the mere non--vanishing of the desired polynomials.  We expect our polynomials to be irreducible and separated in a sense explained  in \S \ref{SEP}. 

This we prove in dimension 3 (in dimension 2 it is almost immediate), by a mixture of theoretical  arguments and a few direct algorithmic verifications. Then this strong result immediately leads to the analyticity of the $\tilde\Ome_k(\xi)$ in the parameters $\xi$  and the verification of the second Melnikov condition.

In order to prove irreducibility and separation for all dimensions it would be necessary to eliminate the direct verification of some special cases and thus strengthen the theoretical  approach. For the moment this remains conjectural.

Another interesting point is that our reduction algorithm of item {\bf 2.} cannot exclude  that a finite number of blocks in $M$ may have complex eigenvalues $\tilde\Ome_k(\xi)$ for positive  $\xi$.  In general, again following Sylvester Theory, one can state the condition that the $\tilde\Ome_k\in \R$ as  a system of a finite number of explicit polynomial inequalities in the parameters, this determines an open  possibly non--empty region and one has to show that it intersects the positive sector. In dimension $n=2$ this can be done trivially, but already the case of dimension $3$ seems very challenging from a computational point of view.

\section{Notations }
\subsection{Symplectic formalism} Consider a Nonlinear Schr\"odinger  equation on the torus $\T^n$ (NLS for brevity):
\begin{equation}\label{main}iu_t-\Delta u=\kappa |u|^2u +\partial_{\bar u}G(u,\bar u) \end{equation}
where $u:= u(t,\varphi)$, $\varphi\in \T^n$ and  $G(a,b)$ is a real analytic function whose Taylor series starts from degree 6; notice that we  restrict to the case where there  is no explicit dependence on the spatial variable $\varphi$.  It is well known that equation  \ref{main},  the NLS,  is a Hamiltonian equation and has  the momentum  $\int_{\T^n} \bar u(\varphi) \nabla u(\varphi)$ as integral of motion (notice that if $f(u,\bar u)= G(|u|^2)u$, then also the  $L^2$ norm  $ \int_{\T^n} |u(\varphi)|^2$ is preserved).

We shall see that the essential part of the equation is the cubic term.  As for the constant $\kappa $  it can be rescaled to $\pm 1$. In our paper se set it equal to 1 since the other case is quite similar.\smallskip

Passing to the Fourier representation
\begin{equation}
u(t,\varphi):= \sum_{k\in \Z^n} u_k(t) e^{\ii (k, \varphi)}
\end{equation} 
Eq. \ref{main} can be written as an infinite dimensional Hamiltonian dynamical system $\dot u=\{H,u\}$ with Hamiltonian
\begin{equation}\label{Ham}H:=\sum_{k\in \Z^n}|k|^2 u_k \bar u_k + \sum_{k_i\in \Z^n: k_1+k_3= k_2+k_4}u_{k_1}\bar u_{k_2}u_{k_3}\bar u_{k_4}+G(u,\bar u) \end{equation}
on the scale of complex Hilbert spaces 
\begin{equation}\label{scale}
{\bf{\bar \ell}}^{(a,p)}:=\{ u=\{u_k \}_{k\in \Z^n}\;\vert\;\sum_{k\in \Z^n} |u_k|^2e^{2 a |k|} |k|^{2p}:=||u||_{a,p}^2 \le \infty \},
\end{equation}$$\ a>0,\ p>n/2.$$
with respect to the complex symplectic form $ i \sum_{k}d u_k\wedge d \bar u_k$.
\smallskip

These choices are rather standard in the literature:
\begin{remark}\label{cotor}
The condition imposed on $u$ by \eqref{scale} means that:
\begin{itemize}
\item We restrict our study to functions which extend to analytic functions in the domain of the complex torus   $ \C ^n/2\pi\Z^n$ where  $(z_1,\ldots,z_n)\in \C^n,\ Im(z_i)\leq a$.
\item The   functions on the boundary are  in the  Sobolev space $H^p$.
\item The condition $p>n/2$  implies that the function space under consideration embeds in $L^\infty$.
In particular the following  uniform bound holds for each $u\in \bf{\bar \ell}^{(a,p)}$:
\begin{equation}
\label{insn}|u_k|\leq C(s,a)\frac{\Vert u\Vert_{a,p} e^{-a|k|}}{\langle k\rangle^{p-n/2}} ,\quad\langle k\rangle:=\max(1,|k|).
\end{equation}
  In fact  this implies that $\bf{\bar \ell}^{(a,p)}$ has a Hilbert algebra structure.
\end{itemize} 
\end{remark}

\subsection{Analytic Hamiltonians}
We consider real
Hamiltonians on the space ${\bf{\bar \ell}}^{(a,p)}$ which can be formally expanded in Taylor series
$$F= \sum_i F^{(i)}(u,\bar u)= \sum_i \sum_{(\alpha,\beta)\in \N^\infty\atop |\alpha+\beta|_1= i}F^{(\alpha,\beta)}\prod_{k\in\Z^n}u_k^{\alpha_k}\bar u_k^{\beta_k}  $$ 
where $F^{(i)}$ is a multilinear form of $i$ variables  $(u,\bar u)\in {\bf{\bar \ell}}^{(a,p)}\times {\bf{\bar \ell}}^{(a,p)}$.

We require that the series  is totally convergent in    some ball of positive radius, so that by definition the function is analytic.  Usually this will be verified, using Formula \eqref{insn},  by determining a positive $r$ so that if $\Vert u\Vert_{a,p}<r$ 
$$ \sum_i  \sum_{(\alpha,\beta)\in \N^\infty\atop |\alpha+\beta|_1= i}|F^{(\alpha,\beta)}|\prod_k\frac{ e^{-a (\alpha_k+\beta_k)|k|}}{ \langle k\rangle ^{(\alpha_k+\beta_k)(p-n/2)}}\Vert u\Vert_{a,p}^i<\infty.$$
Clearly the NLS Hamiltonian belongs to this class and is convergent on any ball $B_r$.

We will also be interested in Hamiltonian vector fields $X_F = \{ \partial_{\bar u_k} F\}_{k\in \Z^n}$, where $F$ is an analytic Hamiltonian in the ball $B_r$. With this notation the Hamilton equations are $  \dot u_k = iX_F$.
We will require that the vector field maps $B_r\to {\bf\bar\ell}_{a,p}$ and  we will use the norm
$$ |X_F|_{r}:= \sup_{u\in B_r}\Vert X_F\Vert_{a,p}/r.$$
notice that if $ |X_F|_{r}<1/2$ then the Hamiltonian flow is well defined up to $t=1$  and depends analytically on the initial data so that the symplectic change of variables is well defined and analytic say from $B_{r/2}\to B_r$.
Notice that this condition is NOT verified by $H^{(2)}$ but just by $H^{(4)}$.
 
\vskip10pt
\subsection{Elliptic--action angle variables\label{Eaav}}
As explained in the introduction we will be interested in non symmetric domains $D \subset B_\epsilon$ 
where some of the variables $u_k$  (the tangential sites) are bounded away from zero and hence can be passed in action angle variables while all the other variables are much smaller.

\smallskip

We first  partition $$\Z^n= S\cup S^c,\quad S:=(v_1,\ldots,v_m) $$ where:
the set  $S$ are called {\em tangential sites} and $S^c$ the {\em normal sites}. We then introduce the 
elliptic action angle coordinates $(\xi +y, x)\times (z,\bar z)\in \R^m\times \T^m\times  {\bf{\ell}}^{(a,p)}$  
where $\T^m:=\R^m/2\pi\Z^m$ and we denote by   $ {\bf{\ell}}^{(a,p)}$   the subspace of $\bf{\bar \ell}^{(a,p)}\times\bf{\bar \ell}^{(a,p)} $  generated by the indices in $S^c$ and  $w=(z,\bar z)$ (considered as row vectors) are the corresponding coordinates. As explained in the introduction the $\xi$ are positive parameters while the $(y,x)$ are the conjugate dynamical variables. The symplectic form is $$ dy\wedge dx +d iz\wedge d\bar z.$$

We consider the domain $$ A_\alpha\times D(s,r):= $$\begin{equation}\label{domain}
 \{  \xi\,:\  \frac12 r^\alpha\le \xi_i\le r^\alpha\,\}\times \{   x,y,w\,:\   x\in \T^m_s\,,\  |y|\le r^2\,,\  |w|_{a,p}\le r\}
\end{equation} $$\subset  \R^m\times \T^m_s\times\C^m\times {\bf{\ell}}^{(a,p)}.$$  Here $0<\alpha < 2,0<r<1,0<s$ are parameters. $\T^m_s$ denotes the open subset of the complex torus $\T_{\C}^m:=\C^m/2\pi\Z^m$ where   $ x\in\C^m,\ |$Im$(x)|<s$ (cf. Remark \ref{cotor}). 
\begin{remark}
It is possible and useful to treat also the case $m=\infty$. In this case we use in Formula \eqref{domain} an exponentially decaying  norm $|y|_{a,p}$ and for $ A_\alpha$ a condition of the form $\frac12 r^\alpha\le \xi_i e^{a |v_i|}\le r^\alpha$.
\end{remark}

Some comments on the choice of this domain are in order.  The main point is the inequality in the $\xi$ which is required to  determine a domain  where the KAM algorithm can be performed  outside the singularities to be determined.  The inequality on the $y$ and on the $w$  are just auxiliary and only used to keep the computations inside the domains of convergence.  The use of the complex domain $\T^m_s$ is motivated by the need to insure that the solutions that we shall find have some analytic behavior.

Given a Banach space $E$, a function $F(\xi,y,x,w):A_\alpha\times D(s,r)\to E$ is said to be analytic in $x,y,w$  if its Taylor-Fourier series in these variables is totally convergent in the domain $ A_\alpha\times D(s,r)$.  We use as norm the sup-norm
$$ \Vert F\Vert_{s,r}= \sup_{A_\alpha\times D(s,r)}| F|_E. $$
As for the variables $\xi$ we will require less regularity, namely we use  weighted   Lipschitz norms
	$$ \Vert F\Vert^\lambda_{s,r}= \Vert F\Vert_{s,r}+\lambda \sup_{\xi\neq \eta \in A_\alpha\,,\;(y,w)\in D(s,r)}\frac{|F(\eta)-F(\xi)|}{|\eta-\xi|}.  $$

In this new set of variables we are also interested in Hamiltonian vector fields $$X_F= \{\partial_y F, -\partial_x F, J d_w F\} $$ which map $D(s,r)\to \T^m_s\times\C^m\times {\bf{\ell}}^{(a,p)}$. For such vector fields we will use the norm
$$ | X_F|_{s,r}:=\sup_{A_\alpha\times D(s,r)}( |\partial_y F| +r^{-2}|\partial_x F|+ r^{-1}\Vert d_w F\Vert_{a,p}) $$ and the corresponding weighted $C^1$ norm.
The Hamilton equations are then
$$\dot x= \partial_y F\,,\quad \dot y=-\partial_x F\,,\quad  \ii \dot w=   J d_w F.  $$
Of particular interest are the ``real quadratic Hamiltonians'' in the normal frequencies; we will represent them by matrices with the notation
$$ Q(w):= -\frac12 w M J w^t \,,\quad J=\begin{pmatrix} 0 & -I\\ I & 0 \end{pmatrix} $$ so that the Hamiltonian  vector field is $ X_Q = M$. The matrices $M$ which may appear are in the Lie algebra  of the group of real symplectic transformations, see Remark \ref{sym} and Proposition \ref{pois}.

\section{Main results}
\subsection{The Theorems}
We are now ready to state our Theorems.

We shall find a finite list  (say  $P_1(y),\dots, P_{N}(y)$, where $N$ depends only on $n$) of non--zero polynomials with integer coefficients   depending on $  2n$ vector variables $y=(y_1,\dots,y_{2n})$ with $y_i\in \C^n$. We call the $P_i$ {\em the resonances,  see \S \ref{therin} }. 

 \begin{definition}\label{gener}
We say that a list of    tangential sites $S=\{v_1,\dots,v_m\}\in\Z^{nm}$ is  {\em generic}  if for any subset  $A$ of $S$ such that $|A|=2n$, the evaluation of the resonance polynomials at $A$ is non--zero. 
\end{definition}If $m$ is finite this condition is equivalent to requiring that $S$ (considered as a point in  $\Z^{nm}$) does not belong to the algebraic hyper--surface where at least one of the resonance polynomials is zero.

We find also a finite list $\mathcal M$ of matrices of dimensions $2\leq k\leq n+1$ with entries polynomials in the elements $\sqrt {\zeta_i}$  for a list $\zeta_i,\ i=1,\ldots,2n$ of auxiliary variables.  We shall denote by $\mathcal M(\xi)$ the list of matrices obtained  by substituting to the variables $\zeta_i$ any $2n$ elements of the list $\xi_1,\ldots,\xi_m$ in all possible ways.

 Given any $m\in \N,\ 0<\alpha<2$ and  appropriately small $s,r$, the following holds (cf. \ref{reteo}):
\begin{theorem}\label{teo1}For all generic choices $S=\{v_1,\dots,v_m\}\in\nobreak  \Z^{nm}$ of the tangential sites,    there exists   an analytic symplectic change of variables
$$\Phi: ( y, x)\times (z,\bar z) \to (u, \bar u) $$   from $ A_\alpha\times D(s,r) \to B_{ r^{\alpha/2}}$  with the following properties.

i)\  the Hamiltonian \eqref{Ham} in the new variables  is analytic and has the form
$$ H\circ\Phi= (\ome(\xi),y)+ \frac12 (y,A(\xi)y)  -\frac12 w M'(\xi) J w^t + P(\xi,y,x,w)\,,$$ where

ii) \ $\ome_i(\xi)= |v_i|^2 -2\xi_i +4\sum_j \xi_j$ and 
the matrix $A(\xi)$ has 1 on the diagonal and 2 off diagonal and hence it is invertible. 

iii)\ The matrix      $M'(\xi)$  is   a block--diagonal matrix with the following properties: all except a finite number of the blocks are self adjoint; all the blocks are sum of a scalar matrix plus a term chosen from the finite list $\mathcal M(\xi)$.

iv) \ The perturbation $P$ is of the form $P(\xi,y,x,w)=P^{(3)}(\xi,x,y,w)+ P^{(6)}(\xi,x,y,w)$, where-- setting $\lambda= r^{\alpha}/\max_i(|v_i|^2)$--  we have   the bounds:
$$\Vert P^{(3)}(\xi,x,y,w)\Vert^\lambda_{s,r}\leq C r^{3+\alpha/2}\, \quad 
\Vert P^{(6)} (\xi,x,y,w)\Vert^\lambda_{s,r}\leq C r^{\max( 5\alpha,1+ \frac52 \alpha) },$$ moreover  $P^{(3)}(\xi,x,y,w)$ is at least cubic in $w$.
\end{theorem}
Completely explicit statements will appear in  \S\ref{reinf}.

{\bf Outline of the proof.} The symplectic change of variables $\Phi$ is the composition of three steps: first we perform one step of Birkhoff  normal form passing to the Hamiltonian \eqref{Ham2}, then we pass to the elliptic action angle variables and obtain the Hamiltonian \eqref{Ham4}. We study the block form of the quadratic part of this Hamiltonian, see  \eqref{hama}, and produce a list of constraints on the tangential sites so that \eqref{hama} is as simple as possible, this is the core of the proof.  Then we perform a simple but non-perturbative symplectic change of variables, see  \eqref{Nc}, so that the quadratic Hamiltonian \eqref{hama} is reduced to constant coefficients. \qed
\smallskip
 
 When we apply the theory of normal forms for quadratic Hamiltonians to $M'$ we obtain.
\begin{theorem}\label{teo1b}
 There exists a real algebraic hypersurface  ${\mathfrak A}$ such that the following holds. There exists a linear change of variables in the $(z,\bar z)$ depending analytically on $\xi$ for all $\xi\in A_\alpha\setminus {\mathfrak A}$, such that in the new variables the Hamiltonian  is
 \begin{equation}\label{hfin}  H_{\rm fin}= (\ome(\xi),y)+ \frac12 (y,Ay) + \sum_{k\in S^c\setminus {\mathcal H}} \tilde\Omega_k |z_k|^2 -\frac12 w R J w^t \end{equation}$$ + P^{(3)}(\xi,x,y,w)+ P^{(6)}(\xi,x,y,w),$$
where:

i)\   $\mathcal H$ is a finite subset of $S^c$ and $R$ is a matrix with only a finite number of non zero entries.

ii)
Let us call $\tilde\Omega_k $ with $k\in {\mathcal H}$ the eigenvalues\footnote{naturally the matrix may not be diagonalizable}  of $R$, we have
 $$ \tilde\Omega_k =  |k|^2 -\sum_i A^{(i)}_k |v_i|^2 + \lambda_k(\xi)\,,\quad \forall k\in S^c$$  
The $ A^{(i)}_k$ are integers with  $\sum_i | A^{(i)}_k|< 2n$ and the correction $\lambda_k(\xi)$   is chosen in a finite list, say
\begin{equation}\label{autov}
\lambda_k(\xi) \in \{ \lambda^{(1)}(\xi),\dots,  \lambda^{(K)}(\xi)\}\,,\quad K:= K(n,m),
\end{equation}
of different analytic functions of $\xi$.

ii)\ The functions $\lambda^{(i)}(\xi)$ are homogeneous of degree one in $\xi$. This implies that for $\xi\in A_{\alpha}\setminus \mathfrak{A}_{\varepsilon}$-- where  $\mathfrak{A}_{\varepsilon}$ is the  spherical neighborhood of $\mathfrak{A}$ of radius $\varepsilon = r^{\alpha}/4$-- the $\lambda^{(i)}(\xi)$   satisfy  the bounds
\begin{equation}
\label{linh} |\lambda^{(i)}(\xi)|\leq Cr^{\alpha} \,,\quad c r^{\alpha}\leq |\lambda^{(i)}(\xi)\pm \lambda^{(j)}(\xi)|\leq Cr^{\alpha}\,, \quad 
|\nabla_{\xi}\lambda^{(i)}(\xi)| \leq  C. 
\end{equation} 
iii)\  For $\xi\in A_{\alpha}\setminus \mathfrak{A}_{\varepsilon}$ item { v)} of Theorem \ref{teo1} holds. 
 \end{theorem}
 \begin{remark}\label{mini}
Notice that we may further restrict the  set $\mathcal H$ by requiring that all the block matrices in $R$ are either non--diagonalizable or have all complex eigenvalues. Indeed all the other blocks may be put in the elliptic normal form by a symplectic change of variables. In dimensions $2,3$ we show that also $R$  is diagonalizable for generic values of $\xi$ over the complex numbers. This gives an hyperbolic normal form where for each complex eigenvalue $\tilde\Ome_k$ one has a two--by--two diagonal block with Re$(\tilde\Ome_k)I +$  Im$(\tilde\Ome_k) J$, where $I$ is the identity and $J$ the symplectic matrix.
\end{remark}

 \medskip 

\begin{corol}\label{coro} For $n=2$ and for appropriate choices of the tangential sites $S$ there exists  a cone--like domain $D$, with meas$ (D\cap A_\alpha)\sim r^{\alpha}$, such that for all $\xi\in D$ 
the set $\mathcal H$ in Theorem \ref{teo1b} is empty. On $A_\alpha\setminus D$ the set   $\mathcal H$ is non--empty and and such that $R$ is  in hyperbolic normal form (see Remark \ref{mini}).
\end{corol}

\begin{theorem}\label{teo2} 
\begin{enumerate}[1.]
\item The set of parameter values $\xi$ for which the first two  Melnikov resonances  in \eqref{mel}
occur has zero measure.
  \item  For $n=2,3$ the set of parameter values $\xi$ such that the three Melnikov resonances \eqref{mel}
occur is  a zero measure set (and for each condition it is algebraic).

\end{enumerate}
\end{theorem}
\begin{remark} To prove this theorem it will be essential that  the NLS equation \eqref{mnls} preserves the total momentum, this implies that the  $(\sigma,\nu,h,k)$ in \eqref{mel} must satisfy a { \em momentum conservation} relation, see Proposition \ref{reteo}.
\end{remark}
The proof is based on a careful analysis of the characteristic polynomials of the matrices in $\mathcal M $ (cf. theorem \ref{teo1}, iv)) and is carried out in sections  \ref{irrid}  and \ref{SEP}.
\begin{remark}
In dimension $>3$  we shall prove a weak form of the second Melnikov resonance condition (cf. Proposition \ref{wemel}). In this form the eigenvalues may have multiplicities but outside of a  set of measure zero these multiplicities are finite and uniformly bounded. Moreover the eigenspace  for any given eigenvalue  is isotropic, it  pairs under Poisson bracket only with the eigenspace for the opposite eigenvalue. This is enough to perform a KAM algorithm by solving an appropriate homological equation, see \S \ref{homeq}
\end{remark}

{\bf The  case $m=\infty$.\quad}
  There are infinitely many infinite sets $S$ which satisfy the  non--resonance conditions (\S \ref{therin}). For these choices of tangential sites most of the previous statements hold verbatim. Some of the quantitative estimates need a more careful analysis. 
  
  \subsection{Quasi--periodic solutions for the NLS equation \eqref{mnls}}\label{conclu}
  As we have stated in the introduction there are two main methods for finding quasi--periodic solutions for non--linear PDE's near to an elliptic fixed point. Let us briefly review how our results on normal form relate to these methods. The technical issues involved in applying these methods for the completely resonant NLS will be analyzed in a separate paper. 
  
  1. The Craig--Wayne--Bourgain method 
  
  The CWB method is based on a Liapunov--Schmidt decomposition  combined with a generalized Newton method to overcome the small divisor problems.  
 We follow the approach in  \cite{BBhe}  which provides both an extension of these results and a clear exposition  using the Nash--Moser approach.
 
 We look for a solution of \eqref{mnls} of the form $v(t,\varphi)= u(\ome t,\varphi)$ where $u: \T^m\times \T^n\to \C$ lives in a Sobolev space of  functions  with a Hilbert algebra structure.
  
   Our theorems \ref{teo1},  \ref{teo1b}  imply that the bifurcation equation for \eqref{mnls} admits a  solution of the form $ u_0(x,\varphi)+ O(\xi^{3/2})$ where
  $$ u_0= \sum_{i=1}^m \sqrt{\xi_i} e^{\ii (x_i+v_i\cdot \varphi)}\,.$$  Consider the bifurcation equation  linearized at $u_0$ and written in the Fourier basis both in the space variables $\varphi$ and in the angles $x$, this equation is represented by an infinite   matrix say $L$. From theorems  \ref{teo1},  \ref{teo1b} $L$ is block diagonal and the  blocks are very simply associated to the blocks of $M'$.  Moreover the eigenvalues of $L$ are  $\ome(\xi)\cdot \nu \pm \tilde\Omega_k(\xi)$ with $\nu\in \Z^m, k\in S^c$. By Theorem  \ref{teo2} item 1.  the eigenvalues are not identically zero. 
  
  Notice finally that since the matrix is block diagonal (with blocks of dimension $\leq n+1$) its invertibility implies invertibility in any norm with the same bounds, and one may impose that $$ || L^{-1} u||_s\leq C ||u||_{s+\tau}$$ where $|| v||_s$ is a Sobolev norm in the space and angle variables.
 From this point we expect to be able to follow the same scheme as in \cite{BBhe}, with only small technical variations.
 \smallskip
 
 2. KAM theory: let us first restrict to the case $n=2,3$ where, by Theorem \ref{teo2}  item 2,  the  second Melnikov conditions hold. The case $n=2$  is essentially already  covered by \cite{GYX}. Our analysis produces several improvements  in their results, in particular for the parameter values $\xi\in D$ we show that the quasi--periodic solutions obtained in \cite{GYX} are linearly stable while for $\xi\in A_\alpha\setminus D$ the solutions are unstable. \smallskip
 
 In the literature (see for instance \cite{BB}) {\em abstract KAM schemes} are based on three main assumptions:
 
 1. A {\em smallness condition} on the perturbation $P$, in our case this is Theorem \ref{teo1b} item {\it iii)}.
 
 2. A {\em regularity condition}, namely  $\ome(\xi)$ must be a diffeomorphism and $\tilde\Ome_k(\xi)-|k|^2$ must be a bounded Lipschitz function (we have analyticity and the bounds \eqref{linh}, by Theorem \ref{teo1b}).
 
 3. A  {\em non--degeneracy condition}, that is the three Melnikov conditions \ref{mel} which in our case are proved in Theorem \ref{teo2}, see also the discussion in \S \ref{homeq}. \smallskip

As we already mentioned a full discussion of the KAM algorithm will appear elsewhere.  
\bigskip

 The abstract KAM schemes produce 
  an analytic  change of variables, depending on $\xi$,  $\Phi= e^{ad(F)}: D(s/4,r/4)\to D(s,r)$ such that
 $$e^{ad(F)}\circ H= (\ome^\infty(\xi),y^\infty)+ \sum_{k\in S^c\setminus\mathcal H}\Ome^\infty_k(\xi)|z^\infty_k|^2 -\frac12(w^\infty)^t R_\infty w^\infty + P^{\infty}, $$ where $\ome^\infty$ and $\Ome^\infty$ are parameters to be determined and  $R_\infty$ has the same structure and is simultaneously diagonalizable with $R'$.
 Finally $$P^{\infty}=\sum_{i,j:\; 2i+j>2} P^{\infty}_{ij}(x^\infty) (y^\infty)^i (w^\infty)^j.$$ In the new variables one immediately shows the existence of the quasi--periodic solution $$ y^\infty=0\,,\quad w^\infty=0\,,\quad x^\infty= x^\infty_0 +\ome^\infty t .$$
  $\Phi$ will be defined for $\xi$ in some complicated Cantor--like set. This part of the algorithms can be performed in our case by following almost verbatim the  KAM Theorem 5.1 of \cite{BB}. 
  
  The final issue is to analyze this set, show that it is non--empty and give lower bounds on its measure. 
To give a flavor of the type of computations required consider the Cantor set  $\mathcal C$, which appears at the first step of the algorithm,  defined by
$$
|(\ome(\xi),\nu)|> \frac{\gamma (A+r^\alpha)}{|\nu|^{\tau_0}}\,, \quad |(\ome(\xi),\nu)+\tilde\Ome_k(\xi)|> \frac{\gamma (A+r^\alpha)}{|\nu|^{\tau_0}}\,,$$
\begin{equation}\label{bobo}\quad  |(\ome(\xi),\nu)+\tilde\Ome_k(\xi)+\sigma \tilde\Omega_h(\xi)|> \frac{\gamma (A+r^\alpha)}{|\nu|^{\tau_0}}
\end{equation} 
 for all non--trivial resonances which are compatible with momentum conservation, moreover $A=1$ if the corresponding  function on the left hand side is zero at $\xi=0$ and $A=0$ otherwise. \begin{corol}[Of Theorems \ref{teo1b}, \ref{teo2}]\label{canto}For $n=2,3$ and for appropriate choices of $\tau_0$ and $\gamma$, the set $\mathcal C$ has positive measure in $ A_{\alpha}\setminus \mathfrak{A}_{\varepsilon}$.
\end{corol}
 For a sketch of the proof see \S \ref{dim23}.

 This shows that the first step of the KAM algorithm produces a  set of positive measure where the desired symplectic change of variables is well defined. As shown by \cite{EK}, in order to give similar estimates at all steps of the KAM iteration one needs to use the  {\em T\"oplitz--Lipschitz property} of the NLS Hamiltonian. We do not discuss this last property in the present paper. Notice however that in Theorem \ref{teo1} and \ref{teo1b} even though  the Hessian matrix  
 $ \partial_{z_k}\partial_{\bar z_h} (P^{(3)}+P^{(6)}), $  is not  a T\"oplitz matrix it still satisfies  {\em T\"oplitz--Lipschitz properties} (as is discussed in detail in \cite{GYX} for the case $n=2$).
We expect to obtain the measure estimates by following  \cite{EK} and \cite{GYX}.

The case $n>3$ is discussed more in detail in subsection \ref{n4}. 
  \part{The study of the dynamics}
 \section{Preparation}
\subsection{Conservation laws}
Recall the laws of Poisson brackets:
\begin{equation}
\label{ffpb}\{iu_k,\bar u_h\}= \delta_{h,k}\,,\quad  \{u_k, u_h\}=\{\bar u_k, \bar u_h\}=0; 
\end{equation}
hence
\begin{equation}
\label{rupb}\{ u_h\bar u_k,u_j\}= i\delta_{j,k}u_h  ,\quad \{u_h \bar  u_k, \bar u_j\}=-i\delta_{h,j}\bar u_k 
\end{equation} in particular \begin{equation}
\label{rupb1}\{|u_k|^2,u_h\}= i\delta_{h,k}u_h  ,\quad \{|u_k|^2,\bar u_h\}=-i\delta_{h,k}\bar u_h 
\end{equation} 
\begin{definition}
We set 
\begin{equation}
M:= \sum_{k\in\Z^n} k |u_k|^2 ,\quad\text{momentum}.\end{equation}
\end{definition}A Hamiltonian  
 defined in ${\bf{\bar \ell}}^{(a,p)}$ which Poisson-commutes with $M=\sum_{k\in\Z^n} k |u_k|^2 $ satisfies the  constraint of  {\em conservation of momentum:}$$\quad {\rm If}\;  F^{(\alpha,\beta)}\neq 0\,,\quad {\rm one}\; {\rm has}\; \sum (\alpha_k -\beta_k) k = 0;  $$ 
 
 A Hamiltonian defined in ${\bf{\bar \ell}}^{(a,p)}$ is {\em even} if it is  sum of monomials of even degree:
Namely if $ F^{\alpha,\beta}\neq 0 $ then $\alpha+\beta$ is even.
 \subsection{Normal form}
The quadratic part $ H^{(2)}:=\sum_k |k|^2 |u_k|^2 $ of \ref{Ham}  is an infinite string of harmonic oscillators with   all rational  frequencies so that the system is {\em completely resonant} (all the bounded solutions are periodic). 

 Using the conventions of Lie Theory we shall always denote by $ad(F)$ the operator of Poisson bracket $X\mapsto \{F,X\}$.

For small $u$ ( i.e. $||u||_{a,p} <\epsilon \ll  1$) we  perform 
a standard step of ``resonant'' Birkhoff normal form removing all the terms of order four  of $H$ which do not Poisson-commute with the quadratic part, see also Remark \ref{topo}.
 
 In fact, by \eqref{rupb},  $u_{k_1}\bar u_{k_2}u_{k_3}\bar u_{k_4}$ is an eigenvector with respect to $\{H^{(2)},-\}$ with eigenvalue $i( |k_1|^2-|k_2|^2+|k_3|^2-|k_4|^2)$. Thus   we perform the  symplectic change of variables $H\mapsto e^{ad(F)}(H)$, generated by the flow of
\begin{equation}\label{forma} F:=  -i \sum_{k_i:\; k_1-k_2+k_3-k_4=0\atop |k_1|^2-|k_2|^2+|k_3|^2-|k_4|^2\neq 0} \frac{u_{k_1}\bar u_{k_2}u_{k_3}\bar u_{k_4}}{|k_1|^2-|k_2|^2+|k_3|^2-|k_4|^2}. \end{equation}  For $\epsilon$ sufficiently small, this is a well known analytic change of variables  (cf. \cite{Bo3},\cite{Bo2},\cite{bamb})   $ {\bf{\bar \ell}}^{(a,p)}\supset  B_\epsilon \to B_{2\epsilon} \subset {\bf{\bar \ell}}^{(a,p)} $ (where $B_\epsilon$ denotes as usual the open ball of radius $\epsilon$)  which brings (\ref{Ham}) to  the form:
\begin{equation}\label{Ham2}H_N:=\sum_{k\in \Z^n}|k|^2 u_k \bar u_k + \sum_{ {k_1+k_3= k_2+k_4} \atop { |k_1|^2+|k_3|^2=| k_2|^2+|k_4|^2} }u_{k_1}\bar u_{k_2}u_{k_3}\bar u_{k_4} +P^{(6)}(u)\end{equation} where $P^{(6)}(u)$ is analytic of degree at least $6$ in $u$ and on the ball $B_\epsilon$ it   is bounded by $C\epsilon^6$ with $C$ a suitable constant.  Since  we will take $\epsilon$   small, $P^{(6)}(u)$ is   small   with respect to the terms of degree $2,4$ which are bounded by $C_1\epsilon^2,C_2\epsilon^4$ respectively (cf. \S\ref{afwd}). Notice that $F$  commutes with $M$ so that $H_N$  still satisfies momentum conservation, moreover $F$ is even and hence $H_N$ is still even.
\vskip10pt
 Denote by $\mathcal P':=\{(k_1,k_2,k_3,k_4)\,|\,  {k_1+k_3= k_2+k_4},\  { |k_1|^2+|k_3|^2=| k_2|^2+|k_4|^2}\}.$

\begin{figure}[!ht]
\begin{minipage}[c]{9cm}{Trivial computations show that the condition  $$k_1+k_3= k_2+k_4, \qquad  |k_1|^2+|k_3|^2=| k_2|^2+|k_4|^2$$ is equivalent to \begin{equation}\label{vinco}k_1+k_3= k_2+k_4,\qquad( k_1-k_2,k_3-k_2)=0\end{equation}
}\end{minipage}\hskip20pt\begin{minipage}[c]{3cm}
{
\psfrag{a}{$k_1$}
\psfrag{b}{$k_2$}
\psfrag{c}{$ k_3$}
\psfrag{d}{$ k_4$}

\includegraphics[width=3cm]{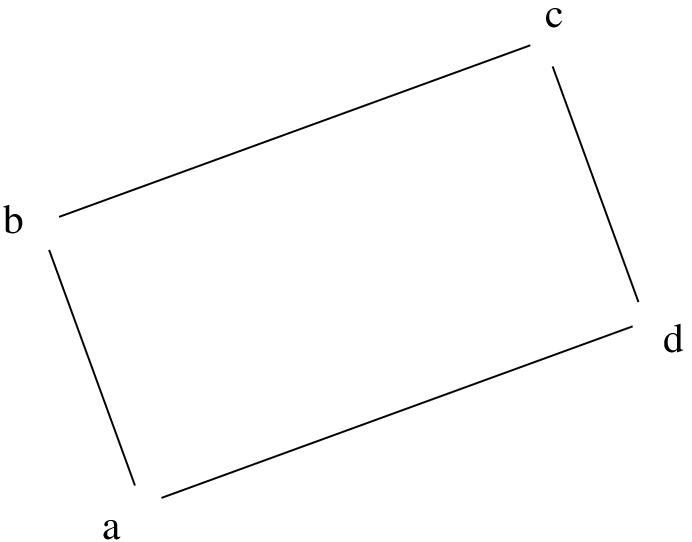}
}
\label{fig0}
\end{minipage}
\end{figure}
In this set  the integer vectors $k_1,k_2,k_3,k_4$ form the vertices of a  {\em rectangle}.
We want to put in  evidence the terms where the rectangle degenerates to a segment.   Thus define $\mathcal P$  to be the subset of $\mathcal P'$ where the rectangles are non degenerate. We obtain: \begin{equation}\label{Ham3} H_N = \sum_{k\in \Z^n}|k|^2 u_k \bar u_k + \sum_k|u_k|^4+  2 \sum_{k_1\neq k_2} |u_{k_1}|^2|u_{k_2}|^2+ \!\!\!\!\!\!\sum_{ (k_1,k_2, k_3,k_4) \in\mathcal P }u_{k_1}\bar u_{k_2}u_{k_3}\bar u_{k_4} \end{equation} \qquad\qquad\qquad $ + P^{(6)}(u)$

 \vskip10pt\subsection{Choice of the tangential sites}
 \vskip10pt
Let us now partition $$\Z^n= S\cup S^c,\quad S:=(v_1,\ldots,v_m),$$ where:
The set  $S$ are called {\em tangential sites} and $S^c$ the {\em normal sites}.
\begin{constraint}\label{co1} $S$ is a finite set (say $|S|=m$)
 such that, given any three distinct elements $v_1,v_2,v_3\in S$, one has that 
$( v_1-v_2,v_3-v_2)\neq 0$. 
\end{constraint}
\begin{remark}
We shall discuss later the extension to $S$ infinite.
\end{remark}In the following we will consider various other constraints on $S$ in order to obtain the simplest possible expression for the Hamiltonian $H$, with respect to this  choice.

Our aim is to  study  the NLS near the tori  associated to the oscillators  $v_i$ keeping the other oscillators constant in time at $u_k=0$. 
The terms of order four in the normal form introduce what is called a {\em twist}. That is an anisochronous term such that the frequency depends on the initial datum $|u_{v_i}|^2=\xi_i$ and $u_k=0$ for $k\in S^c$. 

Let us now set $$u_k:= z_k \;{\rm for}\; k\in S^c\,,\qquad u_{v_i}:= \sqrt {\xi_i+y_i} e^{\ii x_i}\;{\rm for}\;  i=1,\dots m;$$  this is a well known symplectic change of variables which puts the tangential sites in action angle variables  $(y;x)= (y_1,\dots, y_m;x_1,\dots, x_m) $ close to the action $\xi= \xi_1,\dots, \xi_m$, which we now consider as parameters for the system.  The  symplectic form is now $ dy \wedge dx + i \sum_{k\in S^c} dz_k\wedge d \bar z_k $.

   It is convenient to think of the $z_k,\bar z_k$ as a vector  $w$ and denote, for a function of $w$  by $\pd{F}{w}$  the gradient vector (which we think of as a {\em column}). Further denote by $J$ the infinite skew symmetric matrix  $\begin{vmatrix}
0&-1\\1&0
\end{vmatrix}$  where the first block is over the basis $z_k$ and the  second over $\bar z_k$.
Poisson bracket with respect to this form is 
\begin{equation}\label{PoBa}
 \{F,G\}= \pd{F}{y}.\pd{G}{x}-\pd{G}{y}.\pd{F}{x}+    i(\pd{F}{w},J\pd{G}{w})
\end{equation}

 By constraint \ref{co1}, the Hamiltonian \ref{Ham2} can be written as
 \begin{equation}\label{Ham4} H= H_0 + P^{(3)}(z,y; \xi,x)+P^{(6)}(z,y; \xi,x)\,,\quad {\rm with}\quad 
 \end{equation} $$ H_0:=\sum_{i=1}^m( |v_i|^2(\xi_i+y_i) +  (\xi_i+y_i)^2) +4\sum_{i<j} (\xi_i+y_i)(\xi_j+y_j)$$ $$+4\sum_{i;k\in S^c}(\xi_i+y_i)  | z_{k }|^2+\sum_{k\in S^c}|k|^2 |z_k|^2+ 4\sum^*_{  i\neq j ; h,k\in S^c}\sqrt{(\xi_i+y_i)(\xi_j+y_j)}e^{\ii(x_i-x_j)}z_{h}\bar z_{k}  $$ 
  $$+ 2\!\!\!\!\! \sum^{**}_{ i<j\,; h,k\in S^c }\!\!\!\sqrt{(\xi_i+y_i)(\xi_j+y_j)}e^{-\ii(x_i+x_j)}z_{h} z_{k} +
 2\!\!\!\!\!\!\sum^{**}_{ i< j\,; h,k\in S^c }\!\!\!\sqrt{(\xi_i+y_i)(\xi_j+y_j)}e^{\ii(x_i+x_j)}\bar z_{h}\bar  z_{k}  $$ 
  
\begin{definition}\label{decs}
Here $\sum^*$ denotes that $(h,k,v_i,v_j)\in\mathcal P$: 
 $$ \{(h,k,v_i,v_j)\,|\,  {h+v_i= k+v_j},\  { |h|^2+|v_i|^2=| k|^2+|v_j|^2}\}.$$
  and $\sum^{**}$, that  $(h,v_i, k,v_j)\in\mathcal P$:
   $$ \{(h,v_i,k,v_j)\,|\,  {h+k= v_i+v_j},\  { |h|^2+|k|^2=| v_i|^2+|v_j|^2}\}.$$
 \end{definition}

\smallskip

The term  $P^{(3)}$ collects all terms in which at least 3  indices  $k$ are in $S^c$ and it is of degree at least three (and at most four) in $z,\bar z$. Recall that we have assumed (constraint \ref{co1}) that no non degenerate rectangles contain 3 elements of $S$.   

\begin{remark}\label{topo} Notice that, once we have fixed the tangential sites, we have some freedom in the choice of the Birkhoff normal form transformation $F$ in \eqref{forma}. Indeed one may choose first the tangential sites and then choose $F$ as follows:
$$ F:=  -i \sum_{k_i:\; |\{k_1,k_2,k_3,k_4\}\cap S^c|\leq 2 \atop k_1-k_2+k_3-k_4=0\; |k_1|^2-|k_2|^2+|k_3|^2-|k_4|^2\neq 0 } \frac{u_{k_1}\bar u_{k_2}u_{k_3}\bar u_{k_4}}{|k_1|^2-|k_2|^2+|k_3|^2-|k_4|^2}. $$ This normal form transformation (taken from \cite{GYX}) does not change the results in any way (it only changes in a trivial manner the definitions of $P^{(3)}$ and $P^{(6)}$) and may simplify the study of  the {\em T\"oplitz--Lipschitz} properties. 
\end{remark}
{\bf Conservation laws.} 
{\em Momentum:} The conservation of $M$ in the new variables implies that the monomials appearing in $H$ are of the form
\begin{equation}\label{cons1}
z^\alpha \bar z^\beta y^c  e^{i(x,\nu)},\quad \sum_iv_i\nu_i+\sum_{k\in S^c}(\alpha_k-\beta_k)k=0
\end{equation}   where $\nu=(\nu_1,\ldots,\nu_m),\ \nu_i\in\mathbb Z$ and $\alpha,\beta$ are multi-indices in $\mathbb N$.

\noindent{\em Parity:} In the new variables a Hamiltonian is {\em even} if $\sum_i \nu_i$ is even when the total degree in $w$ is even and odd otherwise.

\begin{constraint}\label{co2}
We may further choose the $v_i$ so that $\sum_i\nu_iv_i\neq 0$ when $\sum_i|\nu_i|<10,\ \nu\neq 0.$
\end{constraint} This will imply    that $P^{(6)}$ at $z=0$ does not contain any term of degree  6 or 8  in the elements $u_{v_i}$ and non constant in $x$.
\subsection{  Analytic functions and weight decomposition\label{afwd}}

We want to assign {\em weights} to all the variables, which keep track of their original definition.

Thus we give weight 2 to the variables $y ,$ weight $2>\alpha>0$ to $\xi$  and weight 1 to the $w$. 

We work in the domain $$ A_\alpha\times D(s,r):= $$ $$
 \{  \xi\,:\  \frac12 r^\alpha\le |\xi|\le r^\alpha\,\}\times \{   x,y,w\,:\   x\in \T^m_s\,,\  |y|\le r^2\,,\  |w|_{a,p}\le r\}
$$ $$\subset  \R^m\times \T^m_s\times\C^m\times \bf{\ell}^{(a,p)}.$$  Here $0<\alpha < 2,0<r<1,0<s$ are parameters. $\T^m_s$ has been defined in \S \ref{Eaav}. 

We denote by   $ \bf{\ell}^{(a,p)}$   the subspace of $\bf{\bar \ell}^{(a,p)}\times\bf{\bar \ell}^{(a,p)} $  generated by the indices in $S^c$ and  $w=(z,\bar z)\equiv (z^+,z^-)$ are the corresponding coordinates.
\begin{remark}\label{seba}
We shall write $z_k^\sigma, \sigma=\pm$ when we do not want to specify if we are using $z_k$ or $\bar z_k$.
\end{remark}
\begin{remark}
If $r^\alpha<\epsilon$ the domain $A_\alpha\times D(s,r)$ is contained in $B_\epsilon$ so that the Hamiltonian is well defined and analytic.
\end{remark}

Under this change of variables the total degree of a monomial is preserved, provided we give to $y$ degree (or weight)  2. In the estimates we also  give  to $\xi$ weight $\alpha$. This implies (by Formula \eqref{domain}) that 
$ |\xi^a y ^i  w^j|_{s,r,\C} \leq  r^d$ where the degree $d$ equals the weight $   a\alpha +2i+j$ of the monomial (notice that $a$ can be a half--integer).

\bigskip

Given  a Banach space $E$  
  we consider  analytic functions $F: D(s,r)\to E$.
By definition   $F$ is analytic if its Taylor-Fourier series in $x; y,w$ is totally convergent.
   We choose  by definition, as  norm $| F|_{s,r},$ of an analytic function on $D(s,r):$  
\begin{equation}\label{essr}| F|_{s,r,E}:= \sup_{A_\alpha\times D(s,r)} ||F (\xi,x;y,w)||_E. 
\end{equation} 
    \subsection{ Weight decomposition of formal polynomials}
 
 We call $V^\infty$ the space of formal infinite polynomials in $y,w$ with coefficients in $L^2(\T^m,\C)$ and $F^\infty$ the subspace formed by the functions which satisfy \ref{cons1}.   Here we use the weight only in these variables and keep the $\xi$ as parameters.\smallskip
 
 We call $V^{(i,j)}$ the subspace of functions of degree $i$ in $y$ and $j$ in $w$ (hence of weight $2i+j$). We call $V^a$ the subspace of functions of weight $a$ and $V^{>a}$ (resp. $V^{<a}$) the subspace of functions of weight $>a$ (resp. $<a$):
 $$    V^\infty= \oplus_a V^a\,,\quad V^a= \oplus_{i,j:\, 2i+j=a} V^{i,j} \,,\quad V^{>a}= \oplus_{b>a} V^b.$$
 We use the spaces  of weight $\leq 2$:
\begin{itemize}\item $V^{0}(x):= L^2(\T^m,\C)$,\item  $ V^{1,0}(x)$ the space of elements $\sum_{i=1}^m f_i (x,\xi)y_i$, thought of  as  $m$ dimensional column vectors with entries in $ L^2(\T^m,\C)$ (and depending on parameters $\xi$), \item $ V^{0,1}(x)$ the space of linear forms $\sum_{i\in \mathbb Z^n}(f_i (x,\xi)z_i+e_i (x,\xi)\bar z_i)$, thought of as $\infty$-dimensional column vectors with entries in $ L^2(\T^m,\C)$ (and parameters $\xi$) and finally\item  $V^{0,2} $ the space of  quadratic forms  $$\frac{1}{2}\big(\sum_{i,j\in S^c} q_{(i,+),(j,+)} z_i z_j+ q_{ (i,+),(j,-)}z_i \bar z_j+ q_{(i,-),(j,+)}\bar z_i  z_j+ q_{ (i,-),(j,-)}\bar z_i \bar z_j\big) .$$ 
\end{itemize}
It will be convenient to represent   the elements $ Q(w)\in V^{0,2} $ as  associated to (symplectic) matrices,   where
\begin{equation}\label{mali}Q_M(w):= -  \frac 12 (w,MJ w)
 \end{equation}
and $M$ has entries $M_{(k,\sigma),(h,\tau)}= \tau q_{(k,\sigma),(h,-\tau)}$.
  
\begin{remark}\label{sym}
By definition $MJ$ is symmetric,
if moreover  $Q_M(w)$ is real we have that    $M\Sigma$ is self-adjoint, where$$ \quad
  \Sigma= \begin{vmatrix} -1 &0 \\ 0 &1 \end{vmatrix} .$$
\end{remark}
\begin{remark}
Notice that  we have defined $V^\infty$ as complex functions. However our Hamiltonian is real, and clearly $V^\infty$ contains a subspace of real functions. Indeed all our arguments and symplectic changes of variables will be real and hence preserve the real subspace.
\end{remark}
\medskip
 
The dense subspace $V^{0,2}_{an}$ of $V^{0,2}$ formed by analytic functions is closed under Poisson bracket and   we think of it as a smooth form of the Lie algebra of the infinite symplectic group.     The space $V^{0,1}_{an}$  is then the standard symplectic space, under Poisson bracket, over which $V^{0,2}_{an}$ acts. 
In particular we have
\begin{proposition} \label{pois} If $f,g\in V^{0,1}_{an}$ and $A,B$ are two elements of the Lie algebra of the symplectic group of $V^{0,1}_{an}$
\begin{equation}
\label{pam}\{f,g\}=i\,f^tJg,\quad \{Q_A(w),f\}=iAf,\quad \{Q_A(w),Q_B(w)\}=i\, Q_{ [A,B]}.
\end{equation} \end{proposition}

In particular the evolution of $w$, defined by $Q_A$  is
\begin{equation}
\label{Havf1}\dot w=i  wA
\end{equation}
We now require conservation of  momentum and parity.

In order to stress this
\begin{definition}\label{glieffe}  We denote the subspaces of  $V^0$, $V^{(1,0)}$, $V^{(0,1)},V^{(0,2)} $ which satisfy conservation of  momentum  and parity by $F^0, F^{(1,0)}$, $F^{(0,1)},F^{(0,2)}.$ The direct sum of these spaces we denote by $F^{\leq 2}$. In general we denote by $F^{(i,j)} $ the subspace of $V^{(i,j)}$ which satisfies \eqref{cons1} and parity.
\end{definition}
\begin{remark}\label{LABASE}
i)\ $F^{(0,1)}$ has as basis the elements\begin{equation}\label{LABASE1}
(\nu,+)\to e^{\ii \sum_j\nu_jx_j}z_k ,\quad (\nu,-)\to e^{-\ii \sum_j\nu_jx_j}\bar z_k;$$ $$ \sum_j\nu_jv_j+k=0\,,\quad \sum_i\nu_i= {\rm odd}.
\end{equation}

ii) \ In the same way $F^{(0,2)}$ has as basis the products of elements of $F^{0,1}$. That is we have a surjective linear map $b:F^{0,1}\otimes F^{0,1}\to F^{0,2}.$

Under this map    $b[(\nu,\sigma)\times(\mu,\tau)]=b[ (\nu',\sigma')\times(\mu',\tau')]$ if  and only if $\sigma\pi(\nu)=\sigma'\pi(\nu')$, $\tau\pi(\mu)=\tau'\pi(\mu')$ and $\sigma\nu+\tau\mu=\sigma'\nu'+\tau'\mu'$.  The image of these product  form a basis of $F^{0,2}$.  Explicitly  (cf. \ref{seba}) the elements  $$e^{ \ii \sum_j\nu_jx_j}  z_k^\sigma z_h^\tau,\ \sum_j\nu_jv_j+\sigma k+\tau h=0.$$

\end{remark}
\subsection{Quadratic normal forms} In the space $F^{\leq 2}$ we will be particularly interested in the subspace of ``normal forms'' i.e. the Hamiltonians of the form
$$N:=(\ome,y) -\frac12 w Q(x) J w^t, $$
 for some choice of  the frequency $\ome\in \R^m $ and of the matrix $Q(x)$,  both  possibly depending on the parameters $\xi$.
\smallskip 
 
The core of the KAM algorithm is to study the action of the operator $ad(N)$ on $F^{\leq 2}$.
\begin{definition}\label{lhomol}
Given $R\in F^{\leq 2}$,  the equation
$ad(N ) F = R $ for $F\in F^{\leq 2}$ is called {\em the homological equation}. 
\end{definition}
\begin{lemma}\label{homol} The operator  $ad(N):=x\mapsto \{N,x\}$  acting on $F^{\leq 2}$ can be represented as the block   matrix:
\begin{equation}\label{rep1}\begin{pmatrix} \ome \cdot \partial_x & 0 &0&0 \\ 0 & \ome \cdot \partial_x & 0 &0 \\ 0 &0&\ome \cdot \partial_x+i Q &0 \\ 0 &  i\nabla_x Q  & 0 & \ome \cdot \partial_x+i[Q, \cdot]    \end{pmatrix}, \end{equation} .
 \end{lemma}Here, by abuse of notations,  $ \ome \cdot \partial_x =\sum_{i=1}^m\omega_i\pd{}{x_i}$, is the operator on the entries of the vector or matrix, i.e. the {\em scalar}  operator times the corresponding identity matrix.  
We have used the basis   \ref{LABASE1} for $F^{0,1}$ and the matrix representation for $F^{0,2}$. \begin{lemma}\label{homol2}
If for $F^{0,2}$ we use as basis the products of the elements of the basis  in $F^{0,1}$, then by the Leibniz rule then  $ad(N)$ on $F^{0,2}$ is induced, under the map $b: F^{0,1}\otimes F^{0,1}\to F^{0,2}$  (cf. Remark \ref{LABASE}) by the matrix 
\begin{equation}
\label{homol1}ad(N)\vert_{F^{0,1}} \otimes I + I \otimes  ad(N)\vert_{F^{0,1}}.
\end{equation}
\end{lemma}
\begin{proof}[Proof of \ref{homol}, \ref{homol2}] We just apply the rules of Poisson brackets discussed in Proposition \ref{pois}.
\end{proof}
\subsection{Final form for the Hamiltonian}

By definition an analytic function on $D(s,r)$ can be Taylor expanded in $y,w$  to obtain an element of $V^\infty$; given $F\in V^\infty$ we will denote by $F^{(i,j)}$ the projection of $F$ on $V^{(i,j)}$, same for all the other subspaces.

\bigskip

   For $r$ small enough, $H:  D(s,r)\to \R$  is analytic, so it is an element of $V^\infty$.     We obtain
   a formal polynomial  whose monomials are of the form $ m(x)\xi^a y^i w^j $ where  $a$ can have half--integer coordinates (this corresponds to a term in $u$ of degree $2(i+a)+j$).  

 We drop in formula \eqref{Ham4}  the constant part (depending only on the parameters $\xi$)  and  separate $H= N+P$  where  $N = H_0^{\leq 2}\in V^2$   is a ``normal form"   and $P= H_0^{> 2}+ P^{(3)}+ P^{(6)}$ is {\em small} with respect to $N$. 
 
 We obtain, with the notation of Formula \eqref{mali}  \begin{equation}\label{hama}N:= (\ome(\xi),y)+ \sum_k \Ome_k  |z_k|^2 +  Q_M(w):=D+   Q_M(w) ,\end{equation}
where $ D:=(\ome(\xi),y)+ \sum_k \Ome_k(\xi)  |z_k|^2  $  and  \begin{equation}\label{gliome}\ome_i(\xi):  = |v_i|^2 -2\xi_i + 4 \sum_{j } \xi_j , \quad \Ome_k(\xi)= |k|^2+4\sum_i \xi_i .
\end{equation} By $( \cdot,\cdot )$ we denote the real scalar product. Finally the quadratic form is \begin{equation}
\label{quafo}Q_M(w)= 4\sum^*_{  1\leq i\neq j\leq m    \atop h,k \in S^c}\sqrt{\xi_{i}\xi_{j}}e^{\ii(x_{i}-x_{j})}z_{h}\bar z_{k } + 
\end{equation} $$2 \sum^{**}_{ 1\leq i< j\leq m   \atop h,k \in S^c }\sqrt{\xi_{i}\xi_{j}}e^{-\ii(x_{i}+x_{j})}z_{h} z_{k } + 
 2\sum^{**}_{ 1 \leq i<j\leq m    \atop h,k \in S^c }\sqrt{\xi_{i}\xi_{j}}e^{\ii(x_{i}+x_{j})}\bar z_{h}\bar  z_{k }. $$ 
 \begin{remark}
The separation of the Hamiltonian in two parts is  justified by the following two facts.
\begin{itemize}
\item $Q_M(w)$ commutes with the variables $x$.
\item  $D$  has a natural {\em diagonal form}.
It can be useful to identify the term $ \sum_k \Ome_k(\xi)  |z_k|^2 $ with its symplectic representation  $\Ome=$ diag$(\{\Ome_k\}_{k\in S^c},-\{\Ome_k\}_{k\in S^c}) $ 
\end{itemize}
\end{remark}
\begin{remark}\label{twist} We have grouped in the perturbation $P$ all the terms of $H_0^{> 2}$, this is just a convention since such terms are explicit. In particular $H_0$ contains a quadratic part in $y$ of the form $(y, A y)/2$  where $A$ is defined in
Theorem \ref{teo1}. 
\end{remark}
\begin{lemma} \label{stime}Fixing $\lambda= r^\alpha/\max(|v_i|^2)$ the perturbation $P$ satisfies the following bounds:
$$\Vert \Pi_{F^{0}} P\Vert_{s,r}^\lambda \leq Cr^{5\alpha}\,,\quad  \Vert \Pi_{F^{1,0}} P\Vert_{s,r}^\lambda \leq Cr^{2+4\alpha}\,,\quad \Vert\Pi_{F^{0,1}} P\Vert_{s,r}^\lambda \leq Cr^{1+5\alpha/2}\,,$$$$\Vert\Pi_{F^{0,2}} P\Vert_{s,r}^\lambda \leq Cr^{2+2\alpha}\,, \quad \Vert\Pi_{V^{>2}}P\Vert_{s,r}^\lambda\leq Cr^{3+\alpha/2}.$$
\end{lemma}
\begin{proof} All the bounds are purely dimensional, notice only that we have used constrain \ref{co2} to impose that the first non--zero contribution to $F^0$ is from a polynomial of degree $10$ and not $6$.
\end{proof}
\vskip10pt  
To further simplify the Hamiltonian we assume:
\begin{constraint}\label{co1b} Given any four different elements $h_1,h_2,h_3,h_4$ in $S$ one has 
\begin{itemize}
\item $ h_1\pm h_2 \neq h_3\pm h_4$.
\item $(h_1+h_2-h_3-h_4,h_3-h_4)\neq 0$.
\end{itemize}
\end{constraint}
As for the   the matrix $M(x)$ we have

\begin{lemma}\label{matri} Given $h, k\in S^c$ there exist at most one  couple $v_i\neq v_j\in S$ such that one and only one of the next two properties holds: \begin{enumerate}[(1)]
\item  
$ h -k = v_j- v_i$ \,, $|h|^2-|k|^2 =|v_j|^2- |v_i|^2$. 

\noindent In this case  one has
$M_{(h,\sigma),(k,\sigma)}=\bar M_{(k,\sigma),(h,\sigma)}=  4\sigma\sqrt{\xi_i \xi_j} e^{\ii\sigma (x_i-x_j)}$ and $M_{(h,\sigma),(h,-\sigma)}=0$.
\smallskip

\item  $i<j$, $ h +k = v_i+ v_j$ \,, $|h|^2+|k|^2 =|v_i|^2+ |v_j|^2$.

\noindent  In this case  one has
$M_{(h,-\sigma),(k, \sigma)}= - \bar M_{(k,\sigma),(h,-\sigma)}  =  4\sigma\sqrt{\xi_i \xi_j} e^{ \ii \sigma (x_i+x_j)}$ and $M_{(h,\sigma),(k,\sigma)}=0$.\medskip

\item If there exists no couple $v_i\neq v_j\in S$ satisfying either (1) or (2) then $M_{(h,\pm),(k,\pm)}=0$ (this includes naturally the case $h=k$).
\end{enumerate} 
\end{lemma}
\begin{proof}  With the notations of Formula \eqref{quafo} constraint \ref{decs} now means that   $\sum^*$ denotes the constraint given in Item $(1)$   and $\sum^{**}$, the constraint  given in Item $(2)$.

By Formula \ref{ffpb} we have $$\{4 \sqrt{\xi_{i}\xi_{j}}e^{\ii(x_{i}-x_{j})}z_{h}\bar z_{k },z_k\}=4i \sqrt{\xi_{i}\xi_{j}}e^{\ii(x_{i}-x_{j})}z_{h}  .$$ The terms $2\sqrt{\xi_{i}\xi_{j}}e^{-\ii(x_{i}+x_{j})}z_{h} z_{k } $  repeat twice and $$\{\sqrt{\xi_{i}\xi_{j}}e^{-\ii(x_{i}+x_{j})}z_{h} z_{k } ,\bar z_k\}=-i\sqrt{\xi_{i}\xi_{j}}e^{-\ii(x_{i}+x_{j})}z_{h} $$  while   
$$\{\sqrt{\xi_{i}\xi_{j}}e^{\ii(x_{i}+x_{j})}\bar z_{h}\bar  z_{k },z_k\}=i \sqrt{\xi_{i}\xi_{j}}e^{\ii(x_{i}+x_{j})}\bar z_{h}.$$

   Item (3) is trivial by the definitions. The fact that the two different conditions cannot hold contemporarily with the same $v_i,v_j$ is again trivial. We only need to prove that (1) and (2) cannot hold  for two different pairs.
 Suppose now that item (1) holds for some $h,k$ and  for two different pairs $\{v_i, v_j \},\{v_l,v_m\}$ then
$v_i-\nobreak v_j=  v_l-v_m$ contrary to hypothesis \ref{co1b}. The same if item (2) holds for some $h,k$ and  for two different pairs $\{v_i, v_j \},\{v_l,v_m\}$. Suppose now that  for some $h,k$  item (1) holds with $\{v_i, v_j \}$ and  item (2) holds with $\{v_l,v_m\}$, then by substituting we contradict the second item of \ref{co1b}.
\end{proof}

\section{  Graph representation\label{Gra}} It will be convenient to associate to $ad(N)$ or equivalently to the matrix $M$  (associated to \eqref{quafo} through Formula \eqref{pam})  two graphs    $
\Lambda_S,\Gamma_S$   encoding the information of its non--zero off diagonal entries. In fact in  $ad(N)$  the part associated to $D$ is diagonal.

The two graphs arise from the following complementary points of view.

\bigskip

\subsection{Geometric graph $\Gamma_S$} 
The geometric point of view is to consider  the action of $ad(N)$ on the space $V^{(0,1)}$. In fact by a simple inspection from Lemma \ref{matri} we have \begin{proposition} The operator $ad(N)$ preserves 
  the subspace  $V^{(0,1)}_f$ of $V^{(0,1)}$ of finite linear combinations of the elements  $z_k,\bar z_k,\ k\in S^c$ with coefficients in the algebra $A$ of finite Fourier series in the variables $x$. 
\end{proposition}   We shall refer to $z_k,\bar z_k,\ k\in S^c$  as {\em the geometric basis}  (of $V^{(0,1)}_f$  as a free module over the algebra $A$).  \footnote{When the coefficients in a basis are in an algebra and not a field it is customary to use the word {\em module} and not vector space. }
 
   For the geometric graph we shall initially forget the difference between  $z_k,\bar z_k$ and remember only the vector $k$.
\begin{definition}
The graph $\Gamma_S$  has vertices  $\mathbb Z^n$  and edges corresponding to non zero entries of $M$.
\end{definition} 

 In order to keep track of complex conjugation it is convenient to {\em decorate}  the edges of the graph with two colors, black and red according to the rules that we presently explain, and a {\em marking or label}.

1.  If item (1) in \ref{matri} holds  we connect 
$h,k$ with a black edge  oriented from $k$ to $h$ and labeled by $v_j-v_i$;  

2.  If item (2) in \ref{matri} holds  we connect 
$h,k$ with  a red  non-oriented edge  labeled by $v_j+v_i$;

3. If item (3) holds  we do  not connect $h$ and $k$.

These rules are purely geometric and can in fact be applied to all vectors in $\R^n$ or even $\C^n$ and not just integral vectors.  They define thus a {\em colored geometric graph} with vertices  points in space and edges given by the previous rules.  To be specific:

 \begin{definition}\label{lasfe}  Given two vectors $v_h,v_k$ we set  $S_{h,k}=S_{k,h }$ to be  the sphere  having as one of its diameters $v_h,v_k$.  It has the equation
\begin{equation}\label{sfera}(x-v_h,x-v_k)= 0,\quad \text{or}\ (x,v_h+v_k)= ( v_h, v_k)+ (x)^2. 
\end{equation}
We also define the hyperplane 
$$H_{h,k}:=\{x\,|\,(x,v_k-v_k)=(v_k,v_k-v_h)\}.$$ \end{definition}
\begin{remark}
The hyperplanes $H_{h,k}$ and $H_{k,h}$ are   parallel   and
$$H_{h,k}=v_k-v_h+H_{k,h}.$$
\end{remark} 

Observe that    by definition $v_h\in H_{h,k}$  but  if $v_h\in H_{i,j}, i\neq h$ we must have $(v_h,v_j-v_i)=(v_i,v_j-v_i)$.  By Constraint 1. this is not    satisfied, if $i\neq h$.

In other words, if $i\neq h$ then $v_i\notin H_{h,k}$. \smallskip

\begin{figure}[!ht]
\centering
\begin{minipage}[b]{11cm}
\centering
{
\psfrag{a}{$v_k$}
\psfrag{b}{$v_h$}
\psfrag{c}{$ a_2$}
\psfrag{d}{$ b_2$}
\psfrag{e}{$ a_1$}
\psfrag{f}{$ b_1$}
\psfrag{H}{$H_{h,k}$}
\psfrag{S}{$S_{h,k}$}
\psfrag{m}{$ v_k-v_h$}
\psfrag{l}{$ v_k+v_h$}
\includegraphics[width=11cm]{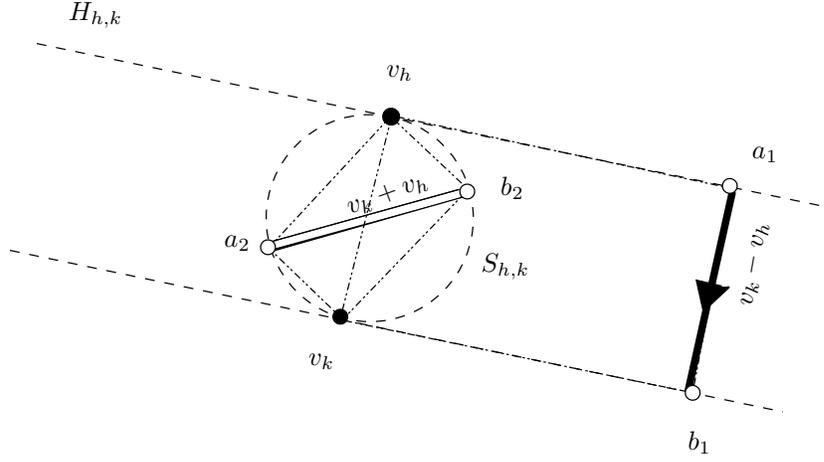}
}
\caption{\footnotesize{The plane $H_{h,k}$ and the sphere $S_{h,k}$. The points $a_1,b_1,v_k,v_h$ form  the vertices of a rectangle. Same for  the points $a_2 ,v_h,b_2,v_k$}}\label{fig1}
\end{minipage}
\end{figure}

We construct now the {\em colored graph} $\Gamma_S$ with vertices in $\mathbb R^n$.  Consider two points $a,b\in\mathbb R^n$ (sometimes even in $\C^n$).
 
If there exists a pair $h,k,\ h\neq k$ so that  $a\in H_{h,k}, \ b=a+v_k-v_h\in H_{k,h}$  we join the two points by a {\em black  } edge oriented form $a$ to $b$ and marked $v_k-v_h$. See the points $\{a_1,b_1\}$ in Figure \ref{fig1}.
In other words:
 \begin{equation}
\label{colo}\{a,b\}\ \Bigg|\begin{cases}
|b_1|^2-|a|^2=|v_k|^2-|v_h|^2\\ b-a=v_k-v_h
\end{cases}.
\end{equation} 
or equivalently $(a,v_h-v_k)=(v_h,v_h-v_k)$. \medskip

If there exists a pair $h,k,\ h\neq k$ so that  $a\in S_{h,k}, \ b=-a+v_k+v_h\in S_{k,h}=S_{h,k}$  we join the two points by a {\em red edge} marked by $v_h+v_k$. See the points $\{a_2,b_2\}$ in Figure \ref{fig1} (we represent red edges by a double line).  \begin{equation}
\label{colo1}\{a,b\}\ \Bigg|\begin{cases}
|a|^2+|b|^2=|v_h|^2+|v_k|^2\\ a+b=v_h+v_k
\end{cases}.
\end{equation} 
In other words $(a-v_h,a-v_k)=(b-v_h,b-v_k)=0$.

The points  $a ,v_h,b,v_k$ form the vertices of a rectangle.   In other words, \begin{lemma}
\label{oppo}$a,b$ are opposite points in the sphere $S_{h,k}$   having as one of its diameters $v_h,v_k$.
\end{lemma} 

\begin{remark}
For each pair $h,k$ there are finitely many   points  in  $  \mathbb Z^n$, in the sphere  $S_{h,k}$ . Therefore there are only finitely many red edges with integral vertices.
\end{remark} 
\begin{definition}
We construct the graph  $\Gamma_S$  with vertices all the points of $\mathbb R^n$ and  edges the black and red edges described.
\end{definition}
 We want to understand the connected components of the graph  $\Gamma_S$.\smallskip

By Constraint \ref{co1} if  $i\neq h, i\neq k$ then $v_i\notin S_{h,k}$.
We immediately have.
\begin{lemma}\label{spco}
The vectors $v_1,\ldots,v_m$ are a component of the graph  $\Gamma_S$. In this component every two vertices are joined by a red and by a black edge.
\end{lemma}
\begin{proof}
If $v_i$ is joined to another vector $u$ by a black edge of type $H_{h,k}$ we must have  by constraint  \ref{co1} that $h=i$ and $u=v_k$. Similarly for a red edge of type $S_{h,k}$ using constraint \ref{co1}  we have $i=h$ or $i=k$.
\end{proof}
\begin{definition}
The component $v_1,\ldots,v_m$ is called the {\em special component} of the graph  $\Gamma_S$. 
\end{definition}

\subsection{The combinatorial graph $\Lambda_S$}\label{ilgrco}
Consider the space $F^{(0,1)}$  with its basis over $\C$ given by Remark \ref{LABASE}.

This basis is really indexed by $\mathbb Z^m\times\Z/(2)$, that is an integral vector plus a sign, we shall refer to it as {\em the frequency basis}.

 From Lemma \ref{matri} the  linear operator $ad(N)$ has the property that it  transforms every element of this basis  into a finite linear combination  of the same basis. 
\begin{definition}\label{lambdas}
The graph $\Lambda_S$  has vertices  $\mathbb Z^m\times\Z/(2)$  and edges corresponding to non zero entries of $ad(N)$.
\end{definition} 
\subsection{Abstract colored marked graphs.}
It will be useful to also use completely abstract graphs defined as follows
\begin{definition}\label{acmg}
A {\em abstract colored marked graph}  or {\em $\mathcal M$--graph}  for short, is
\begin{itemize}
\item A connected graph $\Gamma$ (without  repeated edges).
\item  A color {\em red or black} on each edge, displayed 
$$ \xymatrix{&&a\ar@{-}[r]^{black} &b&&&& c \ar@{=}[r]^{red} &d }$$
\item A marking $(i,j),\  1\leq i\leq m,  1\leq j\leq m, i\neq j$ on each oriented edge with the convention that the opposite orientation corresponds to the exchanged marking $(j,i)$.
\end{itemize}
\end{definition}
A {\em geometric realization} of the graph  $\Gamma$ is a graph isomorphism with a connected component of $\Gamma_S$ such that each black edge of $\Gamma$ marked $(i,j)$ corresponds to a  black edge of $A$ marked $v_j-v_i$. In the same way to each red edge of $\Gamma$ marked $(i,j)$ corresponds to a red edge of $A$ marked $v_i+v_j$.
\subsection{Summary of results}
In term of the frequency basis, denote by $e_i$ the basis of $\Z^m$ and consider the map 
\begin{equation}\label{ILPI}
\pi:\Z^m \to\Z^n,\qquad \pi(\nu_1,\ldots,\nu_m ):= \sum_{i=1}^m\nu_iv_i.
\end{equation}
If   $k=-\pi(\mu )  $, the vector $ e^{\ii \mu.x}z_k$ lies in $ F^{(0,1)}$. The vectors  $z_k \in V^{(0,1)},\quad e^{\ii \mu.x}z_k $  map  under $ad(N)$  to a linear combination in which respectively: \begin{enumerate}
\item If $h,k$ are connected with a black edge  oriented from $k$ to $h$ and labeled by $v_j-v_i$:
\begin{itemize}\item $ h=-\pi(\mu+e_i-e_j )=k+v_j-v_i$
\item In the geometric basis, $z_h$ has coefficient $4  \sqrt{\xi_i \xi_j} e^{\ii  (x_i-x_j)}$ in $M(z_k).$
\item In the frequency basis,  $e^{\ii (\mu+e_i-e_j).x}z_h$  has coefficient $4 i\sqrt{\xi_i \xi_j}  $ in $ad(N)(e^{\ii (\mu .x)}z_k).$
\end{itemize}       
\item If $h,k$ are connected with   a red  non-oriented edge  labeled by $v_j+v_i$
\begin{itemize} \item $    h= \pi(\mu+e_i+e_j) =-k+v_i+v_j$
\item In the geometric basis,  $\bar z_h$ has coefficient $4  \sqrt{\xi_i \xi_j} e^{\ii  (x_i+x_j)}$  in $M(z_k).$
\item In the frequency basis, $e^{\ii (\mu+e_i+e_j).x}\bar z_h$  has coefficient $4i \sqrt{\xi_i \xi_j}  $  in $ad(N)(e^{\ii (\mu .x)} z_k).$ 
\end{itemize}   

\end{enumerate}   
 \begin{remark}\label{iprima}
Under our convention on the indexing of the basis, $(\mu,+)$ corresponds to $e^{\ii (\mu  .x)}z_h$ while  $(-\mu-e_i-e_j,-)$ corresponds to $e^{\ii (\mu+e_i+e_j).x}\bar z_h$.  In the frequency graph $\Lambda_S$ therefore, the elements $\mu,\mu+e_i-e_j$ are joined by an edge in case i). In case ii) if $\nu=-\mu-e_i-e_j$, that is  if $\mu+\nu+e_i+e_j=0$, we have that $(\mu,+), (\nu,-)$  are connected by an edge marked $e_i+e_j$.
\end{remark}

We will show that, provided we choose the vectors $v_i$  {\em generically},  we  shall have several essential properties for these graphs  which we will  need in order to  and prove the reducibility Theorem \ref{teo1}, and study the Homological equation. 

 The  {\em generic assumption}  will be expressed by the fact that the coordinates of the vectors  $v_i$ do not satisfy some polynomial equation (a product of several equations which will be constructed in the course of the proof).  That is, we think of $(v_1,\ldots,v_m)\in \mathbb R^{nm}$  and will impose that this point does not lie in a  certain algebraic hypersurface whose equation will be at least implicitly given (in term of certain graphs).    The precise statements are contained in \S \ref{therin}.\smallskip 

We now discuss the properties  deduced from the generic assumption.

The first relates the two graphs  $\Lambda_S,\Gamma_S$.  In fact, take a frequency $\mu$, and let $\A$ be the associated component in $\Lambda_S$. Set $k=-\pi(\mu)$ and $A$ be the associated component   in $\Gamma_S$. We shall see  \begin{theorem}\label{oneone}The map $-\pi$ establishes a graph isomorphism between $\A$ and $A$, compatible with the markings.

Hence the space  spanned by all transforms of  $e^{\ii \mu.x}z_k$   applying the operator $  ad(M)$ has a basis extracted from the frequency basis in  correspondence, under $\pi$,   with  the vertices of  $A$.

Same statement for its conjugate generated by  $e^{-\ii \mu.x}\bar z_k$. 
 \end{theorem}
\begin{remark}
We shall often refer to such a subspace as {\em a block}, for $ad(N)$.  For generic $S$  the  symplectic form restricted to this block is identically 0 and if we add a block with its conjugate we have a non degenerate symplectic space decomposed as sum of two $ad(N)$--stable Lagrangian subspaces.
\end{remark}  This is thus  a block for $ \, ad(M)$ for the graph $\Lambda_S$,   and all other blocks  in the set of elements $\nu\,|\,-\pi(\nu)\in A$  are obtained from this block by multiplying with all the elements $\nu$ such that $\pi(\nu)=0$.

The entire space $F^{(0,1)}$ therefore decomposes into free submodules under the algebra $C$ of finite Fourier series in $e^{\ii \nu.x} \ |  \pi(\nu)=0$ corresponding to all the geometric blocks in $\Gamma_S$. For each such block $A$ a basis of the corresponding space over $C$ is obtained as follows.   We  choose a specific element $r\in A$, {\em a root} and then a specific $\mu$ with $\pi(\mu)=-r$. From $\mu$, and applying $ad(N)$,  we construct   the basis for the submodule in correspondence with the vertices of $A$.   \smallskip

 Let us summarize the most important properties that we shall prove in \S \ref{ilG} and \S \ref{Magt}.   \begin{theorem}\label{sunto}
\begin{enumerate} \item All connected components of the graph $\Gamma_S$   have as vertices points which are affinely independent hence at most $n+1$ vertices.
\item There are finitely many components containing red edges.  

\item The connected components of  $\Gamma_S$  consisting only of black edges  are divided into a finite number of families.

Each family is indexed  by an abstract marked graph  with $k\leq n$  edges, and it  depends on the elements of  a  $n-k$  dimensional sub-lattice.
\end{enumerate}
\end{theorem}
Given a connected component $A$ with $k+1$ vertices ($k\leq n$) of $\Gamma_S$ we   fix a   vertex $x(A)\in A$ which we call  {\em the root}. Then:
\begin{corollary}\label{sunto1} For every other vertex $x_a$ (with $a=1,\dots,k)$ one has two functions on $A$:
\begin{equation}
\label{alcc}\sigma(a)=\pm 1\,,\   L(a)=\sum A^{(i)}_a e_i\,,\ A^{(i)}_a\in\Z;
\end{equation}$$ |L(a)|<n,\; \sum_iA^{(i)}_a=1-\sigma(a)\,, $$ 
such that $\sigma(a)$ is $1$ if the path from $x(A)$ to $x_a$ has an even number of red edges, $-1$ otherwise. We have from Formulas \eqref{bacos0} and \eqref{bacos}
\begin{equation}\label{iubl}x_a= \sigma(a)x(A) + \pi(L(a))= \sigma(a)x(A)+\sum A^{(i)}_a v_i \,,\end{equation}$$  |x_a|^2 = \sigma(a)|x(A)|^2 + \sum A^{(i)}_a |v_i|^2  \,.$$

The matrix $M$ is block diagonal with two blocks (denoted by $A,\pm$) in correspondence with each connected components $A$ of $\Gamma_S$. 

The  matrix 
$M$ restricted to $A,+$ is denoted by $M_{A,+}$ and given by $$M_{a,b}:= M_{(x(a),\sigma(a)),(x(b),\sigma(b))}= 0 $$
 if $(a,b)$ is not an edge (in particular it is zero on the diagonal).
$$M_{a,b}:= M_{(x(a),\sigma(a)),(x(b),\sigma(b))}=4\sigma(b)\sqrt{\xi_i\xi_j}e^{\ii (\sigma(b)L(b)-\sigma(a)L(a)).x}$$  if $(a,b)=e$ is  an edge marked $(i,j)$.

$M_{A,-}=-\overline{ M}_{A,+}$ is minus the conjugate of $M_{A,+}$.
\end{corollary}
\begin{proof}
  We fix a root $x$ on $A$ and a sign $\sigma(x)=+$. We associate to each vertex in $A$ a corresponding vertex in $\A$, by Theorem \ref{oneone}. This associates to each $x_a$ a sign $\sigma(a)$.  We use the relations \eqref{colo} and \eqref{colo1} to compute the $A_i^{(a)}$ by choosing a path from $x$ to each vertex $x_a$. Since $\pi$ is a graph isomorphism,  compatible with the two markings, it  defines $L(a)$ uniquely. 
   Notice that if say $\sigma(b)=+$  then  $\sigma(b)L(b)-\sigma(a)L(a)= e_j-e_i $ if the edge is black and $\sigma(b)L(b)-\sigma(a)L(a)= e_j+e_i $ otherwise.
To prove the second statement we implement the matrix rules of Lemma \ref{matri}.
\end{proof}
 
 \section{Proof of Theorems \ref{teo1}, \ref{teo1b} and \ref{teo2} \label{reinf}}

\subsection{Theorem   \ref{teo1}}We will prove Theorem \ref{teo1} by exploiting the block structure discussed in Corollary \ref{sunto1}. 
 For all  $k\in S^c$ set $x(k):=x(A)$ to be the root of the component $A$  of $\Gamma_S$ to which it belongs. Set $ L(k)=0$ if $k=x(A)$ is the root (this includes the connected components made of one point). Otherwise   $k= x_b$ for some index $b=b(k)$.  We then  set $L(k):=L(b)$ and $\sigma(k)=\sigma_b$(cf. \ref{sunto1}).
 Theorem \ref{teo1} is contained in the following, more precise, propositions:
  \begin{proposition}\label{reteo} i)\  The equations 
\begin{equation}\label{labella}
z_k= e^{\ii L(k).x}z_k',\ y=y'-\sum_{k\in S^c} L(k)|z_k'|^2,\ x=x'.
\end{equation}
define a symplectic change  of variables $D(s,r/2)\to D(s,r)$, which preserves the spaces $V^{i,j}$.\smallskip

{\em We denote by  $X=$ diag$(\{e^{\ii L(k).x}\}_{k\in S^c},\{ e^{-\ii L(k).x}\}_{k\in S^c})$,  the change of variables on  $w$ and  $\Omega'=$ diag$(\{\Ome_k -(\ome,L(k))\}_{k\in S^c},-\{\Ome_k -(\ome,L(k))\}_{k\in S^c}) $.} \smallskip

ii)\  The Hamiltonian $H$ in the new variables is
$$N+(y', A y')+P^{(3)}+P^{(6)}  $$
where  $N$ in the new variables is 
\begin{equation}\label{Nc}
N:= (\omega(\xi),y')+Q_{M'}(w'),\ M'= \Omega'+ X M X^{-1},
\end{equation}

and the terms $P^{(3)}$, $P^{(6)}$ satisfy the bounds of Theorem \ref{teo1}, {  v)}.
\end{proposition}
\begin{proof} {\it i)}\ 
Since $$ \sup_{D(s,r/2)}| w'|_{a,p}\le e^{ C s}|w|_{a,p}\le e^{ C s}r/2 \le r$$ for $s$ small enough the transformation is well defined from $D(s,r/2)$ to $D(s,r)$. It is symplectic because: 
$$dy\wedge dx +\ii dz\wedge d\bar z=dy'\wedge dx' - \sum_kL(k)d(|z_k|^2)\wedge dx' +$$ $$\ii dz'\wedge d\bar z'  - \sum_kL(k).dx'(z'_k\wedge d\bar z'_k-\bar z'_k\wedge d z'_k)=dy'\wedge dx'+\ii dz'\wedge d\bar z'.$$
Finally it preserves the spaces $F^{i,j}$ since it is linear in the variables $w$.\smallskip

  {\it ii)}\  We   substitute the new variables in the Hamiltonian and use the relation  $JX = X^{-1}J$, i.e. the fact that $X$ is symplectic. The bounds follow from Lemma \ref{stime}, notice that we have put all the terms in $H_0^{>2}$ which are not quadratic in $y$ in the perturbation $P^{(6)}$ where they contribute to the terms of weight $\geq 3$ with a  term of order $r^{4+\alpha/2}$ which is negligible with respect to $P^{(3)}$. \end{proof}
 From an algebraic point of view,  we have performed a diagonal change of coordinates using the matrix $X$ on the free module  $V^{(0,1)}_f$. Recall that this is  the space   of finite linear combinations of the element  $z_k,\bar z_k,\ k\in S^c$ with coefficients in the algebra $A$ of finite Fourier series in the variables $x$. 

It is clear that the block structure is still preserved, the main fact is now that the matrix has constant coefficients.

 \begin{proposition}\label{reteo1}
i)\  $M'$ has constant coefficients and is block diagonal with the same block structure as $M$. On a block $(A,+)$ with root $x(A)$,   \begin{equation}
\label{MpA}M'_A= (|x(A)|^2+ 4\sum_j \xi_j)I+ 2C_A,
\end{equation} where $C_A$ has the following entries in the vertices $a,b$.  \begin{equation}
\label{CAab}C_A(a,b)=\begin{cases}
0\quad \text{if}\quad M_{a,b}=0,\quad a\neq b\\4\sigma(b)\sqrt{\xi_i\xi_j  }\quad \text{if  $(a,b)=e$ is  an edge marked $(i,j)$.
}\\ \sigma(a)(\xi,L(a)) \quad \text{if }\quad a=b.
\end{cases}  .
\end{equation}

ii) $C_A$ is self-adjoint if $A$ does not contain red edges. 
If $A$ contains red edges there is a diagonal matrix $\Sigma$ with entries $\pm1$ such that $\Sigma C_A$ is self adjoint.

iii)\  The matrix $C_A$ depends only from the abstract $\mathcal M$--graph corresponding to $A$ and hence is chosen from a finite list. 
 
\end{proposition}
 \begin{proof}
 i) By definition $M'= \Omega'+ X M X^{-1}$ so let us analyze separately these two operators.  Since $X$ is diagonal the matrix $X M X^{-1}$ is 0 on the diagonal. 

The off diagonal entry associated to the indices $a,b$ of two vertices in the given block is 0 unless they are joined by an edge marked $(i,j)$. Then 

  \begin{equation}\label{0fdi2} (XM  X^{-1})_{a,b}=\end{equation} $$\sigma(b)4\sqrt{\xi_i\xi_j}  e^{  \sigma(a)\ii L(a).x}  e^{  -\sigma(b)\ii L(b).x}e^{\ii(\sigma(b)L(b)-\sigma(a)L(a)) .x}=4\sigma(b)\sqrt{\xi_i\xi_j}$$

   On the diagonal  of $M'$ we have the contribution of $ \Omega'$.  Applying  Formulas \eqref{iubl} and \eqref{gliome}, since by \eqref{alcc}  we have $\sum_iA^{(i)}_a=1-\sigma(a)$,  we get:\begin{equation}\label{iblocchi}
\Omega'_a=\sigma(a)(|k(a)|^2 +4\sum_j\xi_j -(\omega,L(a)))= 
\end{equation} 

$$|x(A)|^2 + \sigma(a)\sum A^{(i)}_a |v_i|^2+ \sigma(a)( 4\sum_j\xi_j -(\omega,L(a)))=$$$$|x(A)|^2 + \sigma(a)\sum A^{(i)}_a |v_i|^2+ \sigma(a)( 4\sum_j\xi_j -\sum_i(|v_i|^2 -2\xi_i + 4 \sum_{j } \xi_j )A^{(i)}_a  )=$$

$$|x(A)|^2 +   \sigma(a)( 4\sum_j\xi_j -\sum_i(  -2\xi_i + 4 \sum_{j } \xi_j )A^{(i)}_a  )= $$
$$|x(A)|^2 +    \sigma(a)[2(\xi,L(a))+ 4\sum_j\xi_j -  4 \sum_{j } \xi_j  (1-\sigma(a)) ]= $$$$|x(A)|^2 +    \sigma(a) 2(\xi,L(a))+ 4\sum_j\xi_j . $$
We then define $2C_A$ to have off-diagonal entries given by  \eqref{0fdi2} and on the diagonal $2\sigma(a)(\xi,L(a)).$ 
\smallskip

{\it ii) } Is immediate.\smallskip

{\it iii)} \   In the matrix $C_A$ the off--diagonal entries depend only on which pairs $a,b$ are connected by a marked edge, which depends only on the abstract $\mathcal M$--graph. In the same way the diagonal entries depend  on $L(a)$ which depends only on the path from the root to a (again this depends only on the abstract marked graph). Finally there are only a finite number of abstract marked graphs with $k\leq n+1$ vertices.

\end{proof}\begin{remark} Notice that in the new variables    the term $H_0$ is independent of $x$. 
\end{remark}
\subsection{Combinatorial blocks}  We have described the matrix  $M'$ in the basis $z'_k$. In particular, for a given geometric block $A$ we have chosen the block  in $V^{(0,1)}$ generated by the element $z_{x(A)}=  z'_{x(A)}.$
We have to understand in this formalism the elements in $F^{(0,1)}$-- that is  momentum conservation  and parity in the new variables-- and then compute the matrices of the operator $ad(N)$ on each combinatorial block of the graph $\Lambda_S$ (cf. Definition \ref{lambdas}).  
\begin{lemma}\label{mom2} Momentum conservation \eqref{cons1} and parity  in the new variables give the constrain:
\begin{equation}
\label{lama}e^{i \nu.x}y^i\prod_{k}(z_k')^{\alpha_k}(\bar z_k')^{\beta_k}\in F^{i,j} \to  \sum_i \nu_iv_i +\sum_{k}\sigma(k)(\alpha_k-\beta_k)x(k)=0\,,
\end{equation} $$ \sum_i \nu_i +\sum_k(\alpha_k+\beta_k)= {\rm even}.$$ In particular $e^{i \sigma(k)\nu.x} z_k'\in F^{0,1}$ if and only if $ x(k)=-\pi(\nu)$ and $\sum_i\nu_i$ is odd.
\end{lemma}
\begin{proof}The momentum conservation in the variables $w'$ can be derived directly from \eqref{cons1} by substitution:
$$e^{\ii \nu\cdot x}\prod_k(z_k')^\alpha(\bar z_k')^\beta= e^{\ii (\nu\cdot x -\sum_k(\alpha_k-\beta_k)L(k)\cdot x)}\prod_k(z_k)^\alpha(\bar z_k)^\beta  $$
then, if $L(k)= \sum_i A(k)^{(i)}e_i$ (cf. Corollary \ref{sunto1}),   momentum conservation reads 
\begin{equation}\label{mocon}
\sum_i\nu_iv_i -\sum_i\sum_k(\alpha_k-\beta_k)A(k)^{(i)}v_i +\sum_k(\alpha_k-\beta_k)k . 
\end{equation}    Recalling that, by Corollary \ref{sunto1} with $k=x_a$, $k- \pi(L(k))= \sigma(k) x(k)$
one obtains formula \eqref{lama}.    The parity condition is  that $\sum_i(\nu_i - \sum_k (\alpha_k-\beta_k)A(k)^{(i)})+\sum_k(\alpha_k+\beta_k)$ is even. In Corollary \ref{sunto1} we have seen that $\sum_i A(k)^{(i)}=0,2$ and the parity follows.\end{proof} 
\smallskip

 We have described the matrix $M'$ on geometric blocks. From this description we can deduce a description of the operator $ad(N)$ on the combinatorial blocks, in particular on $F^{(0,1)}$. A combinatorial block $(\A,+)$ over a given geometric block $ A$ is generated by  an arbitrary lift  $e^{\ii\nu.x}z_{x(A)}$ of the root $x( A)$  in $F^{(0,1)}$ (i.e. by choosing a $\nu$ such that $x(A)=-\pi(\nu)$). By Lemma \ref{mom2},  the operator  $(\omega(\xi),y')$  contributes to $ad(N)_\A$ the scalar matrix  $-2\ii (\xi,\nu)I$. \begin{definition}\label{Cacomb}
Given a combinatorial block $\A$ generated by  the lift  $e^{\ii\nu.x}z_{x(A)}$ we define $C_\A= C_A -(\xi,\nu)I$.
 For the blocks $\A,-$ we apply sign change.\end{definition} 
 \begin{proposition}
\label{LaCA}  let $A$ be a geometric block with root $x(A)$. On the combinatorial block $(\A,+)$ generated by  an element $e^{\ii\nu.x}z_{x(A)}$ the operator $ad(N)$ has matrix:
\begin{equation}\label{allowa1}
-\ii ad(N)_\A=    \Big(|x(A)|^2+\sum\nu_i|v_i|^2  +4(\sum_i\nu_i+1)\sum_j\xi_j\Big)I+2 C_\A .
\end{equation}
 \end{proposition}
\begin{remark}
\label{delag}Notice that, for any given block $A$,  the two combinatorial blocks $\A,+$ and $\A,-$ form two Lagrangian subspaces of a non degenerate symplectic space and the full space $F^{(0,1)}$ is the direct sum of these subspaces. 
\end{remark}
 \subsection{Theorem \ref{teo1b}}
 The core of  Theorem \ref{teo1b} is to put the normal form $N$ in a canonical form through a linear change of variables in the $w'$ which depends smoothly on the parameters $\xi$.

Let $\Sigma$ be a  diagonal matrix with $p$ entries  equal to 1 and $q=n-p$ entries equal to $-1$.  If $p>0,q>0$  the matrix   $\Sigma$ defines an indefinite scalar product preserved by a non--compact form of the orthogonal group usually denoted $O(p,q)$.      If $C$ is an $n\times n$ matrix such that $\Sigma C$ is symmetric then it can be brought into a suitable canonical Jordan form  by conjugation with elements of  $O(p,q)$.  Is furthermore $C$ is semisimple  then this normal form consists in decomposing the space $\R^n$ into orthogonal subspaces stable under $C$ and irreducible under $C$. These have either dimension 1 and correspond to the real eigenvalues or dimension 2 and correspond to the pairs of complex conjugate eigenvalues  (one can also see that we have at most $\min(p,q)$ of such 2 dimensional subspaces.  In the case of our blocks containing red edges we have chosen a basis  where the symplectic products are of the form $\pm  i$ associated to such a sign matrix 
   $\Sigma$. For such a symplectic block sum of $A,+$ and $A,-$  a matrix  of the form  $\begin{vmatrix} U&0\\0&U\end{vmatrix}$ with $U\in O(p,q)$ is symplectic and brings  $\begin{vmatrix} C&0\\0&-C\end{vmatrix}$ in a  {\em diagonal} normal form (of course some diagonal blocks are $2\times 2$ and are matrices associated to complex numbers.  
\smallskip

 We need to study matrices which depend algebraically on parameters, we have the following  proposition:
\begin{proposition}\label{smoot}
If $C(\xi)$ depends algebraically on parameters $\xi\in \R^m$, one can  define globally  and algebraically its eigenvalues  provided we {\em cut}  some real semialgebraic hypersurface (in an arbitrary way) so that   the complement is simply--connected. 

If $C(\xi)$ is semisimple  on an open set,   one can  define globally  and algebraically   the change of coordinates which brings $C(\xi)$ in the given diagonal form provided we {\em cut}  some real semialgebraic hypersurface (in an arbitrary way) so that   the complement is simply--connected.
\end{proposition}
This is a fairly standard fact  and we leave out the proof.

 \begin{proof}[Proof   of Theorem \ref{teo1b}]  i) By formula \eqref{MpA}, on the  symplectic block  $(A,\pm)$ the quadratic Hamiltonian $N$  is given by the matrix   \begin{equation}\label{form}\MM_A=(|x(A)|^2+4\sum_j\xi_j)\begin{vmatrix}I &0 \\0&-I\end{vmatrix}+2 \begin{vmatrix}
C_A&0\\0&-  C_A 
\end{vmatrix} ,\end{equation}
where $x(A)$ is  the root of $A$ while $C_A$ depends only on the combinatorial block ${\mathcal A}$ of which $A$ is a realization. and   let $\lambda_1(\xi),\dots ,\lambda_k(\xi)$ be its eigenvalues.  We denote by $\tilde\Ome_i=\tilde\Ome_i(\xi)$ the corresponding eigenvalues of  $\MM_A^+$ thus
\begin{equation}
\label{tildOm}\tilde\Ome_i(\xi)=|x(A)|^2+4\sum_j\xi_j+\lambda_i(\xi).
\end{equation}

Suppose that $\MM_A$ is semisimple,   by Proposition \ref{smoot}   $\MM_A$  can be conjugated  by a real invertible matrix  $U$ in $D_{s,k}$.  Moreover for $\xi$ outside a semialgebraic set the matrix $U(\xi)$ is algebraic in the $\xi$. 
 We proceed as described above for all the $\MM_A$ which are self--adjoint, i.e. for all $A$ that do not contain red edges. We apply 
$w'\to w^{\rm fin}= U_A w'$ so that  $Q_{\MM_A}(w')= Q_{D_A}(w^{\rm fin})$ where $$D_A=\begin{pmatrix}\mathcal{D}_A & 0 \\ 0& -\mathcal{D}_A\end{pmatrix}\,,\quad \mathcal D_A ={\rm diag}(\tilde\Ome_{k(a)})$$ $\mathcal D_A$ is the diagonal matrix of the (real) eigenvalues of $\MM_{ A}$. Now the total transformation $U$ is block diagonal. Since $|k(a)|^2 \leq |x(a)|^2+ C$ for some uniform $C$ (depending only on the size $\max_{v_i\in S}|v_i|^2$) we have that 
$$|w^{\rm fin}|_{a,p}\leq C |w_+|_{a,p} .$$ 
Item i) follows by setting $R'$ to be the matrix of  $\sum_{k,h\in \mathcal H}(w^{\rm fin}_h,M_{h,k}'Jw^{\rm fin}_k)
$.\smallskip

ii)\ The $\tilde\Ome_k$ are the eigenvalues of the matrices ${\mathcal M}_A$ of Formula \eqref{form} as $A$ varies of the connected components of $\Gamma_S$. Since by Proposition \ref{reteo} iv) $C_A$ depends only on the abstract M-graph, there is only a finite number of such matrices. Moreover the entries of these matrices are homogeneous of degree one in $\xi$. We apply Proposition \ref{smoot}; we cut away a semialgebraic set ${\mathfrak A}$ such that the complement is simply connected. In $A_\alpha\setminus {\mathfrak A}$ we obtain a list of eigenvalues
 $\lambda^{(i)}(\xi)$ and a list of matrices (with determinant equal to one) $U_i(\xi)$ which diagonalize the $C_A$ which are semisimple (this set contains at least all the $C_A$ which do not contain red edges). These are all algebraic functions.

 Let ${{\mathfrak A}}_\epsilon$ be a neighborhood of ${\mathfrak A}$ of radius $\epsilon$ then in $\tilde A_\alpha:=A_\alpha\setminus {{\mathfrak A}}_\epsilon$ the $\lambda^{(j)}$ are analytic functions in $x$, the bounds \eqref{linh} follow by Cauchy estimates by setting $\epsilon= r^{\alpha}/4$.
 
iv)\ Follows directly from Lemma \ref{stime} since the linear change of variables $(x,y,z)\to (x',y',z^{fin})$ does not   modify the bounds. Indeed the only point is to prove that the Lipschitz bounds are unchanged, this follows by the estimates of item {\it iii)} and the fact that  in $\tilde A_\alpha$ also the eigenspaces are analytic in the parameters $\xi$.
\end{proof}
\begin{remark}\label{iperb} Suppose that $C_A$ for some $A$ is semisimple for all $\xi\in A_\alpha$. If the eigenvalues of $C_A$ are real then one can proceed as above.
 In the case of hyperbolic blocks (with eigenvalues $ a\pm  \ii b$)  we obtain  the Hamiltonian:
 $$ -a(z_2z_1+\bar z_2\bar z_1)-b \frac{\bar z_2^2+z_2^2}{2} =-\frac12  w\begin{pmatrix} D & 0 \\ 0 &-D\end{pmatrix}J w^t\,,\quad D= \begin{pmatrix} a & -b \\ b &a\end{pmatrix}, $$ we have proved Remark \ref{mini}.

 In general however $C_A$ may not be semisimple and thus have Jordan blocks of size $>1$. A normal form classification of this quadratic Hamiltonian is feasible, see \cite{A},  but intricate and not particularly useful. \end{remark}

\smallskip
\subsection{The homological equation\label{homeq}} The homological equation consists in the analysis of the range and kernel of the operator $ad(N)$ on $F^{\leq 2}$.  Using the block diagonal form \eqref{rep1}  this analysis can be split into 4 different equations. The crucial ones are the last two which for historical (and confusing) reasons are called the first and second Melnikov equations. 
 
In order to study the Homological equation it is not really necessary to perform the
reduction of Proposition \ref{reteo}, since we have seen that the change of variables leaves it essentially unchanged. In the same way it is not necessary to diagonalize the Hamiltonian,  since this just conjugates the operator $ad(N)$. The main point in the homological equation is to study the Kernel of $ad(N)$ on the real subspace of $F^{\leq 2}$, and to show that the kernel coincides with  the subspace $F^{int}$:
$$ F^{int}:=\left\{F\in F^{\leq 2}\vert\; F(\lambda,W)= c+(\lambda(\xi),y)+ Q_{W(\xi)}(w)\right\}$$ where $c$ is a constant, $\lambda\in \R^m$ and  $W$ is a block--diagonal matrix with the same blocks as $M'$ and (block by block) simultaneously diagonalizable with $M'$. 

\begin{lemma}\label{hoho}
For any given $\xi$ the condition $\ker(ad(N))=F^{int}$ is equivalent to the  fact that the non-trivial Melnikov resonances  \eqref{mel} do not occur.
\end{lemma}
\begin{proof} By Lemma \ref{homol} $ad(N)$ restricted to $F^{0,0}$ and $F^{1,0}$ is the operator $\ome(\xi)\cdot\partial_x$ which has eigenvalues $(\ome(\xi),\nu)$.

 By Lemma \ref{homol} $ad(N)$ restricted to $F^{0,1}$ has eigenvalues $\ii \big((\ome(\xi),\nu )\pm \tilde\Ome_{k(a)} \big)$ and hence is invertible if they are non--zero.
 
 For $F^{0,2}$ we use Lemma \ref{homol2}:
 for each combinatorial block  consider a basis of eigenvectors $a_i^{\A}$ for $ad(N)$  on $\A,+$ denote by $b_i^{\A}$ the dual basis in $\A,-$. A basis for $F^{0,2}$ formed by eigenvectors of $ad(N)$ is obtained by making all the possible products of two eigenvectors in the basis. The corresponding eigenvalue is the sum of the two eigenvalues.
 
Thus $ad(N)$ restricted to $F^{0,2}$ is invertible on an eigenvector (say for example $a_i^\A b_j^\mathcal B$)  if the corresponding Melnikov resonance does not occur. Moreover the trivial Melnikov resonances $a_i^\A b_i^\A$ span $F^{int}\otimes\C\cap F^{0,2}$.
 \end{proof}
\begin{proof}[Proof of Theorem \ref{teo2} item i)]  Given a block $A$ with root $x(A)$,  let $(\A,+)$ be a corresponding combinatorial block. By the previous Lemma we need to show that $ad(N)_\A$ is invertible.
The operator $-\ii ad(N)_\A$ computed at $\xi=0$   is the scalar matrix $(\ome(0),\nu )\pm |x(A)|^2$, if this number is non zero then  $ad(N)_\A$  is invertible for all $\xi$ small. If $ad(N)_\A (\xi=0)=0$ then we study $ -\frac\ii2 ad(N)_\A$ using  \eqref{allowa1}. By definition $C_\A$ has  off-diagonal entries  $2\sqrt{\xi_i\xi_j}$, thus  $ \frac\ii2 ad(N)_\A$ modulo $2$ is diagonal. Moreover, since $\sum_i\nu_i$ is odd while  $\sum_i A(k)^{(i)}$   is even (see the proof of the parity property),  we have that $\frac12 ad(N)_\A$ modulo $2$ is formally invertible hence also $ad(N)_\A$ is invertible for  values of $\xi$ outside some proper algebraic hypersurface.
\end{proof}
\subsection{Dimension $2$ and $3$\label{dim23}}
In order to prove Theorem \ref{teo2} item {\it ii)} we need stronger conditions, which we shall verify in dimension $n=2,3$. In all the statements of this paragraph we assume that $n=2,3$.

\begin{theorem}\label{LIRRI} The characteristic polynomials of all the block matrices $M'_A$ are irreducible.  
\end{theorem}
\begin{proof}
See \S  \ref{irrid} where we do a case analysis.
\end{proof}

Given a combinatorial block $\A,\pm $  set $\chi_{\A,\pm}(t,\xi)$ to be the characteristic polynomial   of the operator $ad(N)$ restricted to the given block 
\begin{corollary}\label{sepca}  

i)\ For all combinatorial blocks  $\chi_{\A,\pm}(t,\xi)$  is irreducible.

\smallskip

ii)\ For all $\xi$ outside a zero-measure set, the eigenvalues $\tilde\Omega_k$ relative to the same block are all distinct. Moreover the matrix $R$ defined in Theorem \ref{teo1b} i) is formed of $2\times 2$ hyperbolic blocks as in Remark \ref{iperb}. 
\smallskip

iii)\ If two different combinatorial blocks $\A,\pm,\A',\pm$  associated to two geometric  graphs $A,A'$ have a common eigenvalue of $ad(N)$  (as function of $\xi)$ then $\chi_{\A,\pm}(t,\xi)= \chi_{\A,\pm}(t,\xi)$. In particular the roots are on the same sphere: $|x(A)|^2=|x(A')|^2$. 
\end{corollary}
\begin{proof} The characteristic polynomial of all the blocks are obtained by translating the variable $t$  from the polynomials of the previous Theorem so are irreducible. Since the characteristic polynomial of all the blocks is irreducible then the eigenvalues of each block are all distinct as algebraic functions. If we remove an arbitrary open set of parameters $\xi$ which contains the zero-set of the discriminant we have for each block all distinct eigenvalues (which implies that the blocks are all diagonalizable and that the eigenvalues and eigenvectors depend smoothly on the parameters). 

This follows from the fact that an irreducible polynomial is also the minimal polynomial satisfied by any algebraic function giving one of its roots, over the field of rational functions. Thus  two irreducible polynomials have a root in common if and only if they are equal. This implies in particular that  $|x(A)|^2=|x(A')|^2$.
\end{proof}

In order to prove Theorem \ref{teo2} item {\it ii)}, it is enough to verify the conditions of Lemma \ref{hoho}. This follows from: 
\begin{theorem}[Separation]\label{separ}  A  basis for $\ker(ad(N))$ on $F^{0,2}$ is given by the elements $a_i^{\A}b_i^{\A}$  (for two eigenvectors in conjugate blocks $\A,+; \A,-$).

\end{theorem}
The Theorem follows from a technical Lemma. Take a combinatorial block  $\A,\pm$ and denote by   $\chi_{\A,\pm}(t,\xi)$ the characteristic polynomial of the corresponding matrix $C_\A$ or $-C_\A$.

\begin{lemma}\label{SEPA}
The block $\A$ and its sign can be uniquely reconstructed by the polynomial $\chi_{\A,\pm}(t,\xi)$.
\end{lemma}
\begin{proof}
See \S \ref{SEP}.
\end{proof}
In practice in order to prove this Lemma we shall inspect 
colored marked graphs with 
  a root so that, when we embed $\A_i$ in  $\tilde\Lambda$  using Theorem \ref{limbed}  we have a complete list of representatives  under the group of translations and sign change.  \medskip

\begin{proof}[Proof of Theorem \ref{separ}]
This follows immediately from Corollary \ref{sepca} and Lemma  \ref{SEPA}.
\end{proof}

\begin{proof}[Proof of Corollary \ref{canto}]
The only difficulty comes from the study of the second Melnikov condition i.e. the third inequality in 
\eqref{bobo}. The main point is that  the functions 
$$ (\ome(\xi),\nu) + \tilde\Ome_k(\xi) +\sigma \tilde\Ome_h(\xi) = (\ome(\xi),\nu) + L + \lambda^{(i)}(\xi) +\sigma \lambda^{(j)}(\xi), $$  are  algebraic and non trivial. Moreover $L$ is an integer while  the  $\lambda^{(i)}(\xi)$ are defined in \eqref{autov} so that $ \lambda^{(i)}(\xi) +\sigma \lambda^{(j)}(\xi)$ is a {\em finite} list of analytic functions, which respect the bounds \eqref{linh} in the domain $A_{\alpha}\setminus \mathfrak{A}_{\varepsilon}$.
 Then for a given $\nu$, we impose the second Melnikov condition  for all the {\em infinite} choices of $h,k$ compatible with momentum conservation by removing from  $A_\alpha$ a {\em finite} number of ``strips'' of measure of order $\gamma (A+r^\alpha)/ (1+|\nu|^{\tau_0+1})$ (the number of ``strips'' is given by the number of different functions $ \lambda^{(i)}(\xi) +\sigma \lambda^{(j)}(\xi)$) .  Finally we notice that if $A=1$ then in order to have $(\ome(\xi),\nu) + \tilde\Ome_k(\xi) +\sigma \tilde\Ome_k(\xi)\leq 1$  we must have $|\nu| > r^{-\alpha}$.
\end{proof}
\subsection{Dimension $n>3$.}\label{n4}
In dimension $n>3$ we have not been able to prove the validity of the second Melnikov condition, although we suspect that it is true, according to our Conjecture 1  in \S\ref{matrici} and a conjectural separation Lemma. In this subsection we want to sketch how one can still apply a KAM scheme by discussing the homological equation in this more general setting.  The reader should compare this discussion with the one presented in \S \ref{conclu} and again with the KAM scheme developed in \cite{BB}.

\smallskip

 The first step consists in performing the canonical {\it Fitting decomposition} of the operator $  ad(N)$, namely we decompose the space so that  each block corresponds to one (and only one) of the generalized eigenvalues, as usual in our definitions we say that two eigenvalues coincide if they do so identically as functions of $\xi$. We  define the change of variables of the  Fitting decomposition  on each abstract combinatorial block and then  notice that the same change of variables decomposes all the translations. Thus this change of variables  is analytic for all 
$\xi$ outside some real semialgebraic hypersurface. \begin{proposition}
In the Fitting decomposition of $F^{0,1}$ under $ad(N)$ each eigenvalue appears with a uniformly bounded multiplicity $\leq \kappa$.
\end{proposition}
\begin{proof}The eigenvalues of $\ii$ $ad(N)$ are $\mu= \pm \big((\ome(\xi),\nu)+\tilde \Omega_k\big)$, hence  two eigenvalues  may coincide only if they coincide at $\xi=0$.  Supposing that this is true we study the remaining linearly homogeneous terms.  The linear terms  in an eigenvalue are a translation of the finite list $\lambda^{(i)}(\xi)$, and $$ (\ome(\xi)-\ome(0),\nu \pm \nu') = \lambda^{(i)}\pm \lambda^{(j)}$$ may hold only if $\nu\pm \nu'$ is uniformly bounded. 
\end{proof}
Notice that with each eigenvalue $\mu$ there is also the eigenvalue $-\mu$ with the same multiplicity.
Let us denote by $F^{0,1}_\mu$ the generalized eigenspace relative to $\mu$. It is easy to see that, from the first Melnikov condition, we have
\begin{lemma}
 $$\{F^{0,1}_\mu,F^{0,1}_{\mu'}\}=0 \,\quad \iff \mu+\mu'\neq 0 $$ moreover $F^{0,1}_\mu$ and $F^{0,1}_{-\mu}$ are in duality under Poisson bracket.
\end{lemma}
\begin{proof}
We start from the canonical decomposition into pairs of conjugate combinatorial blocks.  Two different pairs are orthogonal (under Poisson bracket) each block is isotropic and in duality with the conjugate block. If we decompose a given block (identified by the frequency $\nu$) in generalized eigenspaces under $ad(N)$ we have, a list of non zero eigenvalues $\mu_\nu^{(i)}$  and in the conjugate space the  eigenvalues $-\mu_\nu^{(i)}$.  The duality pairing of Poisson bracket puts in duality the eigenspace relative to $\mu_\nu^{(i)}$  and in the conjugate space the  one relative to the eigenvalue  $-\mu_\nu^{(i)}$. It follows that the total generalized eigenspace relative to a given eigenvalue  (which is necessarily non--zero) appearing in the sum of two conjugate blocks is isotropic and in duality with the generalized eigenspace relative to the opposite eigenvalue    in the same sum of two conjugate blocks.
\end{proof}
\smallskip

We decompose $F^{0,2}$ according to the pairs $(\mu,\mu')$ and  denote by $F^{0,2}_{\mu,\mu'}= F^{0,1}_\mu \cdot F^{0,1}_{\mu'}$ (the product of functions).  By Leibniz rule $ad(N)$ on this space  has generalized eigenvalue $\mu+\mu'$. 
As il \S \ref{conclu}, for the first step we consider the Cantor set $\mathcal C$:
\begin{equation}\label{mama}
|\mu_\nu^{(i)}(\xi) +  \mu_{\nu'}^{(j)}(\xi)| \geq \frac{\gamma (A+r^\alpha)}{1+|\nu +\nu'|^{\tau_0}}
\end{equation}    for all  $\mu_\nu^{(i)}(\xi) +  \mu_{\nu'}^{(j)}(\xi)\not\equiv 0$ where $A= 1$ if the integer $\mu_\nu^{(i)} (0)+  \mu_{\nu'}^{(j)}(0)\neq 0$ and $A=0$ otherwise. The fact that this Cantor set has positive measure follows by the same reasoning as in Corollary \ref{canto}. \begin{proposition}\label{wemel}
For all $\xi\in \mathcal C$ the operator $ad(N)$ is invertible on each $F^{0,2}_{\mu_\nu^{(i)} ,\mu_{\nu'}^{(j)}}$ such that  $\mu_\nu^{(i)} +  \mu_{\nu'}^{(j)}\not\equiv 0$. For all  $x\in F^{0,2}_{\mu_\nu^{(i)} ,\mu_{\nu'}^{(j)}}$ one has:
$$ |ad(N)^{-1} x |_{D(s,r)} \leq C\frac{|\nu+\nu'|^{\kappa^2 \tau_0 +1 }}{A+ r^{\alpha}}|x|_{D(s,r)},$$ with $C$ some universal constant and $A$ defined as in \eqref{mama}.
\end{proposition}
\begin{proof}
The operator $ad(N)$ on $F^{0,2}_{\mu_\nu^{(i)} ,\mu_{\nu'}^{(j)}}$ has generalized eigenvalue $\mu_\nu^{(i)} +  \mu_{\nu'}^{(j)}$ so that if $ad(N)$ is semi--simple the result is obvious with $\kappa=1$. In general we notice that $F^{0,2}_{\mu_\nu^{(i)} ,\mu_{\nu'}^{(j)}}$ is a finite dimensional space with dimension $\leq \kappa^2$.   The entries of  $ad(N)$ on $F^{0,2}_{\mu_\nu^{(i)} ,\mu_{\nu'}^{(j)}}$  are all bounded by $C|\nu+\nu'|$. Then,  if  $A=1$   the result follows by Cramer's rule.  If $A=0$  $ad(N)$ is homogeneous of degree one in $\xi$, hence the entries of  $ad(N)r^{-\alpha}$ are bounded by 
$C|\nu+\nu'|$ (recall that $\xi \leq r^\alpha$), again we apply Cramer's rule.
\end{proof}
Let  $[R]$ be the projection of $R$ on the spaces $F^{0,2}_{\mu,-\mu}$. 
\begin{corollary} The operator  $ad([R])$ on $F^{(0,1)}$  is block diagonal relative to the Fitting decomposition of $ad(N)$.
\end{corollary}
\begin{proof} The result follows by the duality of $F^{0,1}_\mu$ and $F^{0,1}_{-\mu}$ by Poisson bracket.
\end{proof}

The purpose of a KAM algorithm will now be to construct a change of variables $\Phi= e^{ad(F)}: D(s/4,r/4)\to D(s,r)$ such that
 $e^{ad(F)}\circ H= N^\infty + P^{\infty}, $ where  $$N^\infty= (\ome^\infty,y^\infty)-\frac12 w^t M(x^\infty) J w\,,\quad P^{\infty}=\sum_{i,j:\; 2i+j>2} P^{\infty}_{ij}(x^\infty) (y^\infty)^i (w^\infty)^j$$ and $M$ is represented on  $F^{0,1}$ as a block--diagonal matrix on the blocks  $F^{0,1}_\mu$ defined by the Fitting decomposition of $ad(N)$.  
  
From a purely formal point of view at each step of the KAM algorithm, we have the Hamiltonian
$$H_k=N_k + R_k + P_k $$ where $ad(N_k)$ is block diagonal with respect to the Fitting decomposition of $ad(N)$,  $R_k$ belongs to $F^{\leq 2}$ and $P_k\in F^{>2}$. Naturally one should prove appropriate estimates to show that the sequence $R_k$ tends to zero super-exponentially while   $P_k$ is bounded.
To define $H_{k+1}$ we   solve the homological equation 
\begin{equation}\label{hon}
ad(N_k) F_k = R_k -[R_k],
\end{equation} 
where $[\cdot]$ is defined with respect to the  Fitting decomposition of $ad(N)$.  We define   $N_{k+1}:=N_k+[R_k]$  which remains block diagonal with respect to the initial Fitting decomposition, although it is possible that a  given eigenspace may split  into different eigenspaces. The main point is that, due to the fact that  the $R_{h}$ are very  small, the Fitting decomposition of $N_k$ may only be a refinement of the  decomposition of $ad(N)$. Indeed  two different eigenspaces of $ad(N)$ are different for all the $ad(N_k)$, since the correction to the eigenvalue is small.  Then formula \eqref{hon} defines $F_k\in F^{\leq 2}$ uniquely for $\xi$ on an appropriate Cantor set ${\mathcal C}_k$ where, $$ |ad(N_k)^{-1} x |_{D(s,r)} \leq  r^{-\alpha}|\nu+\nu'|^{ \tau}|x|_{D(s,r)},$$ for all $x\in F^{0,2}_{\mu_\nu^{(i)} ,\mu_{\nu'}^{(j)}}$. 


By definition $ H_{k+1}:=\exp(ad(F_k))H_k$, we define $R_{k+1}:= \Pi_{F^{\leq 2}}\exp(ad(F_k))H_k-N_{k+1}$ and $P_{k+1}$ consequently.

\smallskip

 \subsection{The issue of linear stability}
As we have seen in the previous paragraph, in dimension 2 and 3 the three Melnikov conditions enable us to solve the Homological equation and hence put up a KAM iteration (as discussed in \S \ref{conclu}). 
 For the linear stability of the normal form we must verify weather the eigenvalues $\tilde\Ome$ may be all real.   Indeed it is not true that for all small parameters the eigenvalues are real (as one can see even in the case of a single red edge where the condition is that the discriminant $(\xi_1+\xi_2)^2-16\xi_1\xi_2>0$). Nevertheless the condition on the parameters for the roots to be real is given  by a system of inequalities (still given by Sylvester's theory).

The question is to verify that this open region intersects our domain $A_\alpha$. 
Let us recall briefly the Theory, (cf. \cite{PrC} for a modern exposition).  Given a monic polynomial $f(t)=\prod_{i=1}^n(t-x_i)$  its coefficients are up to sign the elementary symmetric functions $\sigma_h$ in the $x_i$. Consider the {\em Newton functions} $\psi_h:=\sum_{i=1}^nx_i^h$. There are simple recursive formulas expressing the $\psi_h$ as polynomials in the $\sigma_k,\ k\leq h$ with integer coefficients.  Consider next the {\em Bezoutiante matrix}, that is the symmetric $n\times n$ matrix $B$ with entries $\psi_{i+j-2}$.  Its determinant is the {\em discriminant} and equals $\pm\prod_{i\neq j}(x_i-x_j)$. If the polynomial has real coefficients then   $b$ is a real symmetric matrix and  its signature is the number of real roots of $f(t)$.

In particular $B$ is positive definite if and only if all the roots are real and distinct.  The condition on a symmetric matrix to be positive definite is given by the positivity of the determinants of all the principal minors. In
 our setting thus, for every block containing red edges   we deduce a finite number of inequalities in the parameters $\xi_i$. The region where all these inequalities are satisfied is thus the region where all the eigenvalues are real and distinct.  This region is a cone  and the issue is to show that it is non--empty.  This requires again a very complicated case analysis which we do not perform.
We refer to \S \ref{posit} for a partial discussion.

In dimension $n=2$ however we may follow the approach by \cite{GYX} to further simplify the allowable blocks $A$ and hence the matrix $M'$ of Theorem \ref{teo1}.  

\begin{constraint}[\cite{GYX}]\label{natu} We choose the tangential sites $S$ so that there are no geometric blocks in Theorem \ref{sunto} with more than one edge. 
\end{constraint}
The meaning of this constraint is as follows:  a vector $k\in \Z^n$ is the root of a geometric block with $k$ edges if $k$ satisfies a system of $k$ linear and quadratic equations (see Formula \eqref{Ilsom}, notice that the equations depend only on the abstract marked graph). The constrain above can be verified by requiring that all the $2\times 2$ systems which identify blocks with two edges do not have any  {\em integer} solution.  Naturally such a strong constrain may hold only in dimension $n=2$.

\begin{proof}[Proof of Corollary \ref{coro}] We impose Constrain \ref{natu}.
For all $S$ such that Constrain \ref{natu} holds, the matrices $C_A$ are $2\times 2$ and quite simple (the possible $ t I-C_A$ are written in Formula \eqref{22}). Then all the eigenvalues can be computed explicitly. By inspection one sees that all the $\tilde\Ome_k$ are real in the domain:
$$D:=\left\{\xi\in \R^m\vert\quad \forall i,j=1,\dots,m\,,\; i\neq j\,,\quad \xi^2_i+\xi^2_j-14\xi_i\xi_j>0\right\}. $$ Clearly $D$ is a cone and the intersection with $A_\alpha$ is non empty. 
\end{proof}
\part{The algebraic combinatorial Theorems}
 
 \section{A graph problem\label{ilG}}

\subsection{A universal  graph\label{umg}}
We need to develop some combinatorics which is useful to study  the graph $\Lambda_S$   introduced in \S \ref{ilgrco}.

Let us thus choose  $m$ variables  $e_1,\ldots,e_m$.  Denote by $$\Lambda:=\{\sum_{i=1}^ma_ie_i\},\ S^2[\Lambda]:=\{\sum_{i,j=1}^ma_{i,j}e_ie_j\},\ a_i,a_{i,j}\in\Z$$ the lattices generated by these variables $e_i$ and that generated by the products $e_ie_j $.  Set $ \eta:\Lambda\to \Z,\  \eta(e_i):=1$ the {\em augmentation}.

We define a structure of {\em colored marked graph} on $$\Lambda\times \Z/(2):=\Lambda^+\cup \Lambda^-;\quad  \Z/(2)=\pm 1$$   as follows.

We take two elements $(a,\sigma), a=\sum_im_ie_i,\quad (b,\tau), b=\sum_in_ie_i$:
\begin{enumerate}
\item We join $(a,\sigma), (b,\tau)$ with an oriented  black edge, marked $(i,j)$ if  \begin{equation}\label{latonero}
\sigma=\tau,\ b=a+e_i-e_j,\iff a=b+e_j-e_i.
\end{equation} 

\item We join $(a,\sigma), (b,\tau)$ with an unoriented  red edge, marked $(i,j)$ if  \begin{equation}
\label{latorosso}\sigma=-\tau,\ b+a+e_j+e_i=0.
\end{equation} 
\end{enumerate} 
\begin{remark}
In case $i)$ we have $\eta(a)=\eta(b)$ in case  $ii)$  we have $\eta(a)+\eta(b)=-2.$ 
\end{remark}
We will draw  a black edge with its orientation either as horizontal or as vertical edge as
\begin{equation}
\label{dised}\xymatrix{b=a+e_i-e_j&&a\ar@{-}[r]^{(i,j)} &b&a& c \ar@{=}[r]^{(i,j)} &d\\ c+d+e_j+e_i=0&&&&b\ar@{-}[u]^{(i,j)}&& }
\end{equation} 

Recall that \begin{definition}
i)\quad   A {\em path} $p$ of length $\ell(p)=k$, from a vertex $a$ to a vertex $b$  in a graph is a sequence of vertices
$p=\{a=a_0,a_1,\ldots,a_k=b\}$  such that $a_{i-1},a_i$ form an edge for all $i=1,\ldots, k$.

The vertex $a$ is called the {\em source} and $b$ the {\em target} of the path.
\smallskip

ii)\quad   A {\em circuit}  is a path from a vertex $a$ to itself.
\smallskip

iii)\quad A graph  without circuits is called a {\em tree}.
\end{definition} 

It is sometimes useful to project the graph on the factor $\Lambda$  and just follow the paths here.

Let $\sigma(p)$ be  1 if  $p$ has an even number of red edges and $-1$ if odd.  We see that 
\begin{lemma}
$\eta(a)=\eta(b)$ if  $\sigma(p)=1$,\ $\eta(a)+\eta(b)=-2$ if  $\sigma(p)=-1.$
\end{lemma}
Given an integer $\alpha$ set $\Lambda_\alpha:=\{a\in\Lambda\,|\,\eta(a)=\alpha\}.$
\begin{proposition}\begin{enumerate}\item Each $\Lambda_\alpha^+$ or $\Lambda_\alpha^-$ is a  connected component of the graph in which we only use black edges. \item For each pair of integers $\alpha,\beta$ with $\alpha+\beta=-2$  we have that the set $\Lambda_\alpha^+\cup  \Lambda_\beta^-$ is a connected component of the graph.

\end{enumerate}

\end{proposition}
\begin{proof}
By the previous identity   each connected component of the graph  is inside one of these sets.
It is clearly enough to prove i)  which we easily see by induction.

\end{proof}

The natural choice coming from conservation of mass is to take  $\Lambda_{-1}^+\cup  \Lambda_{-1}^-.$

There are symmetries in the graph. First the symmetric group $S_m$ of the $m!$ permutations of the elements $e_i$ preserves the graph. Next the {\em sign change map} and,  given any   element $c\in \Lambda$,  the {\em translation map}:
\begin{equation}
\label{sitr}(a,\sigma)\mapsto (a,-\sigma),\quad \tau_c:  (b,\sigma)\mapsto  (b+\sigma c,\sigma). 
\end{equation}  If we want to restrict to $\Lambda_{-1}\times\Z/(2) $ we need to have $\eta(c)=0$.

If $A$ is any subgraph we set $\bar A$ to be the transform of $A$  under   sign change.

In particular  we  may shift and possibly change sign to  the graph so that any element $(a,\sigma)$  is sent to $(0,+)$.

The element  $(0,+)$ is called {\em the root}.
\begin{definition}\label{CMG}
A {\em complete marked graph}, on a set $V\subset    \Lambda\times\Z/(2)$ is the full sub--graph generated by the vertices   in $V$.
\end{definition}
We shall see that these are the combinatorial objects appearing in our context. By the previous remarks given such a $V$ and an element $(a,\sigma)\in V$ we can apply translations and sign change so that the root  $(0,+)$ is in the transformed graph. Thus we start by studying the connected graphs containing the root.

It is very convenient  to understand this picture as follows. \begin{remark}
\label{zerdu} A connected   graph containing $(0,+)$  is completely determined by its projection in $\Lambda$. This lies in the set  \begin{equation}
\label{latil}\tilde\Lambda:=\{a\in\Lambda\,|\,\eta(a)=\{0,-2\}\}.
\end{equation}  \end{remark} 

Given   a vector $a\in\tilde\Lambda$ we have either $\eta(a)=0$ and then we assign to it  the sign $+$ otherwise $-$.  We shall then consider  $\tilde\Lambda$ as a {\em universal graph}.

In particular the  graphs that will appear are given by
\begin{definition}\label{cgli}
A {\em non degenerate graph}, is a complete graph in $\tilde\Lambda$ containing 0 and other linearly independent vectors.
\end{definition}
\begin{remark}
If we choose $k$ linearly independent elements in $\tilde\Lambda$  in general the  complete graph  generated by 0 and these elements is not connected.  The possible connected ones are easily seen to be a finite number.
\end{remark}
We shall need in fact to consider equivalent  two non degenerate graphs if we can translate one into the other. Inside $\tilde\Lambda$  the translations by elements $c$ with $\eta(c)=0$  preserve  the sign,  the translations by elements $c$ with $\eta(c)=-2$  exchange the sign. They are used to {\em change the root}.\smallskip

\subsection{Geometric  realization\label{gerea}}
Given a list $S$ of  $m$ vectors $v_i\in\C^n$, we define a linear map
$$\pi_S=\pi:\Lambda\to\C^m,\quad \pi(e_i):= v_i.$$ 
We consider $\C^m$   with the standard scalar product (NOT the Hilbert product). Set $(u)^2:=(u,u),\ u\in \C^m$.
Extend the map $\pi$  to a linear map  of the  polynomials  $S^2(\Lambda)$ of degree $2$  in the $e_i$ setting (cf. \eqref{ILPI}):
$$  \pi(e_i):= v_i,\quad 
\pi(e_ie_j):=(v_i,v_j),\
$$
We have $   \pi(AB)=(\pi(A),\pi(B)),\forall A,B\in\Lambda.$
We next define the linear map:
$$L^{(2)}:\Lambda\to S^2(\Lambda),\ e_i\mapsto e_i^2.$$ \begin{remark}
Notice that we have   $L^{(2)}(a)=a^2$ if and only if  $a$ equals 0 or one of the variables $e_i$.
\end{remark} 
\begin{definition}
Given an edge $(a,\sigma),(b,\tau)$   set $h:=-\pi(b),k:=-\pi(a)$. We say that the edge is {\em compatible} with $S$ or $\pi$ if
\begin{equation}
\begin{cases}
(h)^2-(k)^2=(v_j)^2-(v_i)^2\quad \text{if}\ \sigma=\tau,\ b-a=e_i-e_j\\
(h)^2+(k)^2=(v_j)^2+(v_i)^2\quad \text{if}\ \sigma=-\tau,\  b+a+e_j+e_i=0\end{cases}
.\end{equation}

\end{definition}
\begin{proposition}
The sub--graph  of  $\Lambda\times \Z/(2) $   in which we keep only the  edges compatible with $S$, coincides with the graph $\Lambda_S$ defined in \S \ref{Gra}.
\end{proposition}
\begin{proof} If the edge is black $b=a+e_i-e_j$ we have $h=k+v_j-v_i$  for a red edge instead we have $h+k=v_i+v_j$. 
\end{proof}
\begin{remark}
$\Lambda_S=\bar\Lambda_S$ is closed under sign change.
\end{remark}
It is convenient to express in a unique form the previous identities as
\begin{equation}\label{ricon}
\sigma (k)^2-\tau (h)^2=\pi[L^{(2)}(\tau b-\sigma a)]  \end{equation}$$\iff \sigma [(k)^2+\pi[L^{(2)}(   a)]]=\tau [(h)^2+\pi[L^{(2)}(   b)]].$$\begin{theorem}
If $(a,\sigma),(b,\tau)$ are in the same connected component of $\Lambda_S$ we have  
\begin{equation}
\label{cnno}\sigma\pi\big[  a ^2+ L^{(2)}(  a) \big]=\tau\pi\big[ b ^2+ L^{(2)}(  b) \big]. 
\end{equation}\end{theorem}
\begin{corollary}\label{lacco}
A component of $\Lambda_S$ is a complete subgraph (cf. \ref{CMG}) of the universal graph.\end{corollary}\begin{proof}
Fix an element $(a,\sigma)$ of which we want to find the component. Consider the set of all elements $(b,\tau)$  such that \eqref{cnno}  holds.  They determine a complete subgraph and  the component passing through $(a,\sigma)$ of this graph is  the required one.
\end{proof}
The construction we just made is such that:
\begin{proposition}
\label{ilmco} The map  $(a,\sigma)\mapsto -\pi(a)$  maps to the graph $\Gamma_S$ and it is compatible with the structure of the two graphs. 
\end{proposition}One of the first constraints we shall impose on the $v_i$  will make this map a graph isomorphism on each component, although infinitely many components map to the same geometric component. Given the geometric component $A$ associated to an element $r=-\pi(\mu)$ and a component $\A$ in $\Lambda_S$ mapping to $A$, all the other components mapping to $A$ are obtained as   translates of $\A$ under the subgroup of $\Z^m$ kernel of $\pi$ and their conjugates.  \smallskip

\subsection{The colored marked graphs\label{cmg}}
One can also start in an abstract way returning to and expanding Definition  \ref{acmg}.
\begin{definition}
A {\em colored marked graph}  or {\em $\mathcal M$--graph}  for short, is
\begin{itemize}
\item A graph $\Gamma$ (without  repeated edges).
\item  A color {\em red or black} on each edge, displayed 
$$ \xymatrix{&&a\ar@{-}[r]^{black} &b&&&& c \ar@{=}[r]^{red} &d }$$
\item A marking $(i,j),\  1\leq i\leq m,  1\leq j\leq m, i\neq j$ on each oriented edge with the convention that the opposite orientation corresponds to the exchanged marking $(j,i)$.
\end{itemize}
\end{definition}
The red edges are assumed to be unoriented, so for them we do not distinguish between the markings $(i,j),\ (j,i)$.\begin{example}\label{es1}
$$ \xymatrix{&c\ar@{-}[dr]^{2,1}& &&e&c\ar@{=}[dl]^{2,3}& &&\\(I)\quad a\ar@{-}[r]_{3,4}   &b\ar@{=}[u] ^{2,1} \ar@{-}[r]^{2,5} &d &(II)\quad a\ar@{-}[r] _{5,4}&b\ar@{=}[u]^{2,1}  \ar@{-}[r]_{2,3} &d }$$
\end{example}

Given a path $p$  from $a$ to $b$  and another $q$ from $b$ to $c$  we have the obvious:
\begin{itemize}
\item Opposite path $p^o:=\{b=a_k,a_{k-1},\ldots,a_1=a\}$ from $b$ to $a$.
\item {\em concatenation}  $q\circ p$, a path   from $a$ to $c$.
\end{itemize}

Observe  that a circuit $p=\{a=a_0,a_1,\ldots,a_k=a\}$ induces by {\em rotation}  a circuit from $a_i$ to itself for every $i$ and also an {\em opposite circuit} $p^o.$\smallskip

Given a path  $p$ we define 
\begin{itemize}\item  A {\em sign $\sigma(p)=\pm$}, where\begin{enumerate}
\item $\sigma(p)=+$ if the path has an even number of red edges, we also say that {\em $p$ is even},\item  $\sigma(p)=-$ if the path has an odd number of red edges, we also say that {\em $p$ is odd}.
\end{enumerate} 
\item A linear combination $L(p)=\sum_ia_ie_i,\ a_i\in\mathbb Z$ as follows:
\begin{enumerate}
\item If $p$ is an oriented edge $e$, i.e. $\ell(p)=1$ marked $(i,j)$  we set
\begin{equation}
L(p):=\begin{cases} e_j-e_i\quad \text{if}\quad e\ \text{is black}\\
e_j+e_i\quad \text{if}\quad e\ \text{is red}
\end{cases}
\end{equation}\item If $k>1$  let $p':=\{a_0,a_1,\ldots,a_{k-1}\}$ and $e=(a_{k-1},a_{k })$. 
We set
\begin{equation}\label{lei}
L(p)=L(e)+\sigma(e)L(p').
\end{equation}
 
\end{enumerate}
\end{itemize}

The following Lemma follows from an easy induction.
\begin{lemma}\label{IPA}
\begin{equation}  \sigma( p^o)=  \sigma(p),\quad 
L( p^o)=-\sigma(p)L(p).
\end{equation}\begin{equation}  \sigma(q\circ  p)=\sigma(q) \sigma(p),\quad 
L(q\circ  p)=L(q)+\sigma(q)L(p).
\end{equation}
\end{lemma}

\begin{example}
$$ \xymatrix{p:=&   a\ar@{-}[r]^{1,2} & \bullet\ar@{=}[r]^{3,2} &b;&& q:=&b \ar@{-}[r]^{1,3}& \bullet\ar@{-}[r]^{4,2} &  \bullet \ar@{=}[r]^{3,4} &c }$$
$$ \xymatrix{q\circ p=&   a\ar@{-}[r]^{1,2} & \bullet\ar@{=}[r]^{3,2}& \bullet\ar@{-}[r]^{1,3}& \bullet\ar@{-}[r]^{4,2} &  \bullet \ar@{=}[r]^{3,4} &c }$$
$$ \sigma(p)= \sigma(q)=-,\quad L(p)=e_3  +e_1;\ L(q)=2e_4 +e_1  -e_2 ;\ L(q\circ p)=  2e_4 -e_3  -e_2 .$$
\end{example}
\subsection{Compatibility}
We need to restrict to smaller and smaller classes of $\mathcal M$--graphs for our analysis thus we start setting
\begin{definition} A connected $\mathcal M$--graph $A$ is {\em compatible}, if, given any two  vertices  $a,b$ and two paths $p_1,p_2$ from $a$ to $b$ we have:
\begin{equation}\label{com1}\sigma(p_1)=\sigma(p_2),\quad L(p_1)=L(p_2).
\end{equation}
\end{definition}
Observe that the two conditions in \eqref{com1} are equivalent to saying that, given a circuit $p$ we have 
$$\sigma(p)=+,\ L(p)=0,\quad \forall\ p\quad \text{a circuit}.$$
For instance the graph (I)  of example  \ref{es1}  is not compatible.
\begin{ass} {\em All the graphs  we consider are compatible}
\end{ass} 

We apply  this assumption as follows.  Choose a vertex $r$ of a compatible connected graph $A$ which we call the {\em root}.  Given any vertex $a$ and a path  $p$ from $r$ to $a$  we set
$$\sigma_r(a):=\sigma(p),\ L_r(a):=L(p).$$

Let us   take as root another vertex $s$,  and let $q$  be a path from $s$ to $r$. We clearly have from Lemma \ref{IPA}:
\begin{equation}\label{cambio}
\sigma_s(a)=\sigma_r(a)\sigma_r(s),\quad L_s(a)=L_r(a)-\sigma_r(a)\sigma_r(s)L_r(s)
\end{equation}
\begin{ass}\label{differe} We restrict to those graphs  for which the linear forms  $L_r(a)$ are all  {\em different}.
\end{ass} 
Observe that this assumption  does not depend on the choice of the root. We have
    \begin{equation}\label{Heta}
\eta(L_r(a))=\begin{cases} 0\  \text{if}\ \sigma(a)=+\\2\  \text{if}\ \sigma(a)=-
\end{cases}
\end{equation}

\begin{theorem}\label{limbed}
In a compatible graph $A $ satisfying the previous assumption, the mapping  $\lambda:a\mapsto  - L_r(a) $ embeds  $A$ as a subgraph  of  the universal graph $\tilde\Lambda=\{a\in\Lambda\,|\,\eta(a)=\{0,-2\}\} $ (cf. \ref{zerdu}).
\end{theorem}
\begin{proof}
By Assumption \ref{differe}  we have an embedding at the level of vertices. We need to show that this is compatible with the edges.  Take  an edge $e=(a,b)$ in $A$, we know that   we have  $L_r(b)= L_r(e)+\sigma(e)L_r(a)$  hence if $e$ is red marked $(i,j)$  we have  $0=e_i+e_j-L_r(a)-L_r(b)$. If $e$ is black marked $(i,j)$  we have  $0=e_j-e_i+L_r(a)-L_r(b)$ or $-L_r(b)=e_i-e_j-L_r(a)$.  \end{proof}
This embedding maps $r$ to $(0,+)$.  We could have defined more general embeddings but they can be obtained from this by translations and sign changes.
 From now on, unless it is necessary, we shall drop the symbol $r$  in $\sigma_r(a) ,L_r(a) .$
\begin{remark}\label{segni}
Assume that, for a given root $r$ of a graph $A$ we have $k$ vertices marked $+$ and $h$ marked $-$. If  we change the root to a root $s$,  and $s$ is marked $+$  we still have $k$ vertices marked $+$ and $h$ marked $-$, if $s$ is marked $-$  we   have $h-1$ vertices marked $+$ and $k+1$ marked $-$.

The embedding changes by the translation by $L_r(s)$  from \eqref{cambio} and \eqref{sitr}
\end{remark}The main reason of this paragraph is the following:
\begin{remark} 
 Let $A$ be a component of the geometric graph,  then to $A$ is associated in a natural way an abstract colored marked graph. If we choose a root and perform the embedding  given by Theorem \ref{limbed}  we see that we have lifted the geometric graph  to  the graph $\Lambda_S$  inverting the map $\pi$.

We shall prove that, under the generic assumption, this lift is an isomorphism to a connected component  of  $\Lambda_S$ over $A$. By abuse of notations we shall denote often by $A$ also this component. All other components over $A$ are obtained from this by translations with elements in $Ker(\pi)$.\end{remark}
\subsection{Realizing the graphs}  
 \begin{definition}
 A realization of an $\mathcal M$--graph is obtained by composition of the embedding, a translation  and a realization as in the previous section  $a\mapsto -\pi(\sigma(a)r-L_r(a))=-\sigma(a)\pi(r)+\pi( L_r(a)).$

A   realization is called {\em integral} if all the vectors $v_i$ have integer coordinates and all the $x_a$ lie in the lattice $\pi(\Z^m)=\mathcal L_v$ generated by the $v_i$.

\end{definition}
\begin{remark}
We have defined a realization by the choice of a root and a translation.  Taking a different root we may obtain the same realization by changing appropriately the translation.
\end{remark}
 
 We    suppose that our graphs are rooted and drop the subscript $r$.
\begin{definition}
Given a graph $A$ we then set for a vertex $a$
\begin{equation}\label{bacos0}  V(r):=0, N(r):=0,\ 
V(a):=\pi(L(a)),
\end{equation}$$\ L^{(2)}(a):=L^{(2)}(L(a)),\ N(a):= \pi(L^{(2)}(a)).$$
\end{definition} Then a  realization  of the graph, given the vectors $v_i$ consists in 
  assigning to each vertex $a$ a vector $x_a\in \C^s$ so that, setting $x:=-\pi(r)$,  the following constraints are verified:
\begin{equation}\label{bacos}
x_a=\sigma(a)x +V(a),\quad (x_a)^2=\sigma(a)(x)^2+N(a).
\end{equation}

 Observe that, if the $v_i$ are integral, the realization is integral if and only if $x\in \mathcal L_v$.

The meaning of a realization is  really a reformulation of the Properties  used in Lemma \ref{matri}.

In fact suppose that we have a path  $p$ from $r$ to $a$ and an edge $e$ from $a$ to $b$.  We have by definition \eqref{lei}
$$L(b)=\begin{cases}
e_j-e_i+  L(a),\quad \text{if}\ e\ \text{is black marked}\ e_j-e_i\\ e_j+e_i-  L(a),\quad \text{if}\ e\ \text{is red marked}\ e_j+e_i
\end{cases} $$ \begin{equation}\label{Maiv}
V(b)=\pi(L(b))=\begin{cases}
v_j-v_i+  V(a),\quad \text{if}\ e\ \text{is black marked}\ e_j-e_i\\ v_j+v_i-  V(a),\quad \text{if}\ e\ \text{is red marked}\ e_j+e_i
\end{cases} 
\end{equation}
\begin{equation}
\label{glix}x_a=\sigma(a)x +V(a),x_b=\sigma(b)x +V(b),
\end{equation}
thus  $$x_b=\sigma(b)x+v_j-v_i+V(a)= \sigma(a)x+v_j-v_i+V(a)=v_j-v_i+x_a $$ $\text{if}\ e\ \text{is black marked}\ e_j-e_i $   $$x_b=\sigma(b)x+v_j+v_i-V(a)= -\sigma(a)x+v_j+v_i-V(a)=v_j+v_i-x_a $$ $\text{if}\ e\ \text{is red marked}\ e_j+e_i .$ For the scalar products \begin{equation}
\label{ivi1}(x_a)^2=\sigma(a)(x)^2+N(a),( x_b )^2=\sigma(b)(x)^2+N(b), 
\end{equation} 

\begin{equation}
\label{enneb}N(b)= \begin{cases}
(v_j)^2-(v_i)^2+  N(a),\quad \text{if}\ e\ \text{is black marked}\ e_j-e_i\\ (v_j)^2+(v_i)^2-  N(a),\quad \text{if}\ e\ \text{is red marked}\ e_j+e_i
\end{cases} 
\end{equation}thus finally we have
\begin{theorem}
The constraints \eqref{bacos}  are equivalent to the recursive identities, (which correspond to {\em simple steps} in the graph $\Gamma_S$):
\begin{equation}\label{Ilsom}
\begin{cases}x_b= v_j-v_i+x_a,\quad 
(v_j)^2-(v_i)^2=( x_b )^2-  (x_a)^2  ,\\ \text{if}\ e\ \text{is black marked}\ e_j-e_i\\\\ x_b=v_j+v_i-x_a,\quad  (v_j)^2+(v_i)^2=  (x_a)^2+( x_b )^2,\\ \text{if}\ e\ \text{is red marked}\ e_j+e_i
\end{cases} 
\end{equation} 
\end{theorem}
 In  particular the definition of realization does not depend on the choice of the root $r$.\medskip
 
 It is convenient to reformulate the equations \eqref{bacos} as:  \begin{equation}\label{bacos1}
  \sigma(a)(x)^2+N(a)=(\sigma(a)x +V(a),\sigma(a)x +V(a))\end{equation} $$=(x)^2+2\sigma(a)(x,V(a))+( V(a)  )^2.$$
  Let  $L(a)=\sum_ha_he_h$ so that $V(a)=\sum_ha_hv_h$. 
 \medskip

 We have:
$( V(a)  )^2-N(a)=\sum_h(a_h^2-a_h)(v_h)^2+2\sum_{h<k}a_ha_k(v_h,v_k). $ Thus
denote    by \begin{equation}
\label{LaNo}C(a):=\sum_h\binom{a_h}{2}(v_h)^2+\sum_{h<k}a_ha_k(v_h,v_k) =\frac12(( V(a)  )^2-N(a)) .
\end{equation} 
  Formula \eqref{bacos1} becomes, using   \eqref{LaNo} 
\begin{theorem}
\begin{equation}\label{bacos2}
   -C(a)= \begin{cases}\ ( x,V(a)) \quad  \quad  \quad  \quad  \ \ \text{if}\ \sigma(a)=+\\
- ( x,V(a))  + (x)^2 \quad \ \ \text{if}\ \sigma(a)=-
\end{cases}  \end{equation}
\end{theorem}
\medskip
  
We need one further definition
\begin{definition}
A graph $A$ is {\em allowed}  if for generic choices of integral vectors $v_i$  it admits an integral realization.
\end{definition}
The word {\em generic} is used  as in algebraic geometry. It means  that the property holds outside the zeros of some polynomial equation in the coordinates of the $v_i$.
\medskip

One of our main tasks is to study such graphs.
\subsection{Relations}
Let us introduce a
\begin{definition}\begin{itemize}
\item A  graph $A$ with $k+1$ vertices is  said to be of {\em dimension} $k$.

\item  The  lattice  $\Z^m_r$  generated by the elements $L(a)$  as $a$  runs over the vertices for a given choice of a root  $r$ is independent of the root. We call its dimension  
the {\em rank}, ${\rm rk} A$,  of the graph $A$.
\item If the rank of $A$ is strictly less than the dimension of $A$ we say that $A$ {\em is degenerate}.
\end{itemize}
\begin{proof}[Proof of item 2]  If we change the root from $r$ to another $a$  we have by Formulas \eqref{IPA} that.
$$L_s(a)=L_r(a)-\sigma(s)\sigma(a)L_r(s).  $$  This shows that  $\Z^m_s\subset   \Z^m_r$ and of course also the converse is true by exchanging the two roles.

\end{proof}
\end{definition}
If $A$ is degenerate  then there are non trivial relations $\sum_an_a\sigma(a)L(a)=0,\ n_a\in\Z$  among the elements $L(a)$.

It is also useful to choose a maximal tree $T$ in $A$. There is   a triangular change of coordinates from the $L(a)$ to the markings of $T$. Hence the relation can be also expressed as a relation between these markings.

 If we are given a realization $\pi:e_i\to v_i$ of the graph we must have, for every relation  $\sum_an_a\sigma(a)L(a)=0,\ n_a\in\Z$ that  $\sum_an_a\sigma(a)V(a)=0$ and, using Formula  \eqref{Heta}
\begin{equation}
0=\sum_{a,\ |\, \sigma(a)=-}n_a.
\end{equation}
Applying Formula \eqref{bacos2}  we deduce that we must have
\begin{equation}\label{riso}
\sum_an_aC(a)=-(x,\sum_an_a\sigma(a)V(a))-[\sum_{a,\ |\, \sigma(a)=-}n_a](x)^2  
\end{equation} hence
\begin{equation}\label{riso1}
\sum_an_aC(a)=-(x,\sum_an_a\sigma(a)V(a))   =0.
\end{equation}
Let us thus set 
\begin{equation}
\label{LaNo1}G(a):=\sum_h\binom{a_h}{2}e_h ^2+\sum_{h<k}a_ha_ke_he_k  =\frac12(L(a)  ^2-L^{(2)}(a)) .
\end{equation} We have $C(a)=\pi(G(a))$. 
Remark  that   also $\sum_an_a\sigma(a)L^{(2)}(a)=0$.

We have
$$\sum_an_a G(a)=\frac12  \sum_an_aL(a)  ^2 -\frac12  \sum_an_aL^{(2)}(a) $$ 
$$ \sum_an_aL^{(2)}(a) = 2\sum_{a,\, \sigma(a)=-}n_a L^{(2)}(a) .$$
and
$$0= \sum_an_aC(a) =\pi\big(\sum_an_aG(a)\big)=\frac12\pi\big(\sum_an_aL(a)  ^2 - \sum_{a\,|\,\sigma(a)=-}n_aL^{(2)}(a)\big) .$$
\begin{definition}
If $\sum_an_aL(a)  ^2 - \sum_{a\,|\,\sigma(a)=-}L^{(2)}(a) \neq 0$ then the equation   \eqref{riso}  is a non trivial constraint, and we say that the graph has an  {\em avoidable resonance}.  
 
\end{definition}
\begin{remark}
If   we have an avoidable resonance then for a generic choice of the $S:=\{v_i\}$  the graph is not realized by $\pi_S$.
\end{remark}
 
We arrive now at the main Theorem of the section:
\begin{theorem}\label{ridma}
Given  a compatible graph of rank $k$ then either it has   $k +1$  vertices or it produces an avoidable resonance.  \end{theorem}
\begin{proof}   Let $q+1$ be the number of vertices. By compatibility the rank of the graph equals the rank of a maximal tree which has   $q$ edges, hence $q\geq k$. 

Assume by contradiction that  $q>k$. Choose a root,   we can choose   $k+1$ vertices $(a_0,a_1,\ldots,a_k) $   so that we have a non trivial relation  $\sum_an_a \sigma(a)L(a_i)=0 $ and the elements $L(a_i),\ i=1,\ldots,k$ are linearly independent.  

We claim that  $\sum_an_aL(a)  ^2 - \sum_{ \sigma(a)=-}n_aL^{(2)}(a) \neq 0$ and thus we have produced an avoidable resonance. Suppose  by contradiction that $\sum_an_aL(a)  ^2 - \sum_{ \sigma(a)=-}n_aL^{(2)}(a) = 0$.
Assume first that all these vertices are marked $+$, we have then $ \sum_an_a L (a)^2=0$.
Similarly, if  they are all marked $-$ we have  $-\sum_an_a L(a)=\sum_an_a\sigma(a)L(a)=0$ and also $ \sum_an_a L^{(2)}(a)=0$ so  again $\sum_an_aL(a)  ^2=0$.

We can consider thus the elements $x_i:=L(a_i),i=1,\ldots,k$ as new variables  and then we write  the relation  as
$$0=L(a_{k+1})+\sum_{i=1}^kp_ix_i,\implies (\sum_{i=1}^kp_ix_i)^2+\sum_{i=1}^kp_ix_i^2=0 .$$
Now  $\sum_{i=1}^kp_ix_i^2 $ does not contain any mixed terms $x_hx_k,\ h\neq k$  therefore  this equation can be verified if and only if  the sum $\sum_{i=1}^kp_ix_i$ is reduced to a single term   $p_ix_i$,  and then we have $p_i=-1$ and $L(a_0)=L(a_i)$, contrary to the second assumption. 
\medskip

Finally assume  we have in the relation $m$   vertices   marked $+$ and  $n$ marked  $-$ . We think of the elements $L(a_i)$ as linear functions in some variables  $y_i$.  
Set  $\underline y=\underline e +\underline y '$,  where $\underline e :=(1,1,\ldots,1)$. 
 Assume $$\sum_{i=1}^ma_iu_i(\underline y)-  \sum_{j=1}^nb_jv_j(\underline y)=0$$
   $$u_i(\underline e)=0,\quad v_j(\underline e)=2\implies \sum_jb_j=0
.$$Now assume 
 $$\sum_{i=1}^ma_iu_i^2 -  \sum_{j=1}^nb_j[v^{(2)}_j -v_j ^2]=0$$
  
For any linear form  $L$,
$$L(\underline y)=L(\underline y')+L(\underline e),\quad 
L^{(2)}(\underline y)=L(\underline e)+2L(\underline y')+L^{(2)}(\underline y'),
$$
  in particular $$u_i(\underline y')=u_i(\underline y ),\quad  v_j(\underline y')=v_j(\underline y )-2.$$$$
v^{(2)}_j(\underline y)-v_j(\underline y)^2=2+2v_j(\underline y')+v^{(2)}_j(\underline y')-(2+v_j(\underline y'))^2=-2-2v_j(\underline y')+v^{(2)}_j(\underline y')-v_j(\underline y')^2.$$ The relation becomes
$$0=\sum_{i=1}^ma_iu_i^2(\underline y')-\sum_{j=1}^nb_j[ -2-2v_j(\underline y')+v^{(2)}_j(\underline y')-v_j(\underline y')^2]$$ $$\implies \sum_{i=1}^ma_iu_i^2(\underline y')-\sum_{j=1}^nb_j[  v^{(2)}_j(\underline y')-v_j(\underline y')^2]$$ $$= \sum_{j=1}^nb_j[ -2-2v_j(\underline y') ]=-2\sum_{j=1}^nb_j v_j(\underline y') .$$ The left hand side is homogeneous of degree 2 and the right of degree 1.
This implies $\sum_{j=1}^nb_j    v_j(\underline y')=0$ and we are back in the previous case.

   \end{proof}
 \section{Main Geometric Theorem\label{Magt}}
\subsection{Determinantal varieties} In this section we think of a marking $\pm v_i\pm v_j$  or more generally of an expression $\sum_{i=1}^ma_iv_i$ as a {\em map} from $V^{\oplus m}$ to $V$.  Here $V$ is a vector space where the $v_i$ belong. Thus a list of $k$ markings is thought of as a map $\rho:  V^{\oplus m}\to V^{\oplus k}$. Such a map is given by a $k\times m$ matrix.

When $\dim(V)=n$  we shall be interested in particular in $n$--tuples of markings. In this case we have
\begin{lemma}
An  $n$--tuples of markings $m_i:= \sum_j a_{ij}v_j$ is formally linearly independent -- that is the $n\times m$ matrix of the $a_{ij}$ has rank $n$-- if and only if the associated   map  $\rho :V^{\oplus m}\to V^{\oplus n}$ is surjective. 
\end{lemma} We may identify  $V^{\oplus n}$ with $n\times n$ matrices and  we have   the determinantal variety  $D_n$ of  $V^{\oplus n}$ defined by the vanishing of the determinant  and formed by all the $n$--tuples of vectors $v_1,\ldots,v_n$ which are linearly dependent. The variety $D_n$ defines a similar irreducible variety $D_\rho:= \rho ^{-1}(D_n)$ in $V^{\oplus m}$ which   depends on  the map $\rho$. We need to see when different lists of markings give rise to different determinantal  varieties in $V^{\oplus m}$.  
\begin{lemma}\label{tras} Given $\rho :V^{\oplus m}\to V^{\oplus n},$
a vector $a\in V^{\oplus m}$ is such that $a+v\in D_\rho,\ \forall v\in D_\rho$ if and only if $\rho(a)=0$.
\end{lemma}
\begin{proof}
Clearly  if $\rho(a)=0$ then $a$ satisfies the condition. Conversely if $\rho(a)\neq 0$, we think of $\rho(a)$ as a non zero matrix  $B$ and it is easily seen that there is a matrix $X\in D_n$  such that  $B+X\notin D_n$.
\end{proof}
Let $\rho_1,\rho_2:V^{\oplus m}\to V^{\oplus n}$ be two surjective maps  given by two $n\times m$ matrices $ A=(a_{i,j}),\ B=(b_{i,j});\ a_{i,j},b_{i,j}\in\mathbb C$ .
\begin{theorem}\label{kke}
$\rho_1^{-1}(D_n)=\rho_2^{-1}(D_n)$ if and only if  the two matrices $A,B$ have the same kernel.
\end{theorem}
\begin{proof} First observe that  the two matrices $A,B$ have the same kernel if and only if $\rho_1,\rho_2 $   have the same kernel.

By Lemma \ref{tras}, if  $\rho_1^{-1}(D_n)=\rho_2^{-1}(D_n)$ then  the two matrices $A,B$ have the same kernel. Conversely if  the two matrices $A,B$ have the same kernel we can write $B=CA$ with $C$ invertible.  Clearly $CD_n=D_n$ and the claim follows.
\end{proof}
We shall also need the following well known fact:
\begin{lemma}\label{zade}
Consider the determinantal variety $D$ given by  $d(X)=0$ of $n\times n$  complex  matrices of determinant zero. The real points of $D$ are  are Zariski dense in $D$.\footnote{this means that a polynomial vanishing on the real points of $D$ vanishes also on the complex points.}
\end{lemma}
\begin{proof}
Consider in    $D$   the set of matrices    of rank exactly $n-1$. This set is dense in $D$ and  obtained from a  fixed matrix (for instance the diagonal matrix $I_{n-1}$ with all 1 except one 0)  by multiplying $AI_{n-1}B$  with $A,B$ invertible matrices.  If a polynomial $f$  vanishes on the real points of  $D$ then $F(A,B):=f(AI_{n-1} B)$  vanishes  for all $A,B$  invertible matrices and real. This set is the set of  points in $\mathbb R^{2n^2}$  where a polynomial (the product of the two determinants) is non zero. But a polynomial which vanishes in all the  points of  $\mathbb R^{m}$  where another  polynomial   is non zero is necessarily the zero polynomial. So $f$ vanishes also on complex points. This is the meaning of  Zariski dense.
\end{proof}

\subsection{Special graphs}  
  Let $V:=\mathbb C^ n $ so $V^{\oplus m}=\C^{mn}$.  Take a connected $\mathcal M$--graph   with $n+2$ vertices, assume that for generic $v_i$ this graph is realizable.  By Theorem \ref{ridma}  this implies that the rank  of this graph is $n+1$.  Choose in this graph a root, then the variety  $R_A$ of realizations is given by the solutions of equations \eqref{bacos2},  which we think as equations in both the variables of the vector $x$ (corresponding to the root) and also of the parameters  $v_i$.   

The variety $R_A$ maps to the space $\C^{mn}$ of $m$--tuples of vectors $v_i$, call $\theta:R_A\to \C^{mn}$ the projection map.  For a given choice of the $v_i$ the fiber of this map $\theta$ is the set of realizations.
 
  \begin{proposition}\label{codim}   Under the previous hypotheses there is an irreducible hypersurface   $W$ of $\C^{mn}$    such that the map $\theta$ is invertible on $\C^{mn}\setminus W$ with inverse a polynomial map.
\end{proposition} 
  Assume    for a moment the validity of Proposition \ref{codim}.

\begin{theorem}\label{aMT}
Consider a  graph $A$ which contains at least  $n+1$ edges   such that the markings are linearly independent, and assume $A$  is allowable. 

Then  for generic $v_i$'s it has a unique  realization   in the special component.
\end{theorem}
\begin{proof}   
Consider the system on $n+1$ linear and quadratic equations in the variables $x,v_i$ defining the variety $R_A$. We are assuming that  we have a solution $x=F(v)$ which is a   polynomial in $v_1,\ldots,v_m$.  A degree consideration shows that $F$ is homogeneous and linear  in these variables.  In fact we have since the right hand side of the equations  \ref{bacos2} are quadratic, we have $ F(\lambda v)=\lambda F(v)$. 

Now the equations \ref{bacos2}  are invariant under the action of orthogonal matrices, i.e. if $A$ is orthogonal $F(Av_1,\ldots,Av_m)=AF(v_1,\ldots,v_m)$.  Since the space $V$  of the $v_i$ (which we may take as complex) is irreducible under the orthogonal group, a linear map $V\to V$ commuting with  the orthogonal group is a scalar  so it follows that any linear map $V^{\oplus m}\to V$  commuting with the orthogonal group is of the form
$F(v_1,\ldots,v_m)=\sum_{a=1}^mc_av_a$ for some constants $c_a$.

Now $x=\sum_{a=1}^mc_av_a$  is the point of the realization corresponding to the root and so it satisfies  either $(x,v_j-v_i)=(v_i, v_j-v_i)$  for some $i,j$ 
so $x=v_i$ or the quadratic equation $(x-v_i,x-v_j)=0$ from which $x=v_i$ or $v_j$.

Once we know that one point in the realization is in the special component we have proved (see \ref{spco})  that the whole   tree is special and realized in this component.\end{proof}

\subsection{Proof of Proposition \ref{codim}, black edges\label{codibl}}
  Let us first study the case of all black edges.  The next is standard and follows immediately from the unique factorization property of polynomial algebras:
  \smallskip
  
 {\bf Theorem} 
{\it Let $W$ be a subvariety  of $\mathbb C^N$ of codimension $\geq 2$, let $F$ be a rational function on  $\mathbb C^N$ which is holomorphic on $\mathbb C^N\setminus W$,  then $F$ is a polynomial. }
\medskip

  We  have a list of $n+1$--linear equations  $(x,a_i)=b_i$ with the markings $a_i=\sum_j a_{ij}v_j$ formally linearly independent.  The hypotheses made imply that any $n$ of these equations are generically linearly independent. Call $C$ the matrix with rows the vectors $a_i$.

  Therefore on the open set where $n$ of these are linearly independent the solution to the system is unique and given by Cramer's rule.
  
  In order to complete our statement it is enough to show that the subvariety $W$ where any $n$ of these equations are   linearly  dependent has codimension $\geq 2$.  The condition to be in $W$ is that all the determinants of all the maximal minors should vanish.
  
  Each one of these determinants is an irreducible polynomial so it defines an irreducible hypersurface. It is thus enough to see that these hypersurfaces are not all equal. This follows from Theorem \ref{kke}, indeed by hypothesis the matrix $B=(a_{ij})$ has rank $n+1$. All the matrices obtained by $B$ dropping one row define the various determinantal varieties, the fact that these varieties are not equal depends on the fact that the matrices cannot have all the same kernel (otherwise the rank of $B$ is $\leq n$).
  
\subsection{Proof of Proposition \ref{codim}, red edges}
When we also have red edges we see that the equations\eqref{bacos2}  are clearly equivalent to a system on $n$ linear equations associated to formally linearly independent markings, plus a quadratic equation chosen arbitrarily among the ones appearing in \eqref{bacos2}.  We then put as constraint the non vanishing of the determinant of the linear system  we have found. Thus a realization is obtained by solving this system and, by hypothesis, such solution satisfies also the quadratic equation.   
 
 Let $P$ be the space of functions $\sum_{i=1}^mc_iv_i,\ c_i\in\mathbb R$ and $(P,P)$ their scalar products.
Assume we have a list of $n$ equations  $\sum_{j=1}^ma_{ij}(x,v_j)=(x,t_i)=b_i$  with the $t_i=\sum_{j=1}^ma_{ij} v_j$ linearly independent  in the space $P$ and $b_i=\sum_{h,k}a^i_{h,k}(v_h,v_k)\in (P,P)$.

Solve these equations  by Cramer's rule    considering the $v_i$ as parameters. Write  $ x_i=f_i/d $, where  $d(v):=\det(A(v))$ is the determinant of the matrix $A(v)$ with rows $t_i$,     $f_i(v)$ is also a determinant of another matrix $B(v)$ both depending polynomially on the $v_i$. We have thus expressed  the coordinates $x_i$ as rational functions  of the coordinates of the $v_i$. The denominator is an irreducible polynomial  vanishing exactly on the determinantal variety of the $v_i$ for which the matrix of rows $t_j,\ j=1,\ldots,n$ is degenerate.
 \begin{lemma}
\label{aMT1} Assume there are two elements $a\in P, b\in(P,P)$ such that $(x)^2+(x,a)+b=0 $ holds identically (in the parameters $v_i$);
then $x$ is a polynomial in the $v_i$.  \end{lemma}
\begin{proof} Substitute  $x_i=f_i/d$ in the quadratic equation and get 
$$d^{-2}(\sum_if_i^2)+d^{-1}\sum_if_ia_i+b=0,\implies \sum_if_i^2 +d \sum_if_ia_i+d^{ 2} b=0.$$
Since $d=d(v)=\det(A(v))$ is irreducible this implies that $d$ divides $\ \sum_if_i^2.$

For those $v_i\in\R^n$ for which   $d(A(v))=0$, since the  $f_i$ are real  we have $f_i(v)=0,\forall i$, so $f_i$ vanishes on all real solutions of $d(A(v))=0$.  These solutions are Zariski dense, by Lemma \ref{zade},  so $f_i(v)$ vanishes on all the  $v_i$ solutions of $d(A(v))=0$ and $d(v)$  divides $f_i(v)$, hence  $x $ is a polynomial.

\end{proof}
\section{Proof of  Theorems \ref{oneone} and \ref{sunto}}
   \subsection{The resonance inequalities\label{therin}}
    We   are now ready to explain which  restrictions we impose on    the vectors $v_i$  in  order to say that the vectors have been chosen   {\em   in a generic way}.    
\begin{enumerate}[1)]\item We impose Constraint \ref{co1}.
\item We require that any $n$ of the $v_i$ are linearly independent.

\item We want that  any linear combination of  the $v_i$ arising from an odd circuit of length $\leq n$  is non--zero.  

\item    We list all  degenerate  graphs with $k+1,\ k\leq n+1$  vertices, we list all the avoidable resonance equations and impose that the vectors $v_i$ do not satisfy any of these equations, this is a set of quadratic inequalities. 

\item Finally for each graph with $n+1$  formally  linearly independent markings, we impose the following relations. From equations   \ref{bacos2} we select $n$ linear equations and  impose that  the corresponding determinant  is non--zero.   This allows us to apply Theorem \ref{main}.
\end{enumerate} We are now ready to prove Theorems \ref{oneone} and \ref{sunto}.
  The first four constraints control the components  with at most $n$  linearly independent edges. 
 
 The fifth constraint   handles  graphs with $n+2$ vertices.  
       
            Let us consider a connected component  $A$ of the graph $\Gamma_X$. If we choose an element $r\in A$ and a $\mu\in\Z^m$ with $r=-\pi(\mu)$  we see that
       \begin{lemma}
The connected component $B$ of the graph $\Lambda_S$  containing $\mu$ maps surjectively  to $A$.
\end{lemma}   
   \begin{proof}
Clearly  any edge in $A$ can be lifted to an edge in $B$.
\end{proof}    

Theorem \ref{oneone} follows from the more precise
 \begin{proposition}\label{bah}
If $A$ is different from the special component, then $B$ has rank  $\leq n$ and is non degenerate and $-\pi$ maps $B$ isomorphically to $A$.
\end{proposition}      
 \begin{proof}
Assume  first that $B$ has rank  $\leq n$, we have  that $B$ has to be non--degenerate since we have imposed that the $v_i$ do not satisfy the  resonance equations. In particular $B$ has at most $n+1$ vertices. Then the constraints that we have required imply first that $-\pi$ is bijective on the vertices, in fact if we had to vertices in $B$ which map to the same vertex $k$ in $A$ we  consider a path connecting them we have that either $2k=\sum_ia_iv_i$  if the path is odd and then we can exclude this  by imposing that the quadratic  equation satisfied by $k$ is not satisfied by the $v_i$, or we get  $\sum_ia_iv_i=0$  and again we may exclude this, in fact since $B$ varies on graphs with at most $n+1$ vertices all these are a finite number of inequalities.  The isomorphism at the level of edges follows from Corollary \ref{lacco}.

So it only remains to show that for generic $v_i$ we cannot have  $B$ of rank $\geq n+1$. This follows from Theorem \ref{aMT}.
\end{proof}      
\begin{proof}[Proof of Theorem \ref{sunto}]
{\it i)}\; Let $A$ be a connected component of the graph $\Gamma_S$, let $B$ be a corresponding component in $\Lambda_S$ so that $A=-\pi(B)$. By proposition \ref{bah} $B$ has at most $n+1$ vertices.

{\it ii)}\;Two points are connected by a red edge if they belong to a sphere $S_{ij}$ with diameter $v_i,v_j$ and there are finitely many integral points on such a sphere.

{\it iii)} \; Each component of $\Gamma_S$ is a realization of an abstract marked graph, which encodes the equations that the root $x_r$ should satisfy. In the case of black edges the equations are all linear and the general solution is given by adding the solutions of the associated homogeneous system.
 Finally, there are only a finite number of abstract marked graphs with at most $n+1$ vertices.

\end{proof}
\begin{remark}
In this general discussion we have  restricted the possible types of components that we can have in the geometric graph $\Gamma_X$.  In fact it may very well be that some of the possible components which are allowed do not appear.

This depends upon the fact   that our conditions are algebraic and we only claim that a certain system of equations has a solution.  But, for contributing a component to the graph $\Gamma_X$,  this solution must lie in the lattice spanned by the $v_i$.  This we have not tried to discuss.  In fact in dimension 2, in the paper \cite{GYX}, Geng, You and Xu use arithmetic conditions to exclude all graphs with 3 vertices.
\end{remark} 
 \subsection{$m=\infty$}

It is easy to see that we can also construct infinite sequences of integral vectors $v_i$ satisfying all constraints. For this recall the known, fact. Let $f(x_1,\ldots,x_p)$  be a polynomial with integer coefficients. Assume that all coefficients of the polynomial are $<C$ in absolute value, and all exponents are $<D$, with $C,D$ two positive integers.  Consider the sequence $\mathcal L:=a_i:=C^{D^i}<\ i=1,2,\ldots $. Then
 \begin{lemma} For every choice of  $p$ distinct elements $a_{i_1},\ldots, a_{i_p}$ in this list we have
 $$f(a_{i_1},\ldots, a_{i_p})\neq 0.  $$
\end{lemma}
\begin{proof}
Consider the monomials  appearing   in $f$.
$$x_1^{h_1}x_2^{h_2}\ldots x_p^{h_p}\mapsto C^{\sum_{j=1}^p h_jD^{a_j}}.$$

Since all indices $h_i<D$ we have that different monomials of $f$ give rise to a different exponent $\sum_{j=1}^p h_jD^{a_j}$.

Now the polynomial gives a linear combination of integers $a_i$ with $|a_i|<C$ (the coefficients) times distinct elements  $C^{d_i}$. By the uniqueness of the expression of a number in a given basis  we deduce the claim.
\end{proof}

We now apply this to our setting, we take $C,D$  bigger than the coefficients of all constraints  in dimension $n$, similar bigger than the exponents  in these constraints.

If we now partition in any desired way the list   $\mathcal L$  into disjoint sublists each made of $n$ elements they form an allowable infinite list.  
 
\section{The Matrices\label{matrici}}
{\sl In this and the following sections we discuss the combinatorial and algebraic features of the matrices appearing in the blocks of $ad(N)$ in order to complete the proof of Theorems \ref{LIRRI},  \ref{separ} and Lemma \ref{SEPA}.}
 \subsection{Combinatorial matrices} We now discuss the matrices $C_{\A}$ introduced in Definition \ref{Cacomb}.
  The vertices  of  ${\A}$ index a basis of  this  block. 
Given a vertex $(a,\sigma)$ if $a=\sum_im_ie_i$ we set $$a(\xi):=\sum_im_i\xi_i.$$  
 \begin{lemma}\label{allowa}
The entries of the matrix $C_{\A}$, over the  indexing set of the vertices of ${\A}$, are:

\begin{itemize}\item In the diagonal at the vertex $(a,\sigma)$ equals $-\sigma  a (\xi)$.\item At the position  $(a,\sigma),(b,\tau)$  we put 0 unless they are connected by an oriented edge $e=((a,\sigma),(b,\tau))$  marked with $(i,j)$.  In this case we  place  \begin{equation}
\label{ilc}C(e):=2\tau  \sqrt{\xi_i\xi_j}. 
\end{equation}
\end{itemize} \end{lemma}
\begin{proof}
Let $(\nu,+)\in\A$ correspond to the root $x(A)=-\pi(\nu)$. Take another element $a=(\mu,\sigma)\in\A$ such that $-\pi(\mu)=x_a\in A$ and $\sigma(a)=\sigma$. From Formula \ref{bacos0} and \ref{bacos} and Proposition \ref{bah} we have 
$\mu= \sigma\nu-L(a) $ hence $-\sigma a(\xi)=-\sigma(\xi,\mu)= -(\xi,\nu)+\sigma(\xi,L(a))$ which is the diagonal entry of $C_\A$ by Formula \eqref{CAab} and Definition \ref{Cacomb}.
\end{proof}
We need to see the behavior under the symmetric group, translation and sign change. Under the symmetric group we just permute the variables  $\xi_i$.
\begin{theorem}
\begin{equation}
C_{\tau_c({\A})}=c(\xi)I+C_{\A},\quad C_{\bar {\A}}=-C_{\A}.
\end{equation}
\end{theorem}\begin{proof}
If we translate ${\A}$, to $\tau_c({\A})$  the edge   $e=((a,\sigma),(b,\tau))$ becomes $\tau_c(e)=((a+\sigma c,\sigma),(b+\tau c,\tau))$ so 
$$C(\tau_c(e))= C(e),\quad \sigma (a+\sigma c)(\xi)= \sigma (a )(\xi)+  c (\xi).$$  Similarly for sign change. 
\end{proof}  
 We denote by
$$\chi_{\A}(t):=\det(t-C_{\A})$$  the characteristic polynomial of $C_{\A}$.

 \begin{remark}
Notice that,   if $e$ is red,   in position $(b,a)$ we have  $\sigma(a)2\sqrt{\xi_i\xi_j}=- \sigma(b)2\sqrt{\xi_i\xi_j}$.

If we change the root $s$ the matrix is changed as $C_{\A}^s=L(s)(\xi)+\sigma(a) C_{\A}^r$. 
\end{remark} 
 
 In particular  $C_{\A}$ is symmetric if and only if there are no red edges.
 
 These are the matrices appearing in our Hamiltonian, but we can immediately change them as follows. Choose  a maximal tree in the graph and a root, then every vertex is connected by a unique minimal path.   Given a vertex $a$  and the minimal path $p=(r=a_0,a_1,\ldots,a_k=a$ from the root $r$ to $a$ and we set  $k=\ell(a)$. We next set
 \begin{equation}
D(r):=1,\ D(a):= \prod_{i=0}^{k-1}C(a_i,a_{i+1})
\end{equation}
  
This defines a diagonal matrix  $D$.    
 \begin{proposition} Set $\tilde C_{\A}:=DC_{\A}D^{-1}$:
\begin{enumerate}
\item If $(a,b)$ is an edge in the tree  we have 
\begin{equation}
\tilde C_{\A}(a,b)=\begin{cases}
\sigma(b)\ \text{if}\   \ell(b)=\ell(a)+1\\
C(a,b)^2=4\xi_i\xi_j\ \text{if}\   \ell(b)=\ell(a)-1
\end{cases}
\end{equation}
\item If $e=(a,b)$ is not an edge in the tree (but it is an edge in the graph), we have that $\tilde C_{\A}(a,b)$  is  a constant times a monomial in the variables $\xi_i^{\pm 1}$.
\end{enumerate}
\end{proposition}
\begin{proof}
(i)  This  is from the definition.

(ii) This comes from the circuit that shows that modulo 2  the two elements $L(a)=L(b)$ which takes away the squares.
\end{proof}
 \begin{corollary}\label{poli}
For every allowable ${\A}$  the characteristic polynomial of $C_{\A}$  has as coefficients polynomials in the variables $\xi_i$. 
\end{corollary}

By the previous Theorem the square roots disappear.  \bigskip

  {\bf Conjecture 1}\   For every allowable ${\A}$  the characteristic polynomial $\chi_{\A}(t)$ of $C_{\A}$     is irreducible as polynomial with coefficients  in $\C[\xi_1,\ldots,\xi_m]$. 
 
It is clear that the irreducibility property is invariant under all the symmetries, the symmetric group, translation and sign change, so the statement needs to be checked only for finitely many  ${\A}$.  In the next section we discuss dimension 3. We  take as ${\A}$ always a graph containing 0 and we assume that 0 has sign +, (cf. Definition \ref{cgli}).

\subsection{The method}
We have seen in Theorem \ref{sunto} that the graphs we need to consider are complete subgraphs ${\A}$ of the graph   $\Lambda_S$,  constructed as follows. Take  a linearly independent list of  $k$ vectors $\nu_i$ where $k$ is the rank.  Consider the complete graph ${\A}$ generated by $0,\nu_1,\ldots,\nu_k$.

We need to study those ${\A}$ which are connected. 
In Theorem \ref{sunto}  we have seen that the connected components of $\Lambda_S$ are obtained from such a graph by translation and sign change. 

The goal of the rest of the paper is to prove the two Conjectures which allow us to deduce the separation of eigenvalues and thus the second Melnikov equation.  We have at our disposal several theoretical tools which at the moment are not sufficient to treat  all cases. This is why the possible obvious inductions are not available at the moment and we need to perform a rather tedious detailed case analysis.  In principle this can be checked by a finite algorithm in all dimensions but we have not tried to write the  necessary code.

The main ingredient of a possible induction is:
\begin{theorem}\label{lafatt}
Take an allowed graph ${\mathcal A}$  and compute its characteristic polynomial $\chi_\A(t)$. When we  set a variable $\xi_i=0$  we obtain the product  of the polynomials $\chi_{\A_i}(t)$  where  the $A_i$ are the connected components of the graph obtained from $\A$ by deleting all the edges in which $i$ appears as index, with the induced markings (with $\xi_i=0$).
\end{theorem}  
\begin{proof}
This is immediate from the form of the matrices.
\end{proof}

\bigskip

  \subsection{Some tests} Let us make several remarks on Conjecture 1, which we are going to prove in dimension 3 by a case analysis.
 
 First we analyze trees, and for each marked tree we complete it according to Theorem \ref{sunto}. We use systematically Theorem \ref{lafatt}  and also  we may sometimes use the following simple\begin{remark}[Minimality]  If  setting two or more of the variables $\xi_i$ equal we have an irreducible polynomial then the one we started with is irreducible.

\end{remark}
 
 \begin{lemma}[Parity test]\label{parita}  \begin{enumerate}
\item If we compute  $t$ at an odd number $g$, the determinant  $\chi_A(g)\neq 0.$

\item If a linear form $t+\sum_ia_i\xi_i,\ a_i\in\Z$ divides $\chi_A(t)$ we must have $\sum_ia_i$  is even.
\end{enumerate}
\end{lemma}
\begin{proof}
i)   We compute modulo 2 and  set all $\xi_i=1$, recalling that $L_a(\xi )|_{\xi_i=1}=0,2$. We get  $\chi_A(t) =t^m$.

ii) A linear form $t+\sum_ia_i\xi_i,\ a_i\in\Z$ divides $\chi_A(t)$ if and only if  we have $\chi_A(-\sum_ia_i\xi_i)=0$, then set $\xi_i=1$ and use the first part.
\end{proof} 

We shall use the parity test  as follows.\begin{theorem}\label{supertest}
  Suppose we have a graph $A$, in which we find a vertex  $a $ and an index, say 1,  so that  
 $$ \xymatrix{&&c&\\ \ldots &d&a\ar@{.}^{1,h}[r]  \ar@{.}^{1,h}[r] \ar@{.}^{1,k}[l]\ar@{.}^{1,j}[d]\ar@{.}^{1,i}[u]&b \ldots   &\ldots \\&&e&\  }$$
  we have:
\begin{itemize}
\item $1 $ appears in all and only the edges  having $a$ as vertex.
\item When we remove $a$ (and the edges meeting $a$) we have    a connected graph ${\A}$ with at least 2 vertices.
\item When we remove the edges associated to any index, the factors described in Theorem \ref{lafatt}  are irreducible.
\end{itemize}
  Then the polynomial associated to $A$ is irreducible.
\end{theorem}
\begin{proof} 
We take $a$ as root, setting $\xi_1=0$ we have by Theorem \ref{lafatt} and the hypotheses, that  $\chi_A(t)=t\,  P(t)$  with $P=\chi_{\A}(t)$ irreducible of degree $>1$.   Thus, if the polynomial  $\chi_A(t)$ factors, then it must factor into a linear $t-\ell(\xi)$ times an  irreducible polynomial of degree $>1$.

Moreover modulo $\xi_1=0$ we have that $0$ and $\ell$ coincide, thus $\ell$ is a multiple of $\xi_1$.

Take another index $i\neq 1,h$ if $a$ is an end and the only edge from $a$ is marked $(1,h)$ otherwise just different from $1$ and   set $\xi_i=0$. Now    the polynomial $\chi_A(t)$ specializes to the product $\prod_j\chi_{A_j}(t)$ where the $A_j$ are the connected component of the graph obtained from $A$ by removing all edges in which $i$ appears as marking. By hypothesis $\{a\}$ is not one of the $A_j$.

If no factor is linear we are done. Otherwise there is an isolated vertex $d\neq a$   so that $\{d\}$ is one of the connected components $A_j$. The linear factor associated is $t-\sigma(d)L_d(\xi)|_{\xi_i=0}$.  Clearly    we have that the coefficient  of  $\xi_1 $ in  $L_d(\xi)$ is $\pm 1$ (since the marking 1 appears only once).    This implies that $\ell=t\pm\xi_1$  and this is not possible by the parity test.\end{proof}

This Theorem can be used as the basis of a possible, induction. Let us analyze this.

Take a graph  $A$ of dimension and rank $n$. Thus it has  $n+1$ vertices. Assume that, by some inductive procedure, we know that, if  we remove the edges associated to any index, the factors described in Theorem \ref{lafatt}  are irreducible.  If  there is an index $i$ such that, when we remove the edges containing $i$ the graph remains connected, we are done.  Similarly if we can find  an index satisfying the conditions of Theorem \ref{supertest}.

 \section{Dimension 3\label{irrid}}
\subsection{Bases and encoding graphs\label{imarki}}
We first classify the   graphs by rank and up to the symmetry induced by permuting the variables $e_i$ (that is the action of the symmetric group $S_m$).

Thus in order to classify  these graphs, the first step is to understand the possible lists $\nu_i$ which produce a connected complete graph.

A choice of a maximal tree in each such graph determines, through its edges,   a linearly independent list  extracted from the vectors  $E:=\{e_i\pm e_j\}$.  Thus we may start by first classify up to symmetry such lists of rank $k$. In dimension $n$  will appear lists of rank $k\leq n$.

Since we are trying to classify the graphs up to equivalence we have some freedom in choosing the root. In a tree with $n$ vertices  we can always choose a vertex $r$ as root so that every other vertex is at distance at most $[n/2]$, thus each  possible $\nu_h$ is  obtained by adding up at most $[n/2]$ elements in the list $e_i\pm e_j$.

This gives a finite algorithm which is still computationally very heavy even in dimension 3.

Let us start by explaining how to classify the lists of independent edges  for dimension 3.

Observe first that $E:=\{e_i\pm e_j\}$  decomposes under the symmetric group into 2 orbits, the {\em black} and {\em red} edges.

In a list it is first convenient to count the two numbers $e,f$ of black and red edges, and we may have for a list with 3 elements 4 possibilities $(3,0),(2,1),(1,2),(0,3)$.
It is convenient to display the list by its {\em encoding graph}.  This is the subgraph of the full graph with vertices the numbers $1,\ldots,m$ formed by the edges $(i,j)$ which appear as markings in the graph $A$.  We can also color and orient these edges.

When we take 3 linearly independent markings,  their encoding graph   up to permutation of the variables can have the following different combinatorial structure
$$Type\  0: \xymatrix{2&4& 6&\\1\ar@{-}[u] &3\ar@{-}[u]  &5 \ar@{-}[u] },\quad (1,2),(3,4),(5,6)$$\smallskip
    
$$Type\  1: \xymatrix{&1& 5& \\3\ar@{-}[r] &2\ar@{-}[u]  &4 \ar@{-}[u] },\quad (1,2),(2,3),(4,5)$$\smallskip

    $$ Type\  2: \xymatrix{&1& &\\3\ar@{-}[r] &2\ar@{-}[u]  \ar@{-}[r] &4 },\quad (1,2),(2,3),(1,4)$$\smallskip
    
    $$ Type\ 3: \xymatrix{ 1\ar@{-}[r] &2\ar@{-}[r] &3 \ar@{-}[r] &4 },\quad (1,2),(2,3),(3,4)$$\bigskip
    In all these cases we may have any choice of red edges.
    
    We now discuss two special cases.
   
\qquad\quad\  $ Type\ A: $\label{tipoA} \quad    We may have at the same  time an edge, which we may assume to be  $(1,2)$ black and   red, plus another edge, in general disjoint $(3,4)$ but after specialization it can be assumed to be  $(2,3)$  black or red.
\begin{figure}[!ht]
\centering
\begin{minipage}[b]{11cm}
\centering
{
\psfrag{1}{$1$}
\psfrag{2}{$2$}
\psfrag{3}{$3$}
\psfrag{4}{$4$}
\psfrag{I}{I}
\psfrag{II}{II}
\includegraphics[width=11cm]{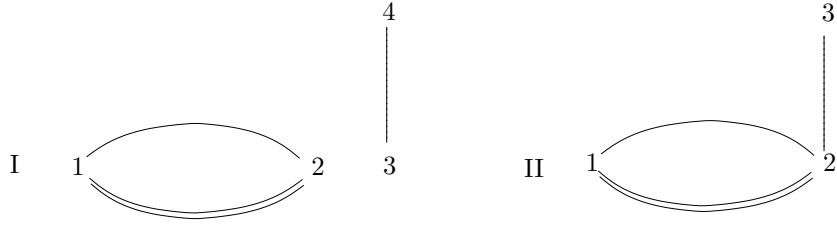}
}
\caption{\footnotesize{ The dotted line may be either red or black.}}\label{fig2}
\end{minipage}
\end{figure}

Finally we may have an {\em odd circuit} (notice that an even circuit gives linearly dependent markings)
 \begin{equation}
\label{typeB}Type\ B: \xymatrix{ &3& &\\I\quad 1\ar@{=}[rr] \ar@{-}[ur]<-5pt>&  & 2\ar@{-}[ul] } \  \xymatrix{ &3&  \ (1,2),(2,3),(3,1)&\\II\quad 1\ar@{=}[rr] \ar@{=}[ur]<-5pt>&  & 2\ar@{=}[ul] }
\end{equation} \smallskip
    
 Type BII  cannot occur since, if we have tow adjacent red edges with a common index in the label, this configuration has to be completed with a black edge.

    In case BI the encoding graph is  completely described  by  the choice of orientation for the black edges  (and again we have some symmetries to consider).

It is clear that  {\em specializing some variables}  one can pass from   more general types to others.  When we pass to analyzing the possible allowable graphs, the 3 markings come from a maximal tree and the actual graph may need to be completed with further edges.

    \bigskip

\subsection{Irreducibility tests}\quad 

{\sl  There are standard algorithms that check if a polynomial with integer coefficients is irreducible. In our analysis we try to avoid them as much as possible in order to give a possible general approach, nevertheless for a few cases we could not find a theoretical explanation so we just verified the irreducibility with these algorithms.}
\bigskip

If $A$ has $s$ vertices the characteristic polynomial is  of degree $s$. For the markings, we   restrict to the case of rank $s-1$   (by Theorem \ref{ridma}). Let us start by considering:\medskip
 
  $s=2$. \quad In this case $A$  is a single edge that we may assume marked by $(1,2)$  or equivalently $(2,1)$.  The corresponding matrix is, in the black and red case:
  
  \begin{equation}\label{22}\begin{vmatrix}
t &-2\sqrt{\xi_1\xi_2}\\&&&\\-2\sqrt{\xi_1\xi_2}&t+\xi_1- \xi_2  
\end{vmatrix},\quad \begin{vmatrix}
t &  2\sqrt{\xi_1\xi_2}\\&&&\\ -2\sqrt{\xi_1\xi_2}&t+\xi_1+ \xi_2  
\end{vmatrix} $$ with determinants
$$ t  ^2+t(\xi_1-  \xi_2 ) -4\xi_1\xi_2,\ t  ^2+t(\xi_1+ \xi_2 ) +4\xi_1\xi_2\end{equation}
 
it is easily seen that these polynomials in $t,\xi_1,\xi_2$ are   irreducible.
  $s=3$. \quad 
   \begin{lemma}\label{adia}
Unless we are in the special component we cannot have in the graph adjacent $ \xymatrix{ a\ar@{-}^{1,2}[r]& b \ar@{=}[r] ^{1,2} &c }$ or $ \xymatrix{a\ar@{=}[r] ^{1,2}&b\ar@{-}[r]^{2,1} &c }$.
\end{lemma}
\begin{proof}
We use the fact that $(H_{1,2}\cup H_{2,1})\cap S_{1,2}=\{v_1,v_2\}$.
\end{proof}

 $s=3$. \quad  For a graph $A$ with 3 vertices $\chi_A(t)$ has degree 3,  we only have two possibilities, up to coloring the edges.

$$   \xymatrix{& & &\\ a\ar@{-}[r]  &b\ar@{-}[r] &c }  \xymatrix{&c& &\\ a\ar@{-}[ru] \ar@{-}[rr] &&b\ar@{-}[lu]  &  }. $$ 

   We may have,  once we specialize  the variables,  only the case   $(1,2),(2,3)$  on the tree
with various colors and orientations. Some of these choices do not give a complete graph. We complete it (this determines uniquely the oriented edges to be added)      we have the circuit  which must be even by the rank assumption and up to symmetry \begin{equation}
\label{icircui}\xymatrix{&c& &\\ a\ar@{<-}[ru] ^{3,1}\ar@{->}[rr]_{1,2} &&b\ar@{->}[lu]_{2,3}  &  }, \xymatrix{&c& &\\ a\ar@{<-}[ru] ^{3,1}\ar@{=}[rr]_{1,2} &&b\ar@{=}[lu]_{2,3}  &  }.
\end{equation} 

All these cases fall under the requirements of Theorem \ref{supertest} and thus give irreducible polynomials. 
 \bigskip

 $s=4$  The possible colored trees with 4 vertices are
\bigskip

\begin{itemize}\item {\bf Star types}
$$\boxed{a}: \xymatrix{  &c &  &\boxed{b}: &&c&   &&  & \\a\ar@{-}[r]&b\ar@{-}[u]  \ar@{-}[r] &d &&a\ar@{-}[r]&b\ar@{=}[u]  \ar@{-}[r] &d &  }$$
$$   \xymatrix{\boxed{c}: &c& &\\a\ar@{-}[r]&b\ar@{=}[u]\ar@{=}[r] &d }\quad    \xymatrix{\boxed{d}: &c& &\\a\ar@{=}[r]&b\ar@{=}[u]  \ar@{=}[r] &d }$$
\item {\bf Linear types}
$$ \xymatrix{\boxed{e}:  a\ar@{-}[r]&b\ar@{-}[r]&c  \ar@{-}[r] &d } \xymatrix{ \boxed{f}: a\ar@{-}[r]&b\ar@{-}[r]& c \ar@{=}[r] &d }$$$$   \xymatrix{\boxed{g}: a\ar@{-}[r]&b\ar@{=}[r]& c \ar@{=}[r] &d }\xymatrix{\boxed{h}: a\ar@{=}[r]&b\ar@{-}[r]& c \ar@{=}[r] &d }$$$$
  \xymatrix{\boxed{i}:  a\ar@{-}[r]&b\ar@{=}[r]& c \ar@{-}[r] &d }\quad    \xymatrix{\boxed{j}:  a\ar@{=}[r]&b\ar@{=}[r]& c \ar@{=}[r] &d }$$
\end{itemize}
For each of these trees we must analyze the markings which can be classified according to their encoding graphs. We easily see that, for the encoding graphs of types  0,1,2,3  there is always an index of a marking which appears only once and in an end edge. Thus all these cases fall under Theorem \ref{supertest}.

Type A cannot occur in the star type and in cases e, h,i,j by Lemma \ref{adia}.

The   analysis is still quite long. We can simplify it a bit  if we weaken our requirements and try to prove only that the polynomials are irreducible over $\Z$. 
This is in any case sufficient for the applications we have in mind. We can set  all variables $\xi_i=\frac{y}{\sqrt{2}}$.  In  fact when we do this for types  $a$ through  $d$ we always have a linear by an irreducible cubic polynomial. This is not a useful  information but for the other cases we have some reductions.
 For the graphs of type $ e $ 
       the matrix becomes
 $$\begin{vmatrix}
t&-y&0&0\\-y&t&-y&0\\0&-y&t&-y\\0&0&-y&t
\end{vmatrix}  $$   with determinant
$$ t^4 - 3\,t^2\,y^2 + y^4=  \left( t^2 - t\,y - y^2 \right) \,
   \left( t^2 + t\,y - y^2 \right)  $$
     For colored case   we set $\xi_i= y$ and we have as possible matrices 
       $$\boxed{f}\qquad \xymatrix{ a\ar@{-}[r]&b\ar@{-}[r]& c \ar@{=}[r] &d^{-} },\quad \begin{vmatrix}
t & -2y & 0 & 0 \\ -2y & t & -2y & 0 \\ 0 & -2y & t &  2y \\ 0 & 0 & -2y & t  +2y
\end{vmatrix}$$
with determinant   
$$ 
t^4 + 2\,t^3\,y - 4\,t^2\,y^2 - 16\,t\,y^3 - 16\,y^4
=
\left( t + 2\,y \right) \,\left( t^3 - 4\,t\,y^2 - 8\,y^3 \right) $$
$$ \boxed{g}\qquad \xymatrix{ a\ar@{-}[r]&b\ar@{=}[r]& c^{-} \ar@{=}[r] &d^{+} },\quad \begin{vmatrix}
t & -2y & 0 & 0 \\- 2y & t & 2y & 0 \\ 0 & -2y & t +2y& -2y \\ 0 & 0 & 2y & t  
\end{vmatrix}$$ with determinant the irreducible polynomial  
$$  t ^4 + 2 t ^3 y + 4  t ^2  y ^2 - 8  t y ^3 - 16  y ^4  $$
  $$\boxed{h}\qquad\xymatrix{  a\ar@{=}[r]&b^-\ar@{-}[r]& c^- \ar@{=}[r] &d^+ } ,\quad \begin{vmatrix}
t & 2y & 0 & 0 \\ -2y & t+2y &  2y & 0 \\ 0 &   2y & t  +2y& -2y \\ 0 & 0 &   2y & t 
\end{vmatrix}$$  with determinant the  irreducible polynomial 
$$t^4 + 2\, t^3\, y + 4\, t^2\, y^2 + 8\, t\, y^3 + 16\, y^4$$ 
$$\boxed{i}\qquad  \xymatrix{ a\ar@{-}[r]&b^{+}\ar@{=}[r]& c^{-} \ar@{-}[r] &d^{-} },\quad \begin{vmatrix}
t & -2y & 0 & 0 \\  -2y & t & 2y & 0 \\ 0 & -2y & t +2y& 2y \\ 0 & 0 &     2y & t +2y 
\end{vmatrix}$$ with determinant the  polynomial 
$$  
t^4 + 4\,t^3\,y - 8\,t\,y^3 =
t\,\left( t + 2\,y \right) \,\left( t^2 + 2\,t\,y - 4\,y^2 \right) $$
$$\boxed{j}\qquad  \xymatrix{ a\ar@{=}[r]&b^-\ar@{=}[r]& c^+ \ar@{=}[r] &d^- },\quad \begin{vmatrix}
t & 2y & 0 & 0 \\ -2y & t+2y & -2y & 0 \\ 0 &  2y & t & 2y \\ 0 & 0 &   -2y & t  +2y
\end{vmatrix}$$ with determinant the irreducible polynomial 
$$ t^4 + 4\, t^3\, y + 16\, t^2\, y^2 + 24\, t\, y^3 + 16\, y^4$$

   this shows that,  \begin{lemma}\label{cas1}\begin{itemize}\item If for a given marking of the tree of type $e$ the polynomial factors  it must be into two irreducible quadratic polynomials.
   \item  If for a given marking of the tree of type $f$ the polynomial factors  it must be into a  linear and an irreducible cubic polynomials.
\item Types $g,j,h$ are always irreducible.
\item For type   $ i$ we have no useful information.

\end{itemize}

\end{lemma} 
   Recall the classification of \S \ref{imarki}.
      So we are left with  Linear types of  type  A in case f,g.   Type BI     for trees of linear and  star type.  
    \bigskip

{\bf Linear types}           
    Type  A.  In the linear graphs $3$ must be in the middle so if the polynomial factors it must be into 2 irreducible quadratics.   We have already seen that we only need to treat cases f,g  and for these the previous factorization  is incompatible with the restrictions given by Lemma \ref{cas1}.
     
     We still need to treat Type  B I which can appear only in the two cases  $f,i$. Some orientations give rise to a graph in the special component so we may ignore them.  We are left with (up to exchanging $1,2$)  $$\boxed{f1}\qquad \xymatrix{ a\ar@{-}[r]^{3,1}&b^+\ar@{-}[r]^{2,3}& c^+ \ar@{=}[r]^{1,2} &d^- }, 
\boxed{f2}\xymatrix{ a\ar@{-}[r]^{3,1}&b^+\ar@{-}[r]^{3,2}& c^+ \ar@{=}[r]^{1,2} &d^- }, 
$$$$\boxed{f3}\qquad \xymatrix{ a\ar@{-}[r]^{1,3}&b^+\ar@{-}[r]^{2,3}& c^+ \ar@{=}[r]^{1,2} &d^- }, 
\boxed{f4}\xymatrix{ a\ar@{-}[r]^{1,3}&b^+\ar@{-}[r]^{3,2}& c^+ \ar@{=}[r]^{1,2} &d ^-}, 
$$\smallskip

Now,  $f1,f4$ do not arise  in the applications, since they are not complete (although they also give rise to irreducible polynomials). We are left with $f2,f4$:
$$f2= \begin{vmatrix}
t & -2\sqrt{\xi_1\xi_3}& 0&0\\&&&\\ -2\sqrt{\xi_1\xi_3}&t-\xi_1+ \xi_3  
 & - 2\sqrt{\xi_2\xi_3}&0\\&&&\\  0& -2\sqrt{\xi_2\xi_3}   &t     -\xi_1-\xi_2+2\xi_3&2\sqrt{\xi_2\xi_1}\\&&&\\ 0&0&-2\sqrt{\xi_2\xi_1}&t+2\xi_3  \end{vmatrix} $$$$  f3=\begin{vmatrix}
t & -2\sqrt{\xi_1\xi_3}& 0&0\\&&&\\ -2\sqrt{\xi_1\xi_3}&t-\xi_3+ \xi_1  
 & - 2\sqrt{\xi_2\xi_3}&0\\&&&\\  0& -2\sqrt{\xi_2\xi_3}   &t     -2\xi_3+\xi_1+\xi_2&2\sqrt{\xi_2\xi_1}\\&&&\\ 0&0&-2\sqrt{\xi_2\xi_1}&t+2\xi_3  \end{vmatrix}  $$\smallskip

Markings
$$  0, \xi_1-\xi_3, \xi_1+\xi_2-2\xi_3, 2\xi_3;\quad
 0, \xi_3-\xi_1, 2\xi_3-\xi_1-\xi_2, 2\xi_3 . $$
In this table we show the possible linear factors setting one of the variables to 0.
 $$\begin{matrix} &&f2&&f3& \\&&&\\
\xi_1=0: &&0,2\xi_3 &&0,2\xi_3&  \\&&&\\
\xi_2=0: &&-\xi_1+2\xi_3, 2\xi_3&& \xi_1-2\xi_3, 2\xi_3& \\&&&\\
\xi_3=0: &&0,  - \xi_1&&0,    \xi_1& 
\end{matrix}$$\medskip
 
  So a possible linear factor could be    $t-\xi_1+2\xi_3 ,t+2\xi_3$  in the first case, $t+2\xi_3,t+\xi_1-2\xi_3$ in the second. The odd cases can be excluded by parity  so we may have   $t+2\xi_3$ in both cases.    
 At $t=-2\xi_3$ the two determinants are  
 
 $$-8\, \xi_1\, \xi_2\, \left( \xi_1 - \xi_3 \right) \, \xi_3 ,\quad -24\, \xi_1\, \xi_2\, \left( \xi_1 - \xi_3 \right) \, \xi_3.$$
 
 Finally Type $i$, observe that 
$$  \xymatrix{ a\ar@{-}[r]^{3,1}&b\ar@{=}[r]^{1,2}& c \ar@{-}[r]^{2,3} &d }, \qquad  \xymatrix{ a\ar@{-}[r]^{1,3}&b^+ \ar@{=}[r]^{1,2}& c^- \ar@{-}[r]^{2,3} &d ^-} $$ $$  \xymatrix{ a\ar@{-}[r]^{3,1}&b\ar@{=}[r]^{1,2}& c \ar@{-}[r]^{3,2} &d },$$
do  not arise since they are not complete.  It remains up to symmetry  
 $$ \qquad  \xymatrix{ a\ar@{-}[r]^{1,3}&b^+ \ar@{=}[r]^{1,2}& c^- \ar@{-}[r]^{3,2} &d ^-}, $$ 
  
 $$\begin{vmatrix}
t & -2\sqrt{\xi_1\xi_3}& 0&0\\&&&\\ 2\sqrt{\xi_1\xi_3}&t-\xi_3+ \xi_1  
 &  2\sqrt{\xi_2\xi_1}&0\\&&&\\  0& -2\sqrt{\xi_2\xi_1}   &t     +2\xi_1+\xi_2-\xi_3& 2\sqrt{\xi_2\xi_3}\\&&&\\ 0&0& 2\sqrt{\xi_2\xi_3}&t+2\xi_1+2\xi_2-2\xi_3   \end{vmatrix} $$
 This  case has passed   a standard factorization algorithm which shows that it is irreducible.\smallskip

    {\bf Star types}
 
  {\bf Type B I} For  only one orientation the tree is complete. Take $b$ as root $$  \xymatrix{&c^+& &\\a^-\ar@{=}[r]_{1,2}&b\ar@{-}[u]^{3,2}  \ar@{-}[r]^{3,1} &d^+ }  
markings
  \xymatrix{& \xi_3-\xi_2& &\\  \xi_2+\xi_1\ar@{-}[r] &\ar@{-}[u]  \ar@{-}[r] & \xi_3  -\xi_1 }$$
 $$  \begin{vmatrix}
t+\xi_2+\xi_1 & -2\sqrt{\xi_1\xi_2}& 0&0\\&&&\\ 2\sqrt{\xi_1\xi_2}&t 
 &  - 2\sqrt{\xi_2\xi_3}&-2\sqrt{\xi_1\xi_3}\\&&&\\ 0 &- 2\sqrt{\xi_2\xi_3}   &t    -\xi_2 +\xi_3&0\\&&&\\ 0&-2\sqrt{\xi_1\xi_3}&0&t+\xi_3-\xi_1  \end{vmatrix} $$
  passes standard tests of irreducibility.

In fact in this case even the    determinant is   irreducible (over $\mathbb Q$):
$$ 4\,\left( \xi_1^2\,\xi_2^2 + \xi_1^2\,\xi_2\,\xi_3 + \xi_1\,\xi_2^2\,\xi_3 - \xi_1^2\,\xi_3^2 - \xi_1\,\xi_2\,\xi_3^2 - 
    \xi_2^2\,\xi_3^2 \right), $$   
   \bigskip

 {\bf Circuits}
 \medskip

 We first start  from the Formulas \eqref{icircui} of a basic circuit $$\boxed{C}\xymatrix{&c& &\\ a\ar@{<-}[ru] ^{3,1}\ar@{->}[rr]_{1,2} &&b\ar@{->}[lu]_{2,3}  &  }\boxed{D}  \xymatrix{&&c& &\\ ,&a\ar@{<-}[ru] ^{3,1}\ar@{=}[rr]_{1,2} &&b\ar@{=}[lu]_{2,3}  &  } 
 $$    and add a linearly independent edge in all possible ways and then close it  as necessary.  If this edge has two indices disjoint from $1,2,3$ we are under the hypotheses of Theorem \ref{supertest}.
Otherwise, up to symmetry   we may have several cases, but we can always assume that 1 appears in the marking  of  the edge. It can be $(1,4),$  black or red, in this case, we complete if necessary the graph but then  we may apply Theorem \ref{supertest}. The other possibility  is that it does not involve   an index different from $1,2,3$  but then it must have a color different from the one appearing in the graph in order to be linearly independent.
\smallskip

{\bf Type C}  We may assume we add the edge $\xymatrix{  \ar@{=}[r] ^{1,2}  &   } $, according to Lemma \ref{adia} the only possible position is at $c$. 
 
 $$\boxed{C1} \xymatrix{&c\ar@{=}[r] ^{1,2}  & d&\\ a\ar@{<-}[ru] ^{3,1}\ar@{->}[rr]_{1,2} &&b\ar@{->}[lu]_{2,3}  &  } 
 $$  
 {\bf Type D}  We may assume we add the edge $\xymatrix{  \ar@{=}[r] ^{1,3}  &   } $ or  $\xymatrix{  \ar@{-}[r] ^{1,2}  &   } $, according to Lemma \ref{adia} the only possible positions are $$   \xymatrix{ &c& &\\  &d& &\\   a\ar@{<-}[ruu] ^{3,1}\ar@{=}[rr]_{1,2} &&b\ar@{=}[luu]_{2,3} \ar@{=}[lu] ^{1,3} &   }  \xymatrix{&&c\ar@{-}[d] _{2,1}& &\\ & & d &  &\\ ,&a\ar@{<-}[ruu] ^{3,1}\ar@{=}[rr]_{1,2} &&b\ar@{=}[luu]_{2,3}  &  } \boxed{D1}  \xymatrix{&&c\ar@{-}[d] _{1,2}& &\\ & & d &  &\\ ,&a\ar@{<-}[ruu] ^{3,1}\ar@{=}[rr]_{1,2} &&b\ar@{=}[luu]_{2,3}  &  } 
 $$  
 the first 2 should be completed to the same
 $$  \boxed{D2}  \xymatrix{ &c\ar@{<- }[d] _{2,1}& &\\  &d& &\\   a\ar@{ ->}[ru] _{3,2}\ar@{<-}[ruu] ^{3,1}\ar@{=}[rr]_{1,2} &&b\ar@{=}[luu]_{2,3} \ar@{=}[lu] ^{1,3} &   }    
 $$  
 So our final computation is with these 3 cases.  The matrices are:

$$ \boxed{C1}  \quad \begin{vmatrix}
t &- 2\sqrt{\xi_1\xi_2}& -2\sqrt{\xi_1\xi_3}&0\\&&&\\-2\sqrt{\xi_1\xi_2}&t+\xi_1- \xi_2  
 &  -2\sqrt{\xi_2\xi_3}&0\\&&&\\  -2\sqrt{\xi_1\xi_3}& -2\sqrt{\xi_2\xi_3}   &t     +\xi_1-\xi_3&2\sqrt{\xi_1\xi_2}\\&&&\\ 0&0&-2\sqrt{\xi_1\xi_2}&t+2\xi_1+\xi_2-\xi_3  \end{vmatrix} $$
passes a standard test of irreducibility.  The determinant is $\xi_1$ times an irreducible cubic.
   
   Take $a$ as root.
$$\boxed{D1}  \xymatrix{&&c^+\ar@{-}[d] _{1,2}& &\\ & & d^+ &  &\\ ,&a\ar@{<-}[ruu] ^{3,1}\ar@{=}[rr]_{1,2} &&b^-\ar@{=}[luu]_{2,3}  &  } 
   \boxed{D2}  \xymatrix{ &c^+\ar@{<- }[d] _{2,1}& &\\  &d^+& &\\   a\ar@{ ->}[ru] _{3,2}\ar@{<-}[ruu] ^{3,1}\ar@{=}[rr]_{1,2} &&b^-\ar@{=}[luu]_{2,3} \ar@{=}[lu] ^{1,3} &   }    
 $$  
      
$$  \begin{vmatrix}
t & 2\sqrt{\xi_1\xi_2}& -2\sqrt{\xi_1\xi_3}&0\\&&&\\-2\sqrt{\xi_1\xi_2}&t+\xi_1+ \xi_2  
 &  -2\sqrt{\xi_2\xi_3}&0\\&&&\\  -2\sqrt{\xi_1\xi_3}& 2\sqrt{\xi_2\xi_3}   &t     +\xi_1-\xi_3&-2\sqrt{\xi_1\xi_2}\\&&&\\ 0&0&-2\sqrt{\xi_1\xi_2}&t+2\xi_1-\xi_3+\xi_2  \end{vmatrix} $$ 
  \bigskip
  
$$  \begin{vmatrix}
t & 2\sqrt{\xi_1\xi_2}& -2\sqrt{\xi_1\xi_3}&-2\sqrt{\xi_2\xi_3}\\&&&\\-2\sqrt{\xi_1\xi_2}&t+\xi_1+ \xi_2  
 &  -2\sqrt{\xi_2\xi_3}&-2\sqrt{\xi_1\xi_3}\\&&&\\  -2\sqrt{\xi_1\xi_3}& 2\sqrt{\xi_2\xi_3}   &t     +\xi_1-\xi_3&-2\sqrt{\xi_1\xi_2}\\&&&\\ -2\sqrt{\xi_2\xi_3}&2\sqrt{\xi_1\xi_3}&-2\sqrt{\xi_1\xi_2}&t+\xi_3-\xi_2  \end{vmatrix} $$ 
  In the first  the characteristic polynomial is irreducible while the determinant is $4\xi_1$ times an   
irreducible cubic. In the second they are both irreducible.

We finally have circuits with 4 edges but not 3. In this case the circuit has 4 markings linearly dependent of rank 3, which satisfy a very special linear relation.  One easily sees that they must have as encoding graph a standard circuit with an even number of red edges and suitably oriented black edges:
$$ \xymatrix{4\ar@{- }[d] \ar@{ - }[r] &3 \\1\ar@{- }[r] &2\ar@{- }[u]    } \quad  \xymatrix{4\ar@{= }[d] \ar@{ = }[r] &3 \\1\ar@{- }[r] &2\ar@{-  }[u]    } \quad   \xymatrix{4\ar@{= }[d] \ar@{ - }[r] &3 \\1\ar@{- }[r] &2\ar@{=  }[u]    }  \quad  \xymatrix{4\ar@{= }[d] \ar@{ = }[r] &3 \\1\ar@{= }[r] &2\ar@{=  }[u]    }   $$

One easily sees that they  fall under theorem \ref{supertest}.  E.g.:
  $$ \xymatrix{d\ar@{->}[d]_{4,1}\ar@{<- }[r]^{3,2 }&c \\a\ar@{->}[r]^{1,2}&b\ar@{->}[u]_{3,4}   } \xymatrix{d\ar@{->}[d]_{3,4}\ar@{<- }[r]^{3,2 }&c \\a\ar@{->}[r]^{1,2}&b\ar@{->}[u]_{4,1}   }  $$
  
  The other cases are similar or not complete.\section{Separation\label{SEP}}

  \subsection{Translations}
  We want to prove here the separation Lemma \ref{SEPA} in dimension 3.

  We  have constructed the polynomials  which appear in the  formulas for $ad(N)$  as follows.
  
  We  take a set $A=\{(a_0,\sigma_0),(a_1,\sigma_1),\ldots, (a_k,\sigma_k)\}\subset \Z^m\times\Z/(2)$  so that the vectors $a_i$ are independent in the sense of affine geometry (they span an affine space of dimension $k$).  
  
  We only consider those such that the complete graph generated by them is connected and  we call $k$ the {\em rank} of $A$. Let us call these sets {\em allowable}.
  
  To such a set we have associated (Definition \ref{allowa}) a $k+1\times k+1$ matrix  $  C_A$ and its characteristic polynomial $\chi_A$, a  polynomial of degree $k+1$ in $t$ and the variables $\xi_i$  which is monic in  $t$.   
  We Conjecture that
 
 {\bf Conjecture 2}\ 
If $A\neq B$ have at least 2 elements then $\chi_A\neq \chi_B.$

For one element $(a,+)$  and $(-a,-)$ give  $t-L_a(\xi).$
\bigskip
 
    We start from a standard edge $[(0,+),(e_1-e_2,+)];$  $ [(0,+),(-e_1-e_2,-)].$

Under translations and sign change we have 4 possibilities:
$$[(a,+),(a+e_1-e_2,+)];\quad [(b,+),(-b-e_1-e_2,-)].$$$$ [(c,-),(c+e_1-e_2,-)];\quad [(d,-),(-d-e_1-e_2,+)].$$

The matrices  are $$\begin{vmatrix}
  -a(\xi)  & 2\sqrt{\xi_1\xi_2}\\&&&\\ 2\sqrt{\xi_1\xi_2}&  -a(\xi) -\xi_1+ \xi_2  
\end{vmatrix},\quad \begin{vmatrix}
 - b(\xi)   &  -2\sqrt{\xi_1\xi_2}\\&&&\\ 2\sqrt{\xi_1\xi_2}&   -b(\xi)  -\xi_1-\xi_2  
\end{vmatrix} $$  $$\begin{vmatrix}
   c(\xi)  & -2\sqrt{\xi_1\xi_2}\\&&&\\ -2\sqrt{\xi_1\xi_2}&   c(\xi) +\xi_1- \xi_2  
\end{vmatrix},\quad \begin{vmatrix}
   d(\xi)   &   2\sqrt{\xi_1\xi_2}\\&&&\\ -2\sqrt{\xi_1\xi_2}&   d(\xi) +\xi_1+\xi_2  
\end{vmatrix} $$   \begin{lemma}\label{recod}
These blocks can be reconstructed from the characteristic polynomial.
\end{lemma}\begin{proof} First, passing modulo 2 we identify from the trace $\tau$ the two variables $\xi_1,\xi_2$.
We   can deduce the elements $a,b,c,d$ in the 4 cases from $\tau$.
$$ -\tau=2(a-\xi_2)+ \xi_1+ \xi_2=2b+ \xi_1+ \xi_2=2c+\xi_2-\xi_1=2d-\xi_1-\xi_2 .$$  So let us write  $\tau=-2a+\xi_2-\xi_1$  and the possible 4 matrices with trace $\tau$ are 
$$\begin{vmatrix}
  -a  & 2\sqrt{\xi_1\xi_2}\\&&&\\ 2\sqrt{\xi_1\xi_2}& - a -\xi_1+ \xi_2  
\end{vmatrix},\quad\text{or}\quad  \begin{vmatrix}
  -a+\xi_2  & -2\sqrt{\xi_1\xi_2}\\&&&\\  2\sqrt{\xi_1\xi_2}&  -a-\xi_1   
\end{vmatrix} $$   
 $$\begin{vmatrix}
   a  +\xi_1-\xi_2 & -2\sqrt{\xi_1\xi_2}\\&&&\\ -2\sqrt{\xi_1\xi_2}&   a +2\xi_1- 2\xi_2  
\end{vmatrix},\quad \begin{vmatrix}
   a+\xi_1   &   2\sqrt{\xi_1\xi_2}\\&&&\\ -2\sqrt{\xi_1\xi_2}&   a +2\xi_1+\xi_2  
\end{vmatrix} $$
we compute the 4 determinants$$a(a +\xi_1- \xi_2)- 4\xi_1\xi_2,\quad  (a- \xi_2)(a +\xi_1)+ 4\xi_1\xi_2,$$
 $$ (a +\xi_1- \xi_2) (a +2\xi_1- 2\xi_2)- 4\xi_1\xi_2,\quad (a +\xi_1)(a +2\xi_1+\xi_2)+ 4\xi_1\xi_2.$$ We  leave to the reader to verify that,  when they are  equal,   the  edge is the same.\footnote{of course the matrix depends on ordering the vertices of the edge, in particular one vertex appears as root}

 \end{proof}  
 
 \medskip
 
  \begin{proof} [Proof In dimension 2,3]  
  We   have already proved in Lemma \ref{recod}, that this statement is true for blocks with two elements.
In   dimension 3 we know that  all polynomials $\chi_C$  with $C$ of rank $<4$ are irreducible.   

Consider thus  the $m$ linear mappings  $\lambda_i:\Z^m\to\Z^{m-1}$  each dropping the  $i^{th}$ coordinate or in more intuitive terms, setting $e_i=0$.

We extend this map to  $ \lambda_i:\Z^m\times\Z/(2)\to\Z^{m-1}\times \Z/(2)$  and notice that, if we remove from the graph  $\Z^m\times\Z/(2)$ all edges in which $e_i$ appears, this is a map of  graphs.  Moreover  take   $A$  allowable of  rank $k$ and     remove  all the edges  in which $e_i$ appears, we obtain in general several connected components   $A^i_1,\ldots,A_s^i$ of the restricted graph. It is easily seen that every connected component $B$ of the induced graph maps injectively under $\lambda_i$  to   $ \Z^{m-1}\times\Z/(2)$, since the index $i$ does not appear in the markings of the  edges of $B$. Theorem \ref{lafatt}  tells us  that
\begin{equation}
\label{lafa1} \chi_A(t,\xi)_{\xi_i=0}=\prod_{j=1}^s\chi_{A_j^i}(t,\xi).  
\end{equation}
This is again the basis of an induction. We can in fact by induction reconstruct the entire graph once we forget the $i^{th}$ coordinate, except for the isolated vertices for which we have no information on the sign.  By convention we describe their known coordinates as if the sign were $+$.

In particular by setting all variables equal to 0  except  two of them   we see that if $\chi_A=\chi_B $   the two graphs must have the same colored encoding graphs. In fact we recover not only the edges but also the coordinates of the vertices relative to the two indices of the edge and their sign.

Now in dimension 3 we treat  first the case of a block with 3 vertices. Thus we know the encoding graph.  This has either two or three edges. Let us treat the  case of 3 edges (the other case is simpler). In this case the  graph  is as in  \eqref{icircui}, we need to compute the coordinates.  In this case set for instance $e_1=0$ we see  $$\xymatrix{&c& &\\ a  &&b\ar@{->}[lu]_{2,3}  &  }, \xymatrix{&c& &\\ a  &&b\ar@{=}[lu]_{2,3}  &  }.
$$
In the first case we know the sign of $a$ and hence the second and third coordinate of all the vertices. This is enough to reconstruct all edges since $b,c$ have different third coordinate.

In the second case, assume for instance $c$ has sign $+$. We see immediately to which of $b,c$ the point $a$ is connected by the black edge $(1,3)$  since we must have one and only one  $a_2=c_2$ or $a_2=-b_2=c_2+1$.

Next the most degenerate cases with 4 vertices.

This, up to symmetry is when the encoding graph is of type  $A,B$ (cf.  \ref{tipoA}).

In Type $A$ when we put $e_3=0$  we  must see the graph  
$$  \xymatrix{ a   \ar@{=} [r]^{1,2}&b&c \ar@{-}[r]^{1,2}    &d    } $$ together with the first two coordinates and the sign of all the vertices.

Set $a=(a_1,a_2,x), c=(c_1,c_2,y)$ so that \begin{equation}\label{labb}  b=(-1-a_1,-1 -a_2,-x),\quad d=(c_1+1,c_2-1,y) .
\end{equation}

When we set   $e_2=0$ we have two possible cases:
$$ 1)\qquad \xymatrix{    u\ar@{=} [r]^{1,3} &v&z&w     }$$  or $$ 2)\qquad  \xymatrix{  u\ar@{-} [r]^{1,3} &v&z&w       }.$$   Possibly  we may have $z=w$. Let us treat the first case.

We need to determine which vertices match coordinates and are thus joined by this  edge.    
If say $a=(a_1,a_2,x)$  is joined with $c=(c_1,c_2,y)$  then we must have that $a+c=(-1,0,-1)$ so
$$a_1+c_1=-1,a_2+c_2=0,x+y=-1. $$   Similarly for the other 3 possible ways in which the edge may appear.  We need to show that only one possible choice is available, so we may assume that the choice $a,c$ is available and then prove that the others are not possible. Let us first exclude the possibility $b,c$.  By \eqref{labb} $b_2+c_2=-1 -a_2+c_2$  by assumption. If  $a,c$ is available we have also  $a_2+c_2=0$, if also $b,c$ is available then $0=b_2+c_2=-1 -a_2+c_2$. This implies   $c_2=1/2$ which is not possible.

Now the possibility  $b,d$ or $a,d$.  We have that  $d=(c_1+1,c_2-1,y)$  so for $a,d$ again we should have $a_2+d_2=0$ that is  $a_2+c_2-1=0$ incompatible. For $b,d$  we would have 
$$b_1+d_1=-1,b_2+d_2=0,-x+y=-1. $$

We have $y=-1,x=0$  and $b_1+d_1=a_1+c_1=-1$. Next $b_1+d_1=-1-a_1+c_1+1$  which implies
$a_1=0,c_1=-1$
 $-a_1+c_1=1$ hence $c_1=-1,a_1=0$  then $b_2+d_2=a_2+c_2=0$ implies 
$-1 -a_2+c_2-1=0$ that is $a_2=-1,c_2=1$. Finally we must have 
$$  \xymatrix{(0,-1,0)    \ar@{=} [r]^{1,2} \ar@{=} [d]^{1,3}&(-1,0,0)\ar@{=} [d]^{1,3}&\\(-1,1,-1 )\ar@{-}[r]^{1,2}    & (0,0,-1 )   }.$$

This is possible only if we ignore the signs. With signs this is incompatible  (the odd circuits   do  not lift from $\tilde \Lambda$ to $\Lambda$ and hence we have excluded  them   with the resonance hypothesis  3) \ref{therin}).

Now the second case.

When we put $e_3=0$  we  must see the graph  
$$  \xymatrix{ a   \ar@{=} [r]^{1,2}&b&c \ar@{-}[r]^{1,2}    &d    }.$$ When we set   $e_2=0$ we have a graph
  $$  \xymatrix{  u\ar@{-} [r]^{1,3} &v&z&w       }$$  
We need to determine which vertices are joined by this  edge.  
If say $a=(a_1,a_2,x)$  is joined with $c=(c_1,c_2,y)$, we have still two possibilities for the orientation.  In   have have either  $a-c=(-1,0,1 ),\ a-c=( 1,0,-1 )$ so
$$i)\quad a_1-c_1= -1,a_2=c_2, x-y= 1\quad\text{or}\quad ii)\quad  c_1-a_1=-1,a_2=c_2, y-x= 1. $$ Either $u$ or $v$ correspond to $a$, that is  either $u_1=a_1$ or $v_1=a_1$  moreover we need to have that $b,d$ correspond to $z,w$ in some order.  Similarly for the other 3 possible ways in which the edge may appear.  We need to show that only one possible choice is available, so we may assume that the choice $a,c$ is available and then prove that the others are not possible. Let us first exclude the possibility $b,c$. This implies, $b_2=c_2=a_2$ since   $b=(-1-a_1,-1 -a_2,-x)$   we have $a_2=-1/2$ impossible.  

 We have $d_1=c_1+1,d_2=c_2-1,d_3=y$. Assume we have the edge $b,d$ marked $(1,3)$.
 $$j)\quad b_1-d_1= 1,b_2=d_2, -x-y= -1\quad\text{or}\quad jj)\quad  d_1-b_1= 1,d_2=b_2, y+x= -1. $$
 
 Assume case $i,j$.
 Then $b_2=d_2= c_2-1$ so $-1 -a_2= a_2-1$ implies $a_2=0$. 
We get
$$a=(-2,0, 1),\ b=( 1,-1,-1),\ c=(-1,0,0),\ d=(0,-1,0).$$

$$  \xymatrix{ (0,-1,0) \ar@{-}[r]^{2,1 }    &(-1,0,0 )& \\(1,-1,-1)  \ar@{-}[u]^{3,1 }   \ar@{=} [r]^{1,2}& (-2,0, 1) \ar@{-}[u]^{1,3 }&   }.$$

This is possible only if we ignore the signs. The other cases are symmetric due to the symmetry in the red edge $(1,2)$.\bigskip

A priori there is a possible further degeneracy, this occurs when the two edges $$  \xymatrix{ a   \ar@{=} [r]^{1,2}&b&c \ar@{-}[r]^{1,2}    &d    } $$ coincide in the projection. This means that $$a+b=-e_1-e_2,\ a-b=\pm(e_2-e_1)\implies a=d=-e_1, 
 b=c=-e_2\  \text{or conversely}.$$   Then   the original points can be 
 $$a=(-1,0,x),b=(0,-1,-x);\quad c=(0,-1,y), d=(-1,0,y).$$ But we cannot  have two of them joined by $(1,3)$ so this case does not occur.\bigskip

In Type $B I$ \eqref{typeB}, when we put $e_3=0$  we  must see the graph  
$$  \xymatrix{ a   \ar@{=} [r]^{1,2}&b&c     &d    }.$$

We have then  $a=(a_1,a_2,x),b= (-1-a_1,-1 -a_2,-x)$ and we know the sign of both. When we set   $e_2=0$ we have a graph
$$  \xymatrix{    u\ar@{-} [r]^{1,3} &v&z&w     }$$   When we set   $e_1=0$  $$  \xymatrix{  p\ar@{-} [r]^{2,3} &q&t&s       }$$  

We need to determine which vertices are joined by these  edges.  What we know are the first two coordinates and sign of $a,b  $, the last two coordinates and sign of   the vertices $p,q$ the first and last of  $u,v $.

If the signs of  $u,v$  and $p,q$  are different  then we see immediately that the graph  is uniquely reconstructed. Let us give some detail, for instance assume that $u,v$ have the same sign of $a$ and $p,q$ the sign of $b$.   Then necessarily either $u=a$ or $v=a$ and which one is the case we see by the first coordinate. Similarly  for $p,q$.

If  $u,v,p,q$ have the same sign   then the two edges $(u,v)$ and $(p,q)$ form a segment. This a priori can occur in 4 ways, but two of them are not complete. The only possibilities are
$$ \xymatrix{  u\ar@{-}[r]^{1,3 }   &v=p\ar@{-}[r]^{2,3 }  &q  &  &p\ar@{-}[r]^{2,3 }   &q=u\ar@{-}[r]^{1,3 }  &v}$$
Which one occurs is determined inspecting the coordinates  (the two end points have different third coordinate).

Next we need to identify in this segment the point $a$ or $b$ depending from the sign,  and this is clear again by the coordinates.
    
 Next cases C1, D1, D2 are distinguished by their encoding graphs.  In case C1 we set $e_3=0$ and see the two edges $(1,2)$.  How to complete the graph is clear by inspecting the second coordinate  since we know all the signs.
 
 Cases D1, D2  are treated in a very similar way.
  
 The less degenerate cases are   easier since we have more coordinates available that change  and follow the same line of reasoning.
\end{proof}

\section{Real roots\label{posit}}
In this section we want to touch on  the issue of when the eigenvalues of the  combinatorial matrices are all real.  We  know that, for black edges and positive $\xi$  the quadratic form is positive definite so the corresponding matrix has real eigenvalues.    When we have red edges  we do not have always real eigenvalues and we need to isolate the regions in the parameters where this occurs.

First of all in dimension 2  we have seen that for a  $2$ block corresponding to a red edge marked  $(i,j)$ we have the inequality $(\xi_i+\xi_j)^2>16\xi_i\xi_j$  determining the region where the eigenvalues are real.

If one follows  the choice  of \cite{GYX}  no further blocks appear and these inequalities suffice and determine a non empty open sector in the parameters $\xi$.

Otherwise we have to analyze the 3 dimensional blocks.

Similarly in dimension 3.  The polynomial inequalities that one obtains are explicit but rather formidable, so we can discuss  just the qualitative aspects.
One can remark  that when we set one variable  equal to 0  we are in position to apply  Theorem \ref{lafatt}.  When we have linear or quadratic terms  we  deduce that either all roots are real for all positive values of the remaining parameters if the quadratic terms come from black edges, or in case of red edges  we have the simple explicit inequality $(\xi_i+\xi_j)^2>16\xi_i\xi_j$.  Moreover if all the roots are different as it happens  in most cases  we have that in an open neighborhood of this set the  roots are still real and distinct.  So we have to make sure that all these neighborhoods have a non--empty intersection.

\bibliographystyle{plain}

\bibliography{bibliografia}
\end{document}